\documentclass[11pt,a4paper]{amsart}
\usepackage[margin=1in]{geometry}
\usepackage[T1]{fontenc}
\usepackage{lmodern}
\usepackage{amssymb,mathtools,booktabs,array,graphicx,microtype}
\usepackage{cases}
\usepackage[numbers,sort&compress]{natbib}
\let\cite\citep
\usepackage{hyperref}
\hypersetup{hidelinks,
  pdftitle={Intermediate Coverings, Random Images, and Projections},
  pdfauthor={Yunhao Zhao; Yingyu Zhao; Chen Zhou; Shengyang Zang},
  pdfsubject={Intermediate dimensions, scale-dependent families, and projections}}

\makeatletter
\let\originalsetaddresses\@setaddresses
\renewcommand{\@setaddresses}{
  \par\begingroup
  \leftskip=\normalparindent
  \parindent=0pt
  \originalsetaddresses
  \endgroup
}
\makeatother

\newtheorem{theorem}{Theorem}[section]
\newtheorem{proposition}[theorem]{Proposition}
\newtheorem{lemma}[theorem]{Lemma}
\newtheorem{corollary}[theorem]{Corollary}
\newtheorem{maintheorem}{Theorem}

\theoremstyle{definition}

\theoremstyle{remark}
\newtheorem{remark}[theorem]{Remark}
\newcommand{\R}{\mathbb R}
\newcommand{\E}{\mathbb E}
\newcommand{\Prob}{\mathbb P}
\newcommand{\dd}{\,\mathrm d}
\newcommand{\eps}{\varepsilon}
\newcommand{\ldim}{\underline{\dim}}
\newcommand{\udim}{\overline{\dim}}

\DeclareMathOperator{\KL}{KL}
\title[Intermediate coverings, random images, and projections]
{Intermediate Coverings, Random Images, and Projections}
\author{Yunhao Zhao}
\address{Yunhao Zhao, School of Advanced Interdisciplinary Sciences,
University of Chinese Academy of Sciences, Beijing, China}
\email{zhaoyunhao23@mails.ucas.ac.cn}
\author{Yingyu Zhao}
\address{Yingyu Zhao, Qiuzhen College, Tsinghua University, Beijing, China}
\email{zhao-yy25@mails.tsinghua.edu.cn}
\author{Chen Zhou}
\address{Chen Zhou, School of Mathematical Sciences,
Zhejiang University of Technology, Hangzhou, China}
\email{zcygo01@163.com}
\author{Shengyang Zang}
\address{Shengyang Zang, School of Mathematical Sciences,
Nankai University, Tianjin, China}
\email{1120230032@mail.nankai.edu.cn}
\date{}

\begin{document}
\raggedbottom
\begin{abstract}
We study intermediate coverings of spectral families, Cantor sets, and
projections of measures. Two perturbation families of a fixed ReLU tangent
kernel have identical source thresholds. Their Brownian images have equal
box thresholds and distinct intermediate thresholds, almost surely along
prescribed geometric scales. For Cantor sets with delayed contractions,
we obtain entropy formulas for the intermediate dimensions of the sets
and their fractional Brownian images. These formulas yield endpoint
asymptotics and a sharp separation theorem for the first phase transition.
An oscillating variant has distinct lower and upper intermediate dimensions
while preserving the Assouad spectrum. Finally, almost every
orthogonal projection preserves both intermediate dimensions of a nonzero
finite compactly supported measure whose quasi-Assouad dimension does
not exceed the target dimension.
\end{abstract}
\maketitle

\section{Introduction}\label{sec:introduction}

We study intermediate coverings for spectral families, Cantor sets with
delayed contraction histories, and orthogonal projections of finite
measures.

\subsection{Covering costs and notation}

For a bounded nonempty set $F\subset\R$, $0<r<1$,
$0<\theta\le1$, and $s>0$, set
\begin{equation*}
 S^s_{r,\theta}(F)
 =\inf\left\{\sum_i|U_i|^s:
 F\subset\bigcup_iU_i,\quad r\le|U_i|\le r^\theta\right\},
\end{equation*}
where the covers consist of intervals and $|U_i|$ denotes diameter.
The lower and upper intermediate dimensions are, respectively,
\[
 \ldim_\theta F=\inf\{s>0:\liminf_{r\downarrow0}S^s_{r,\theta}(F)=0\},
 \qquad
 \udim_\theta F=\inf\{s>0:\lim_{r\downarrow0}S^s_{r,\theta}(F)=0\}.
\]
These are the dimensions introduced by \citet{FFK2020}; see also
\cite{FalconerSurvey2021}. At $\theta=1$ they are the corresponding
box dimensions. We set $\dim_0 F=\dim_H F$.

For a family $G_r$ and a prescribed $a>1$, the \emph{joint covering
threshold} $D(\theta)\in[0,1]$ along $r_n=a^{-n}$ is defined by
\[
 S^s_{r_n,\theta}(G_{r_n})\longrightarrow
 \begin{cases}
 \infty,&0<s<D(\theta),\\
 0,&D(\theta)<s<1.
 \end{cases}
\]
For a fixed set, the covering limits can be checked along geometric
scales. If $r\le R\le ar$, enlarging intervals shorter than $R$ in an
admissible $r$-cover gives an admissible $R$-cover and hence
\begin{equation}\label{eq:fixed-scale-comparison}
 S^s_{R,\theta}(F)\le a^s S^s_{r,\theta}(F).
\end{equation}

The random map is a common real-valued fractional Brownian motion
$B_H$, with $0<H<1$, $B_H(0)=0$, and
\[
 B_H(x)-B_H(y)\stackrel{\mathrm d}=|x-y|^H Z,
 \qquad Z\sim N(0,1).
\]
Intermediate dimensions of measures are defined in
Section~\ref{sec:quasi-measures}.

\subsection{Main results}

Let $T_\delta^{\mathrm{in}}$ be the integral operator of the fixed
ReLU tangent kernel \eqref{nat:eq:kernel} on $S^2$, with input law
\[
 \dd\sigma_\delta(x)=(1+\delta P_2(x_3))\,\dd\sigma(x),
 \qquad P_2(z)=(3z^2-1)/2,
\]
where $\sigma$ is uniform probability measure. Write $\mathcal E_\delta$
for zero together with all distinct positive eigenvalues.

\begin{maintheorem}[Spectral families distinguished by a Brownian image]
\label{main:spectral}
For $\beta=1,2$, the families $\mathcal E_{r^\beta}$ have the same
source joint threshold
\[
 D_{\mathrm{src}}^{(\beta)}(\theta)=\frac{\theta}{3+\theta},
 \qquad 0<\theta\le1,
\]
with both cost limits holding at all real scales.
Under one common Brownian motion $B=B_{1/2}$, for each prescribed
$a>1$, almost
surely simultaneously for all $0<\theta\le1$, their image thresholds
$D_B^{(1)}$ and $D_B^{(2)}$ along $a^{-n}$ satisfy
\[
 D_B^{(1)}(\theta)>D_B^{(2)}(\theta)\quad(0<\theta<3/4),
 \qquad
 D_B^{(1)}(\theta)=D_B^{(2)}(\theta)\quad(3/4\le\theta\le1).
\]
In particular,
\[
 D_B^{(1)}(1/2)=\frac{4-\sqrt{10}}3>\frac14=D_B^{(2)}(1/2),
 \qquad D_B^{(1)}(1)=D_B^{(2)}(1)=\frac25.
\]
Supercritical image costs tend to zero almost surely at all real scales.
For each fixed subcritical pair $(\theta,s)$, divergence also holds in
probability at arbitrary scales.
\end{maintheorem}

Theorem~\ref{thm:general-cluster-transfer} gives the cluster estimate
underlying this example. Its proof splits the energy into diagonal,
same-cluster, and different-cluster terms. Uniform spectral splitting
and a tail estimate yield Theorem~\ref{thm:natural-separation} for the
complete spectra.

The second group of results concerns the Cantor constructions of
\citet{BBFLX2025}. Fix $0<\beta<1$ and $M\ge100$. For a binary word
$\omega$ of length $n$, write $S_k(\omega)=\sum_{j=1}^k\omega_j$, with $S_0=0$.
The set $E_{\beta,M}$ is obtained by placing two child intervals at the
ends of each parent, with lengths
\[
 \ell_\omega=M^{-n}2^{S_{\lfloor\beta n\rfloor}(\omega)-\beta S_n(\omega)}.
\]
Put $\lambda=\log M$, $c=\log2$, $a=\lambda-\beta c$, and $T=2/a$.
For a measurable $f:[0,T]\to[0,1]$, let
\[
 A_f(t)=\int_0^t h(f(v))\,dv,
 \qquad
 L_f(t)=\lambda t+c\left\{\beta\int_0^t f(v)\,dv
                                  -\int_0^{\beta t}f(v)\,dv\right\},
\]
where $h(p)=-p\log p-(1-p)\log(1-p)$ with $h(0)=h(1)=0$.

\begin{maintheorem}[Cantor dimensions and analytic branches]
\label{main:cantor}
For every $0<\theta\le1$,
\[
 \ldim_\theta E_{\beta,M}=\udim_\theta E_{\beta,M}
 =d_\theta
 =\max_{0\le f\le1}\ \min_{\theta\le L_f(t)\le1}
                         \frac{A_f(t)}{L_f(t)}.
\]
For $0<\theta<1$,
\[
 \frac{c}{\lambda}<d_\theta<\dim_B E_{\beta,M},
 \qquad
 d_\theta-\frac c\lambda\sim\bigl(\log(1/\theta)\bigr)^{-2}
 \quad(\theta\downarrow0).
\]
For $\theta_{\mathrm{loc}}\leq\theta<1$, with
$\theta_{\mathrm{loc}}$ defined in \eqref{eq:ifs-entropy-threshold},
the maximizing profile, normalized to terminal word time one, is unique
up to null sets and has four constant frequencies. The curve $d_\theta$
is real analytic, strictly increasing, and strictly concave on
$(\theta_{\mathrm{loc}},1)$. On every $[\varepsilon,1]$,
$\varepsilon>0$, it is analytic on the open intervals of a finite
partition. It has at least one nonanalytic point in $(0,1)$.
\end{maintheorem}

Here $c/\lambda=\dim_H E_{\beta,M}$. Theorem~\ref{thm:tree-dimensions}
and Theorem~\ref{thm:tree-random-images} give the source and fractional
Brownian image formulas. Averaging on delay-orbit partitions reduces
these formulas to finite convex problems. Theorem~\ref{thm:tree-endpoint-rate}
gives the endpoint rate, and Section~\ref{sec:ifs-regularity} treats the
analytic branches.

Theorem~\ref{thm:ifs-first-transition} determines the sharp separation
between the first departure from the analytic branch and the collision
of its frequency intervals, uniformly on compact delay intervals for
sufficiently small contraction ratios. A four-node relative-entropy
certificate gives the upper bound; an equalizing profile gives the
matching lower bound.

For the oscillating construction, Theorem~\ref{thm:oscillating-intermediate}
gives distinct lower and upper intermediate dimensions at every positive
parameter. Its Assouad spectrum agrees with that of the fixed
construction. The spectrum has three analytic branches and is $C^1$, with a jump in
the second derivative at each transition
(Corollary~\ref{cor:ifs-assouad-phases}).

Write $G(n,m)$ for the Grassmannian of $m$-dimensional linear subspaces
of $\R^n$, $P_V$ for orthogonal projection onto $V$, and
$P_V{}_\#\mu$ for the pushforward measure. The quasi-Assouad dimension
$\dim_{qA}\mu$ is defined in Section~\ref{sec:quasi-measures}.

\begin{maintheorem}[Quasi-Assouad criterion for measure projections]
\label{main:measure}
Let $\mu$ be a nonzero finite Borel measure with compact support in
$\R^n$, and let $m\in\{1,\ldots,n-1\}$. If
$\dim_{qA}\mu\le m$, then, for almost every $V\in G(n,m)$ with
respect to the rotation-invariant probability measure, simultaneously
for every $0<\theta\le1$,
\[
 \ldim_\theta(P_V{}_\#\mu)=\ldim_\theta\mu,
 \qquad
 \udim_\theta(P_V{}_\#\mu)=\udim_\theta\mu.
\]
\end{maintheorem}

Theorem~\ref{thm:quasi-measure-projection} proves this assertion,
answering \cite[Remark~5.16]{AMT2026}. The proof combines an enlarged
scale window, a point-potential estimate, and projection averaging.

\subsection{Relation to previous work}
\label{sec:related-literature}\label{sec:technical-overview}

Dimension profiles for packing dimensions and fractional Brownian images
were developed in \cite{FH1997,Xiao1997}. Intermediate profiles and their
applications to projections and random images appear in
\cite{BFF2021,Burrell2022,AngeliniMolter2025,Feng2024}.
For $E_\rho=\{0\}\cup\{m^{-\rho}:m\geq1\}$, $\rho>0$,
\citet{Falconer2022} proved
\[
 \dim_\theta E_\rho=\frac{\theta}{\rho+\theta},\qquad
 \dim_\theta B_H(E_\rho)=\frac{\theta}{\rho H+\theta}
 \quad\text{almost surely}.
\]
The cluster theorem extends the covering and energy estimates to
varying supports with separate block and cluster scales.

The spectral application uses arc-cosine kernels \cite{ChoSaul2009},
neural tangent kernels \cite{Jacot2018}, and spherical ReLU spectra
\cite{BM2019,BB2021}. Nonuniform sampling and density comparison are
studied in \cite{Basri2020,LYCL2024}. Harmonic identities
\cite{DaiXu2013,DLMF} and self-adjoint perturbation theory
\cite{Kato1995} give a Schur-complement estimate uniform in the harmonic
degree. The conditions $\delta m\ll1$ and $\delta m^2\ll1$ control
band isolation and separation of the selected eigenvalues, respectively.

The Cantor constructions and their endpoint dimensions are due to
\citet{BBFLX2025}, whose Section~1.5 asks for intermediate dimensions
and Assouad spectra. Related entropy and pressure methods occur for
Bedford--McMullen carpets \cite{BanajiKolossvary2024} and
Gatzouras--Lalley carpets \cite{BFKR2026}; attainable intermediate-dimension
curves are studied in \cite{BanajiRutar2022}. Here cylinder lengths depend
on delayed digit frequencies. Type counting and conditioned Bernoulli
measures yield the variational formulas. Analytic cell decomposition
\cite{vandenDriesMiller1994} gives locally finite analytic branches.
The first-transition estimate combines Bernoulli relative entropy and
Hoeffding's bound \cite{Hoeffding1963} with the physical-scale constraint.
Uniform local cylinder counts yield the Assouad spectrum, following
\cite{FraserYu2018Spectra,FraserYu2018Families}.

Quasi-Assouad projection criteria are known for box and packing dimensions
\cite{FFS2021} and for intermediate dimensions of sets
\cite[Corollary~3.7]{AngeliniMolter2025}.
For measures, \citet{AMT2026} proved invariance under
$\dim_A\mu<m$. Theorem~\ref{main:measure} uses quasi-Assouad mass growth
\cite{HHT2020} on an enlarged testing window to obtain the point-potential
bound. Projection averaging then follows the kernel method of \cite{BFF2021}.

\subsection{Conventions and organization}

All logarithms are natural. We write $f\lesssim g$ for $f\leq Cg$,
$f\sim g$ for $f\lesssim g\lesssim f$.
Subscripts indicate dependence of the implicit constants. Unless stated
otherwise, constants are independent of the covering scale and the band
or word index. The symbol $\beta$ denotes a perturbation exponent in
Part~\ref{part:spectral} and a delay parameter in Part~\ref{part:cantor}.

Parts~\ref{part:spectral}, \ref{part:cantor}, and~\ref{part:measures}
prove the spectral, Cantor, and measure results. The longer spectral
and pressure calculations are collected in the appendices.

\part{Scale-dependent spectral families}
\label{part:spectral}
\section{A cluster covering theorem}\label{sec:vanishing-splitting}

A block of spectral clusters is described by three quantities: its
number of distinct values, its total span, and the span of one cluster.
The theorem below relates these quantities to covering costs and
applies both to splitting proportional to the adjacent spectral gap
and to the wider clusters produced by an input-distribution perturbation.
Separation is required only for a fixed positive proportion of the
points in each cluster.

\subsection{Finite windows and a separated subset}

Fix $b\geq1$, $H\in(0,1)$, and $0<\omega<\infty$.
For $0<r<1$ put $N_r=\lfloor r^{-\omega}\rfloor$ and let
$C_{r,m}$, $1\leq m\leq N_r$, be nonempty finite subsets of one
fixed interval $[0,A_0]$. For integers $1\leq M\leq N_r$, write
\[
 F_r=\{0\}\cup\bigcup_{1\leq m\leq N_r}C_{r,m},\qquad
 \mathcal B_r(M)=\bigcup_{M\leq m<\min\{2M,N_r+1\}}C_{r,m}.
\]
All families are deterministic, and one common real fractional Brownian
path $B_H$ is used throughout. All comparison constants in the following
four assumptions are positive and independent of $r,m,M$.
\begin{enumerate}
\renewcommand{\labelenumi}{(A\arabic{enumi})}
\renewcommand{\theenumi}{A\arabic{enumi}}
\item\label{ass:cluster-count} \emph{Cardinality.}
Each cluster has at most $Cm^{b-1}$ points and contains an ordered
subset $X_{r,m}=\{x_{m,1}<\cdots<x_{m,J_{r,m}}\}$ with
$cm^{b-1}\leq J_{r,m}\leq Cm^{b-1}$.
\item\label{ass:cluster-spans} \emph{Block and cluster spans.}
There are positive scales $\mathcal L_r(M),\mathcal W_r(M)$ such that
\[
 \operatorname{diam}\mathcal B_r(M)\leq C\mathcal L_r(M),
 \quad
 \operatorname{diam}C_{r,m}\leq C\mathcal W_r(M),
 \quad \forall M\leq m<\min\{2M,N_r+1\}.
\]
\item\label{ass:cluster-separation} \emph{Separation.}
For distinct bands $i,k$ in a complete block
$M\leq i,k<2M\leq N_r+1$,
\begin{equation}\label{eq:general-cross-spacing}
 {\rm dist}(X_{r,i},X_{r,k})\geq c\mathcal L_r(M)\frac{|i-k|}{M}.
\end{equation}
and within each selected subset in that block,
\begin{equation}\label{eq:general-inner-spacing}
 |x_{m,j}-x_{m,i}|\geq c\mathcal W_r(M)\frac{|j-i|}{J_{r,m}},
 \quad 1\leq i,j\leq J_{r,m}.
\end{equation}
\item\label{ass:cluster-exponents} \emph{Scale exponents.} There are continuous functions
$\ell,w:[0,\omega]\to[0,\infty)$ such that, uniformly for integers
$1\leq M\leq N_r$, with $u=\log M/\log(1/r)$,
\begin{equation}\label{eq:general-scales}
 \mathcal L_r(M)\sim r^{\ell(u)},\qquad
 \mathcal W_r(M)\sim r^{w(u)}.
\end{equation}

\end{enumerate}
The lower estimates are needed only on complete blocks
$M\leq m<2M\leq N_r+1$. The upper estimates also apply to the
last, possibly incomplete, block. A fixed number of exceptional
initial bands of uniformly bounded cardinality may be omitted from
these assumptions: their images have cost $O(r^s)$.

Write $x_+=\max\{x,0\}$. For $u,v\geq0$ define
\begin{equation}\label{eq:vs-Q}
 Q_\theta(u,v)=
 \begin{cases}
 (u+\theta-v)/\theta,&0\leq v\leq\theta,\\
 u/v,&\theta\leq v\leq1,\\
 u,&v\geq1.
 \end{cases}
\end{equation}
The formulas agree at both endpoints, including when $\theta=1$.

\begin{theorem}[Transfer from cluster geometry]
\label{thm:general-cluster-transfer}
Under assumptions (\ref{ass:cluster-count})--(\ref{ass:cluster-exponents}),
the joint covering threshold of
$B_H(F_r)$ is
\begin{equation}\label{eq:general-threshold}
 D(\theta)=\max_{0\leq u\leq\omega}
 \min\left\{bu,\left(1-\frac{H\ell(u)}{\theta}\right)_+,
                       Q_\theta\bigl(u,Hw(u)\bigr)\right\}.
\end{equation}
More precisely, fix $a>1$ and put $r_n=a^{-n}$. Then, for all $0<\theta\leq1$,
\begin{align}
 S^s_{r_n,\theta}(B_H(F_{r_n}))&\xrightarrow[n\to\infty]{a.s.}\infty,  \quad 0<s<D(\theta), \label{eq:rn-a.s.-sub}\\
 S^s_{r,\theta}(B_H(F_{r}))&\xrightarrow[r\to0]{P}\infty, \quad 0<s<D(\theta), \label{eq:r-P-sub}\\
 S^s_{r,\theta}(B_H(F_r))&\xrightarrow[r\to0]{a.s.}0, \quad D(\theta)<s<1. \label{eq:r-a.s.-super}
\end{align}
The almost-sure limits in \eqref{eq:rn-a.s.-sub} and
\eqref{eq:r-a.s.-super} hold on one event simultaneously for all
$\theta$ and all indicated $s$; this event may depend on $a$.
The convergence in probability in \eqref{eq:r-P-sub} holds for each
fixed $\theta,s$ in the stated range.
There is no conclusion at $s=D(\theta)$.
\end{theorem}

The proof is given in Section~\ref{sec:cluster-proof}. The lower bound is obtained from energy estimates for suitable measures supported on complete blocks, while the upper bound follows from explicit covers combined with the Hölder regularity of fractional Brownian motion.

The three constraints in \eqref{eq:general-threshold} correspond to
covering individual points, the whole block span, and the individual
cluster spans. The third constraint has three branches according to
whether a cluster image is longer than $r^\theta$, between $r$ and
$r^\theta$, or shorter than $r$. The theorem concerns a changing
family; every individual finite set $F_r$ has ordinary intermediate
dimension zero. No continuity or nesting in $r$ is required.

\subsection{Clusters scaled by the eigenvalue}

The first specialization gives the width law for the natural
input-distribution example.

\begin{corollary}[Eigenvalue-relative widths]
\label{cor:relative-width-transfer}
Let $\rho>0$, $b\geq1$, and $\beta>\omega>0$. Suppose that the positive sequence
$(a_m)_{m\geq1}$ satisfies
\begin{equation}\label{eq:relative-centers}
 a_m\sim m^{-\rho},\quad a_m-a_{m+1}\sim m^{-\rho-1},\quad m\geq1.
\end{equation}
For $1\leq m\leq N_r=\lfloor r^{-\omega}\rfloor$, suppose that
\begin{equation*}
 \#C_{r,m}\leq Cm^{b-1},\quad
 C_{r,m}\subset[a_m-Cr^\beta m^{-\rho},
                a_m+Cr^\beta m^{-\rho}],
\end{equation*}
and that $C_{r,m}$ contains $J_{r,m}\sim m^{b-1}$ ordered
points satisfying
\begin{equation}\label{eq:relative-subset}
 |x_{m,j}-x_{m,i}|\geq cr^\beta m^{-\rho}
                    \frac{|j-i|}{J_{r,m}},\quad 1\leq i,j\leq J_{r,m}.
\end{equation}
All implicit constants are independent of $r,m$. A fixed number of initial
bands of uniformly bounded cardinality may be omitted from the cluster
assumptions. Then the joint covering threshold of
$B_H(F_r)$
is
\begin{equation*}
 D_{\beta,\omega}^{\rho,b,H}(\theta)
 =\max_{0\leq u\leq\omega}
 \min\left\{bu,\left(1-\frac{\rho Hu}{\theta}\right)_+,
             Q_\theta\bigl(u,H(\beta+\rho u)\bigr)\right\},
\end{equation*}
with all probability and scale qualifications of
Theorem~\ref{thm:general-cluster-transfer}.
\end{corollary}

The condition $\beta>\omega$ keeps distinct spectral bands separated
throughout the retained window. Obtaining the subset
\eqref{eq:relative-subset} from a perturbation expansion can require
a stronger condition, as in the application below.

\subsection{Clusters scaled by the adjacent spectral gap}

Let $q>p>0$. For each $\eps\in[0,1]$, suppose that
$(a_m(\eps))_{m\geq1}$ is a positive, strictly decreasing sequence
converging to zero, and that
\begin{align}
 a_m(\eps)&\sim m^{-q}+\eps m^{-p},
 \label{eq:res-values}\\
 g_m(\eps):=a_m(\eps)-a_{m+1}(\eps)
 &\sim m^{-q-1}+\eps m^{-p-1},
 \label{eq:res-gaps}
\end{align}
with comparison constants independent of $m,\varepsilon$. Let $L_m$ be positive integers satisfying
$L_m\sim m^{b-1}$. Fix
\[
 H\in(0,1),\qquad \kappa,\chi\geq0,\qquad
 c_0\in(0,1),\qquad \omega\in(0,\infty].
\]
For every $m\geq1$ and $0<r<1$, choose deterministic points
\[
 0\leq\xi_{m,0}(r)<\cdots<\xi_{m,L_m-1}(r)\leq1
\]
satisfying the uniform separation condition
\begin{equation}\label{eq:vs-quasiuniform}
 \xi_{m,j}(r)-\xi_{m,i}(r)
 \geq c_{\rm sep}\frac{j-i}{L_m},\quad 0\leq i<j<L_m,
\end{equation}
for a fixed $c_{\rm sep}>0$, independent of $m,r,i,j$. Define
\begin{align*}
 \lambda_{m,j}(r)
 &=a_m(r^\kappa)-c_0r^\chi g_m(r^\kappa)\xi_{m,j}(r),\quad 0\leq j<L_m,\\
 F_r&=\{0\}\cup
 \{\lambda_{m,j}(r):1\leq m\leq N_r,\ 0\leq j<L_m\}.
\end{align*}
where $N_r=\lfloor r^{-\omega}\rfloor$ if $\omega<\infty$, and
$N_r=\infty$ otherwise. Thus there are on the order of $M^b$ distinct values in
the first $M$ bands. Each band lies within a fraction $c_0r^\chi$ of the
gap immediately below its original eigenvalue. In particular, the
values remain positive and different bands do not overlap.

Condition~\eqref{eq:vs-quasiuniform} includes the equal grids
$\xi_{m,j}=j/L_m$ and, more generally, quasiuniform configurations
whose consecutive spacings are comparable to $L_m^{-1}$. Only the
lower comparison is required. The configuration may change with $r$
without any continuity or nesting assumption. All constants below
are uniform over configurations with the stated separation constant.

To state the resulting covering threshold more conveniently, define
\begin{equation*}
 \ell_\kappa(u)=\min\{qu,\kappa+pu\},\qquad
 t(u)=H\ell_\kappa(u),\qquad
 v(u)=H\bigl(\chi+\ell_\kappa(u)+u\bigr).
\end{equation*}
For $u\geq0$ let
\begin{equation*}
 d_\theta(u)=\min\left\{
 bu,\ \left(1-\frac{t(u)}{\theta}\right)_+,\
 Q_\theta\bigl(u,v(u)\bigr)\right\},
 \qquad x_+=\max\{x,0\}.
\end{equation*}

\begin{corollary}[Image thresholds for gap-scaled clusters]
\label{cor:gap-width-transfer}\label{thm:vanishing-clusters}
For each fixed deterministic family satisfying
\eqref{eq:res-values}, \eqref{eq:res-gaps}, and
\eqref{eq:vs-quasiuniform}, and each fixed $a>1$, almost surely
simultaneously for every $0<\theta\leq1$, the covering threshold of
$B_H(F_r)$ along $r_n=a^{-n}$ is
\begin{equation}\label{eq:vs-threshold}
 D_{\kappa,\chi,\omega,b}(\theta)
 =\max_{0\leq u\leq\min\{\omega,\theta/(pH)\}}d_\theta(u).
\end{equation}
The probability, scale, and critical-exponent qualifications are as in
Theorem~\ref{thm:general-cluster-transfer}.
\end{corollary}

Corollaries~\ref{cor:relative-width-transfer} and
\ref{cor:gap-width-transfer} are separate specializations of
Theorem~\ref{thm:general-cluster-transfer}. The input-distribution
perturbation in Section~\ref{sec:natural-input} realizes the
eigenvalue-relative widths of the first corollary. The gap-scaled
configurations of the second are abstract point families: the ReLU
calculation in Appendix~\ref{sec:kernels} supplies examples of their
centre and gap laws. It does not realize every prescribed splitting
configuration as a perturbed operator spectrum.

\section{Covering and energy estimates}\label{sec:cluster-proof}

\subsection{Preliminary lemmas}

The first three lemmas also apply to the Cantor constructions in
Section~\ref{sec:tree-entropy}. Fix $0<s<1$, $0<\theta\leq1$, and put
$A=s+\theta(1-s)$. Define the truncated covering kernel and its energy by
\begin{equation*}
 K^s_{r,\theta}(z)=
 \begin{cases}
 1,&|z|\leq r,\\
 (r/|z|)^s,&r<|z|\leq r^\theta,\\
 0,&|z|>r^\theta,
 \end{cases}
 \qquad
 I^s_{r,\theta}(\eta)=
 \iint K^s_{r,\theta}(x-y)\dd\eta(x)\dd\eta(y).
\end{equation*}
The inclusion of the outer endpoint makes the definition valid at
$\theta=1$.

\begin{lemma}\label{lem:cover-gaussian}
If $\eta$ is a Borel probability measure supported on a bounded set
$F\subset\R$, then
\begin{equation}\label{eq:cover-energy}
 S^s_{r,\theta}(F)\geq\frac{r^s}{I^s_{r,\theta}(\eta)}.
\end{equation}
\end{lemma}

\begin{proof}
Let $(U_i)$ be an admissible interval cover of $F$ and write $d_i=|U_i|$.
Form a disjoint partition by setting
\begin{equation*}
    A_1=U_1,\quad A_i=U_i\setminus\bigcup_{j<i}U_j,\quad i\geq2.
\end{equation*}
Then
$\sum_i\eta(A_i)=1$ and $A_i\subset U_i$. If $(x,y)\in A_i\times A_i$, we have $|x-y|\leq |U_i|=d_i\in[r,r^\theta]$, and hence $K^s_{r,\theta}(x-y)\geq r^sd_i^{-s}$. Then
\begin{align*}
    I^s_{r,\theta}(\eta)&=\iint K^s_{r,\theta}(x-y)\dd\eta(x)\dd\eta(y)\\
    &\geq \sum_{i}\iint_{A_i\times A_i} K^s_{r,\theta}(x-y)\dd\eta(x)\dd\eta(y)\\
    &\geq r^s\sum_{i}d_i^{-s}\eta(A_i)^2.
\end{align*}
By Cauchy-Schwarz inequality,
\begin{equation*}
    \sum_{i}d_i^{-s}\eta(A_i)^2\cdot\sum_{i}d_i^{s}\geq\bigg(\sum_i\eta(A_i)\bigg)^2=1.
\end{equation*}
Taking the infimum over covers proves the result.
\end{proof}

\begin{lemma}[Gaussian increment estimates]\label{lem:gaussian-increments}
For distinct $x,y\geq0$,
\begin{equation}\label{eq:vs-gaussian-three}
 \E K^s_{r,\theta}(B_H(x)-B_H(y))
 \leq C_s\min\left\{
 1,\ r^s|x-y|^{-Hs},\ r^A|x-y|^{-H}\right\}.
\end{equation}
The constant is independent of $r,x,y$.
\end{lemma}

\begin{proof}
Let $Z\sim N(0,1)$. We establish the three estimates separately.
\begin{itemize}
    \item $\E K^s_{r,\theta}(B_H(x)-B_H(y))\lesssim_s 1$.
\end{itemize}

Since $0\leq K^s_{r,\theta}
\leq1$, the first estimate follows immediately.

\begin{itemize}
    \item $\E K^s_{r,\theta}(B_H(x)-B_H(y))\lesssim_s r^s|x-y|^{-Hs}$.
\end{itemize}

We have the pointwise inequality $K^s_{r,\theta}(z)\leq r^s|z|^{-s}$ for $z\ne0$ and the increment
law $B_H(x)-B_H(y)\stackrel{\mathrm d}=|x-y|^HZ$, which give
\[
 \E K^s_{r,\theta}(B_H(x)-B_H(y))
 \leq r^s|x-y|^{-Hs}\E|Z|^{-s}\lesssim_s r^s|x-y|^{-Hs}.
\]

\begin{itemize}
    \item $\E K^s_{r,\theta}(B_H(x)-B_H(y))\lesssim_s r^A|x-y|^{-H}$.
\end{itemize}

Using again the increment law and the definition of the kernel, we decompose
\begin{align*}
    \E K^s_{r,\theta}(B_H(x)-B_H(y))&=\mathbb{P}(|x-y|^H|Z|\leq r)\\
    &\quad+\E[K^s_{r,\theta}(|x-y|^HZ)\mathbf{1}_{\{r<|x-y|^H|Z|\leq r^{\theta}\}}]\\
    &=: I+II.
\end{align*}
For $I$,
\begin{equation*}
    I=2^{1/2}\pi^{-1/2}\int_{0}^{r|x-y|^{-H}}e^{-1/2z^2}\dd z\lesssim r|x-y|^{-H}.
\end{equation*}
For $II$,
\begin{align*}
    II&=\E\bigg[\bigg(\frac{r}{|x-y|^H|Z|}\bigg)^s\mathbf{1}_{\{r|x-y|^{-H}<|Z|\leq r^{\theta}|x-y|^{-H}\}}\bigg]\\
    &=2^{1/2}\pi^{-1/2}r^s|x-y|^{-Hs}\int_{r|x-y|^{-H}}^{r^{\theta}|x-y|^{-H}}z^{-s}e^{-1/2z^2}\dd z\\
    &\lesssim_s (r^{s+\theta(1-s)}-r)|x-y|^{-H}.
\end{align*}
Since $A=s+\theta(1-s)\leq 1$, we have $r\leq r^A$ for $0<r<1$. Hence, 
\begin{equation*}
    \E K^s_{r,\theta}(B_H(x)-B_H(y))=I+II\lesssim_s r^A|x-y|^{-H},
\end{equation*}
which proves the third estimate. Taking the minimum of the three bounds proves the lemma.
\end{proof}

The right-hand side of \eqref{eq:vs-gaussian-three} is the scalar-fBm
specialization of Burrell's capacity kernel
\cite[equation~(2.1) and Lemma~3.2]{Burrell2022}. With target dimension
one and $\gamma=1/H$, its normalized form is
\[
 \phi^{sH,H}_{r^{1/H},\theta}(z)
 =\min\{1,r^s|z|^{-Hs},r^A|z|^{-H}\},\qquad z\ne0,
\]
with value one at zero. The cluster estimates below concern its
uniform energy on a support that changes with $r$.

For interval covers, define the continuous functions
\begin{equation*}
 J_{s,\theta}(z)=
 \begin{cases}
 z-\theta(1-s),&0\leq z\leq\theta,\\
 sz,&\theta\leq z\leq1,\\
 s,&z\geq1,
 \end{cases}
 \qquad
 e_{s,\theta}(z)=s-J_{s,\theta}(z).
\end{equation*}
Equivalently,
\begin{equation*}
 e_{s,\theta}(z)=\max\{0,s(1-z),A-z\}.
\end{equation*}

\begin{lemma}
\label{lem:vs-span-cover}\label{lem:vs-span-gaussian}
An interval of diameter at most $Cr^{\alpha}$, where $\alpha\geq0$ and $C$ is independent of $r,\alpha$, has an admissible cover of $s$-cost at most
$C' r^{J_{s,\theta}(\alpha)}$. The constant $C'$ can be chosen uniformly
in $r$ and $\alpha$.
\end{lemma}

\begin{proof}

We distinguish three cases according to the size of $\alpha$. 
If $0\le \alpha\le\theta$, we can cover the interval by $\lesssim r^{\alpha-\theta}$ intervals of
diameter $r^\theta$. Hence
\begin{equation*}
    S^s_{r,\theta}(I)
\le C'r^{\alpha-\theta}(r^\theta)^s
= C'r^{\alpha-\theta(1-s)}.
\end{equation*}
If $\theta\le\alpha\le1$, the interval has diameter at most
$Cr^\alpha$, with $r\le r^\alpha\le r^\theta$. It can therefore
be covered by a uniformly bounded number of admissible intervals
of diameter comparable to $r^\alpha$, giving
\begin{equation*}
    S^s_{r,\theta}(I)\le C'r^{\alpha s}.
\end{equation*}
If $\alpha\ge1$, the interval has diameter at most
$Cr^\alpha\le Cr$, and hence can be covered by a uniformly bounded
number of intervals of diameter $r$. Thus
\begin{equation*}
    S^s_{r,\theta}(I)\le C'r^s.
\end{equation*}
Therefore,
\begin{equation*}
    S^s_{r,\theta}(I)\le C'r^{J_{s,\theta}(\alpha)}.
\end{equation*}
\end{proof}

\subsection{One block estimate for both width laws}

Retain the scales $\mathcal L_r(M),\mathcal W_r(M)$ from
Section~\ref{sec:vanishing-splitting}. On a complete block, let
$\mu_{r,M}$ be uniform measure on
$\bigcup_{M\leq m<2M}X_{r,m}$, which contains on the order of $M^b$
points, and let $\nu_{r,M}$ be the push-forward of $\mu_{r,M}$ by $B_H$. 
More precisely, we denote $\bigcup_{M\leq m<2M}X_{r,m}:=\{x_1,\cdots,x_N\}$, $N\sim M^b$, then
\begin{equation*}
    \mu_{r,M}=\frac{1}{N}\sum_{i=1}^{N}\delta_{x_i},\quad \nu_{r,M}=\frac{1}{N}\sum_{i=1}^{N}\delta_{B_H(x_i)},
\end{equation*}
where $\delta_{x}$ is the Dirac mass at point $x$. We define
\[
 \Psi_r(z)=\min\{1,r^sz^{-s},r^Az^{-1}\},\qquad z>0.
\]
This function is nonincreasing and satisfies
\begin{equation}\label{eq:psi-contraction}
 \Psi_r(az)\leq a^{-1}\Psi_r(z),
 \ a\in(0,1],\quad \text{and}
 \quad
 \Psi_r(r^{\alpha})=r^{e_{s,\theta}(\alpha)},\ \alpha\geq0.
\end{equation}

\begin{lemma}\label{lem:vs-block-energy}
Under (\ref{ass:cluster-count}) and (\ref{ass:cluster-separation}), for
every complete retained block,
\begin{equation*}
 \E I^s_{r,\theta}(\nu_{r,M})
 \leq C\left[M^{-b}
      +\Psi_r\bigl(\mathcal L_r(M)^H\bigr)
      +M^{-1}\Psi_r\bigl(\mathcal W_r(M)^H\bigr)\right].
\end{equation*}
If (\ref{ass:cluster-exponents}) also holds, then at $u_M=\log M/\log(1/r)$,
\begin{equation}\label{eq:vs-energy-u}
 \E I^s_{r,\theta}(\nu_{r,M})
 \leq C\left[r^{bu_M}+r^{e_{s,\theta}(H\ell(u_M))}
                    +r^{u_M+e_{s,\theta}(Hw(u_M))}\right].
\end{equation}
\end{lemma}

\begin{proof}
Let $J_m=J_{r,m}$ and $X_{r,m}=\{x_{m,1},\cdots,x_{m,J_m}\}$. By the definition of $\nu_{r,M}$, we have
\[
\begin{aligned}
\mathbb E I^s_{r,\theta}(\nu_{r,M})
&=
\frac1{N^2}
\sum_{i,j=1}^{N}
\mathbb E
K^s_{r,\theta}\bigl(B_H(x_i)-B_H(x_j)\bigr)\\
&=
\frac1{N^2}
\sum_{\substack{1\le i,j\le N\\i\ne j}}
\mathbb E
K^s_{r,\theta}\bigl(B_H(x_i)-B_H(x_j)\bigr)
+\frac1N.
\end{aligned}
\]

By Lemma~\ref{lem:gaussian-increments}, we have
\[
\mathbb E
K^s_{r,\theta}\bigl(B_H(x_i)-B_H(x_j)\bigr)
\lesssim_s
\Psi_r(|x_i-x_j|^H),
\]
and therefore
\[
\mathbb E I^s_{r,\theta}(\nu_{r,M})
\lesssim_s
\frac1{N^2}
\sum_{\substack{1\le i,j\le N\\i\ne j}}
\Psi_r(|x_i-x_j|^H)
+M^{-b}.
\]

We split the off-diagonal sum according to whether the two points
belong to the same cluster $x_i\sim x_j$ or to different clusters $x_i\nsim x_j$:
\begin{align*}
    \mathcal A:&=
\frac1{N^2}
\sum_{\substack{1\le i,j\le N\\i\ne j}}
\Psi_r(|x_i-x_j|^H)\\
&=\frac1{N^2}\sum_{\substack{1\le i\ne j\le N\\x_i\sim x_j}}\Psi_r(|x_i-x_j|^H)+\frac1{N^2}\sum_{\substack{1\le i, j\le N\\x_i\nsim x_j}}\Psi_r(|x_i-x_j|^H):=A_1+A_2.
\end{align*}

We first estimate $A_1$, which corresponds to pairs of points lying in the same cluster,
we have
\[
A_1
=
\frac1{N^2}
\sum_{m=M}^{2M-1}
\sum_{\substack{1\le i,j\le J_{r,m}\\i\ne j}}
\Psi_r\bigl(|x_{m,i}-x_{m,j}|^H\bigr).
\]

By the separation condition \eqref{eq:general-inner-spacing},
\[
|x_{m,i}-x_{m,j}|
\ge
c\mathcal W_r(M)\frac{|j-i|}{J_{m}}.
\]

Since $\Psi_r$ is decreasing and satisfies the scaling bound in \eqref{eq:psi-contraction}, 
\begin{equation*}
    \Psi_r\bigl(|x_{m,i}-x_{m,j}|^H\bigr)
\lesssim
\Psi_r\left(
\mathcal W_r(M)^H
\frac{|j-i|^H}{J_{r,m}^H}
\right)\lesssim
\left(\frac{|j-i|}{J_{r,m}}\right)^{-H}
\Psi_r\bigl(\mathcal W_r(M)^H\bigr).
\end{equation*}

Consequently,
\[
\begin{aligned}
A_1
&\lesssim
\frac1{N^2}
\sum_{m=M}^{2M-1}
\sum_{\substack{1\le i,j\le J_{r,m}\\i\ne j}}
\left(\frac{|j-i|}{J_{r,m}}\right)^{-H}
\Psi_r\bigl(\mathcal W_r(M)^H\bigr)\\
&=
\frac1{N^2}\sum_{m=M}^{2M-1}J_m^H\left(
\sum_{1\le i\ne j\le J_m}|j-i|^{-H}
\right)
\Psi_r\bigl(\mathcal W_r(M)^H\bigr).
\end{aligned}
\]

Since $0<H<1$, for every integer $J\geq2$,
\[
\sum_{1\le i\ne j\le J}
|j-i|^{-H}
 =2\sum_{h=1}^{J-1}(J-h)h^{-H}
\lesssim_H
J^{2-H}.
\]
The sum is zero for $J=1$; thus $J_m^H$ times this sum is at most
$C_HJ_m^2$, and $N^{-2}\sum_{m=M}^{2M-1}J_m^2\lesssim M^{-1}$.

Therefore,
\begin{equation*}
    A_1
\lesssim
M^{-1}\Psi_r\bigl(\mathcal W_r(M)^H\bigr).
\end{equation*}

We next estimate \(A_2\), corresponding to points belonging to distinct clusters, we have
\begin{equation*}
    A_2=\frac{1}{N^2}
\sum_{\substack{M \le m,k < 2M\\ m\ne k}}
\sum_{i=1}^{J_m}
\sum_{j=1}^{J_k}
\Psi_r\left(\lvert x_{m,i}-x_{k,j}\rvert^H\right).
\end{equation*}

By the separation condition \eqref{eq:general-cross-spacing}, 
\begin{equation*}
    |x_{m,i}-x_{k,j}|
\ge
c\mathcal L_r(M)\frac{|m-k|}{M},
\quad m\ne k.
\end{equation*}

As above, monotonicity and the scaling property of \(\Psi_r\) give,
\[
\Psi_r\bigl(|x_{m,i}-x_{k,j}|^H\bigr)
\lesssim
\left(\frac{|m-k|}{M}\right)^{-H}
\Psi_r\bigl(\mathcal L_r(M)^H\bigr).
\]

It follows that
\[
\begin{aligned}
A_2
&\lesssim
\frac1{N^2}
\sum_{\substack{M\le m,k<2M\\m\ne k}}
\sum_{i=1}^{J_{r,m}}
\sum_{j=1}^{J_{r,k}}
\left(\frac{|m-k|}{M}\right)^{-H}
\Psi_r\bigl(\mathcal L_r(M)^H\bigr)\\
&\sim M^{H-2}\left(
\sum_{\substack{M\le m,k<2M\\m\ne k}}
|m-k|^{-H}
\right)
\Psi_r\bigl(\mathcal L_r(M)^H\bigr)\lesssim
\Psi_r\bigl(\mathcal L_r(M)^H\bigr).
\end{aligned}
\]

Combining the diagonal contribution with the estimates for $A_1$
and $A_2$ yields
\[
\mathbb E I^s_{r,\theta}(\nu_{r,M})
\le
C\left[
M^{-b}
+\Psi_r\bigl(\mathcal L_r(M)^H\bigr)
+M^{-1}\Psi_r\bigl(\mathcal W_r(M)^H\bigr)
\right].
\]
as required. Substituting \eqref{eq:general-scales} and using
\eqref{eq:psi-contraction} gives \eqref{eq:vs-energy-u}.
\end{proof}

The measure uses only the separated subset, whereas the upper covers
below retain all points in each cluster. A positive-proportion subset
therefore suffices even when the smallest spacing among all cluster
points is much smaller.

\subsection{Proof of the transfer theorem}

For this proof, write
\[
 d_\theta(u)=\min\left\{bu,\left(1-\frac{H\ell(u)}{\theta}\right)_+,
                         Q_\theta(u,Hw(u))\right\}.
\]
The function $d_\theta(u)$ is continuous in $(u,\theta)$ for $\theta>0$;
so is $D(\theta)=\max_{0\leq u\leq\omega}d_\theta(u)$.

\begin{lemma}\label{lem:subcritical-lower-bound}
Let $0<s<D(\theta)$. Then there exist $\zeta>s$ and $r_0>0$
such that, for every $0<\eta<\zeta-s$,
\begin{equation*}
\mathbb P\left(
S^s_{r,\theta}(B_H(F_r))
\ge r^{s-\zeta+\eta}
\right)
\ge 1-Cr^\eta,
\qquad 0<r<r_0.
\end{equation*}
\end{lemma}

\begin{proof}
Since $0<s<D(\theta)$, $d_\theta(0)=0$, and $d_{\theta}(u)$ is continuous, there exists $u_0\in(0,\omega)$ such that $s<d_\theta(u_0).$ Then we obtain three inequalities
\begin{numcases}{}
s < bu_0, 
    & \label{eq:cond1}\\
s < \left(1-\dfrac{H\ell(u_0)}{\theta}\right)_+,
    & \label{eq:cond2}\\
s < Q_\theta\bigl(u_0,Hw(u_0)\bigr),
    & \label{eq:cond3}
\end{numcases}

Let $M=\lfloor r^{-u_0}\rfloor$ and $u_M(r)=\log M/\log(1/r)$, then we have
\begin{equation}\label{convergence}
    u_M\to u_0\text{ as }r\to0.
\end{equation}
Since $u_0<\omega$, we also have $2M\leq N_r+1$ for all sufficiently
small $r$, so the block is complete and avoids any exceptional initial bands.

We verify that these inequalities persist, with a strict margin, at $u_M$
for all sufficiently small $r$. 
First, by \eqref{eq:cond1} and \eqref{convergence}, there exist $\varepsilon_1>0$ and $r_1>0$
such that
\[
bu_M\ge s+\varepsilon_1,
\quad 0<r<r_1.
\]

Next, by \eqref{eq:cond2}, since $s>0$,
\[
s<
1-\frac{H\ell(u_0)}{\theta}.
\]
This is equivalent to
\[
A-H\ell(u_0)>s.
\]

By continuity of $\ell$ and \eqref{convergence}, there exist $\varepsilon_2>0$ and
$r_2>0$ such that
\[
e_{s,\theta}\bigl(H\ell(u_M)\bigr)
=
A-H\ell(u_M)
\ge s+\varepsilon_2,
\quad 0<r<r_2.
\]

Finally, by a straightforward calculation, \eqref{eq:cond3} is equivalent to
\[
u_0+e_{s,\theta}\bigl(Hw(u_0)\bigr)>s.
\]

By continuity and \eqref{convergence}, there exist $\varepsilon_3>0$ and $r_3>0$ such that
\[
u_M+e_{s,\theta}\bigl(Hw(u_M)\bigr)
\ge s+\varepsilon_3,
\qquad 0<r<r_3.
\]

Then we choose $0<\varepsilon_0
<
\min\{\varepsilon_1,\varepsilon_2,\varepsilon_3\}$ and let $\zeta=s+\varepsilon_0>s.$ For all sufficiently small $r<\min\{r_1,r_2,r_3\}$, Lemma~\ref{lem:vs-block-energy} gives
\[
\begin{aligned}
\mathbb E I^s_{r,\theta}(\nu_{r,M})
&\le
C\left[
r^{bu_M}
+r^{e_{s,\theta}(H\ell(u_M))}
+r^{u_M+e_{s,\theta}(Hw(u_M))}
\right]\le 3Cr^\zeta.
\end{aligned}
\]

Hence we obtain,
\begin{equation*}
    \mathbb E I^s_{r,\theta}(\nu_{r,M})
\le Cr^\zeta.
\end{equation*}

Now fix $0<\eta<\zeta-s.$ By Markov's inequality,
\[
\mathbb P\left(
I^s_{r,\theta}(\nu_{r,M})
>r^{\zeta-\eta}
\right)
\le
\frac{\mathbb E I^s_{r,\theta}(\nu_{r,M})}
{r^{\zeta-\eta}}
\le Cr^\eta.
\]

Hence
\[
\mathbb P\left(
I^s_{r,\theta}(\nu_{r,M})
\le r^{\zeta-\eta}
\right)
\ge1-Cr^\eta.
\]

Since $\nu_{r,M}$ is supported on
\[
B_H\left(
\bigcup_{M\le m<2M}X_{r,m}
\right)
\subset B_H(F_r),
\]

Lemma~\ref{lem:cover-gaussian} yields
\[
S^s_{r,\theta}(B_H(F_r))
\ge
\frac{r^s}{I^s_{r,\theta}(\nu_{r,M})}.
\]

Therefore,
\begin{align*}
    \mathbb P\bigl(S^s_{r,\theta}(B_H(F_r))\ge r^{s-\zeta+\eta}\bigr)
    \geq \mathbb P\bigl(I^s_{r,\theta}(\nu_{r,M})
\le r^{\zeta-\eta}\bigr)\geq 1-Cr^{\eta}.
\end{align*}
\end{proof}

\begin{proof}[Proof of Theorem~\ref{thm:general-cluster-transfer}]

We need to prove \eqref{eq:rn-a.s.-sub},\eqref{eq:r-P-sub} and \eqref{eq:r-a.s.-super}.

\smallskip\noindent\emph{Step 1: subcritical divergence in probability at all scales \eqref{eq:r-P-sub}.}
Fix $0<s<D(\theta)$. By Lemma~\ref{lem:subcritical-lower-bound}, choose $\zeta>s$ and $0<\eta<\zeta-s$. Since $s-\zeta+\eta<0,$ we have $r^{s-\zeta+\eta}\to+\infty
\text{ as }r\to 0^+.$ Hence, for every $R>0$, there exists $r_R>0$ such that
\[
r^{s-\zeta+\eta}>R,
\qquad 0<r<r_R.
\]
Therefore, 
\[
\begin{aligned}
\mathbb P\left(
S^s_{r,\theta}(B_H(F_r))>R
\right)
&\ge
\mathbb P\left(
S^s_{r,\theta}(B_H(F_r))
\ge r^{s-\zeta+\eta}
\right)\ge 1-Cr^\eta.
\end{aligned}
\]
Letting $r\to 0^+$ gives
\[
\mathbb P\left(
S^s_{r,\theta}(B_H(F_r))>R
\right)\longrightarrow1.
\]
Since $R>0$ is arbitrary, we have
\[
S^s_{r,\theta}(B_H(F_r))
\xrightarrow[r\to0]{P}+\infty.
\]

\smallskip\noindent\emph{Step 2: subcritical almost-sure divergence along geometric scales \eqref{eq:rn-a.s.-sub}.}
Fix $a>1$ and put $r_n=a^{-n}$. By Lemma~\ref{lem:subcritical-lower-bound},
there exist $\zeta>s$ and $n_0$ such that, for any fixed
$0<\eta<\zeta-s$,
\[
 \Prob\left\{S^s_{r_n,\theta}(B_H(F_{r_n}))
                  <r_n^{s-\zeta+\eta}\right\}
 \leq Cr_n^\eta=Ca^{-n\eta},\qquad n\geq n_0.
\]
Since
\[
 \sum_{n=n_0}^\infty
 \Prob\left\{S^s_{r_n,\theta}(B_H(F_{r_n}))
                  <r_n^{s-\zeta+\eta}\right\}
 \leq C\sum_{n=n_0}^\infty a^{-n\eta}<\infty,
\]
Borel--Cantelli gives, almost surely, a finite random $n_1$ such that
\[
 S^s_{r_n,\theta}(B_H(F_{r_n}))
 \geq r_n^{s-\zeta+\eta},\qquad n\geq n_1.
\]
As $s-\zeta+\eta<0$,
\[
 S^s_{r_n,\theta}(B_H(F_{r_n}))
 \xrightarrow[n\to\infty]{\mathrm{a.s.}}+\infty.
\]

\smallskip\noindent\emph{Step 3: supercritical almost-sure convergence at all scales \eqref{eq:r-a.s.-super}.}
On the probability-one event on which $B_H$ is $\gamma$-H\"older
on $[0,A_0]$ for every $0<\gamma<H$, fix $D(\theta)<s<1$.
For each such $\gamma$, there is a finite random constant $C_\gamma$ with
\[
 |B_H(x)-B_H(y)|\leq C_\gamma|x-y|^\gamma,
 \qquad x,y\in[0,A_0].
\]
For a retained block $\mathcal B_r(M)$, put
$u_M=\log M/\log(1/r)$. We construct three admissible covers.
First, (\ref{ass:cluster-count}) bounds its cardinality by $CM^b$,
so covering each image point by an interval of diameter $r$ gives
\begin{equation}\label{eq:point-cover}
 S^s_{r,\theta}\bigl(B_H(\mathcal B_r(M))\bigr)
 \lesssim r^sM^b=r^{s-bu_M}.
\end{equation}
Second, (\ref{ass:cluster-spans}) and (\ref{ass:cluster-exponents}) give
\[
 \operatorname{diam}B_H(\mathcal B_r(M))
 \lesssim_\gamma\mathcal L_r(M)^\gamma
 \sim r^{\gamma\ell(u_M)}.
\]
Lemma~\ref{lem:vs-span-cover} applied to an interval containing this image yields
\begin{equation}\label{eq:block-span-cover}
 S^s_{r,\theta}\bigl(B_H(\mathcal B_r(M))\bigr)
 \lesssim_\gamma r^{J_{s,\theta}(\gamma\ell(u_M))}.
\end{equation}
Third, the same assumptions give, for every cluster in the block,
\[
 \operatorname{diam}B_H(C_{r,m})
 \lesssim_\gamma\mathcal W_r(M)^\gamma
 \sim r^{\gamma w(u_M)}.
\]
Applying Lemma~\ref{lem:vs-span-cover} to each cluster image and
summing over at most $M$ clusters gives
\begin{equation}\label{eq:cluster-span-cover}
 S^s_{r,\theta}\bigl(B_H(\mathcal B_r(M))\bigr)
 \lesssim_\gamma M r^{J_{s,\theta}(\gamma w(u_M))}
 =r^{J_{s,\theta}(\gamma w(u_M))-u_M}.
\end{equation}
These covers include all cluster points and apply to the final partial block.
Taking the minimum of \eqref{eq:point-cover}, \eqref{eq:block-span-cover},
and \eqref{eq:cluster-span-cover},
\begin{gather*}
 S^s_{r,\theta}\bigl(B_H(\mathcal B_r(M))\bigr)
 \lesssim_\gamma r^{E_\gamma(u_M)},\\
 E_\gamma(u):=\max\left\{
 s-bu,\ J_{s,\theta}(\gamma\ell(u)),\
 J_{s,\theta}(\gamma w(u))-u\right\}.
\end{gather*}
For $0<s<1$ and $t,v,u\geq0$, the piecewise formulas give
\begin{align*}
 J_{s,\theta}(t)>0
 &\iff s>\left(1-\frac{t}{\theta}\right)_+,\\
 J_{s,\theta}(v)>u
 &\iff s>Q_\theta(u,v).
\end{align*}
Consequently,
\[
 E_H(u)>0\iff s>d_\theta(u),\qquad
 \delta_0:=\min_{0\leq u\leq\omega}E_H(u)>0.
\]
Continuity in $(\gamma,u)$ allows a choice of $\gamma\in(H/2,H)$ such that
\[
 E_\gamma(u)\geq\frac{\delta_0}{2},\qquad 0\leq u\leq\omega.
\]
Hence every retained block satisfies
\[
 S^s_{r,\theta}\bigl(B_H(\mathcal B_r(M))\bigr)
 \lesssim_\gamma r^{\delta_0/2}.
\]
Decompose the retained window into the blocks $M_k=2^k$,
$0\leq k\leq\lfloor\log_2N_r\rfloor$, including the final partial block.
There are $O(\log(1/r))$ blocks, and therefore
\begin{align*}
 S^s_{r,\theta}(B_H(F_r))
 &\leq\sum_{k=0}^{\lfloor\log_2N_r\rfloor}
 S^s_{r,\theta}\bigl(B_H(\mathcal B_r(M_k))\bigr)+Cr^s\\
 &\lesssim_\gamma \log(1/r)r^{\delta_0/2}+r^s\longrightarrow0.
\end{align*}
If initial bands are exceptional, omit the finitely many dyadic blocks
meeting them and cover their uniformly bounded number of points at cost
$O(r^s)$. The same bound follows. This pathwise estimate holds for every
sufficiently small real $r$.

\smallskip\noindent\emph{Step 4: simultaneous parameters.}
Intersect the lower-bound events for rational $\theta\in(0,1]$
and rational subcritical $s$. For arbitrary $0<s<D(\theta)$,
continuity permits a rational $\theta'<\theta$ and rational $s'$
with $s<s'<D(\theta')$. Every admissible cover at $\theta$ is
admissible at $\theta'$, and its diameters are less than one, so
\[
 S^s_{r,\theta}(B_H(F_r))
 \geq S^{s'}_{r,\theta'}(B_H(F_r)).
\]
This proves the simultaneous lower bound along the fixed geometric
sequence. The upper-bound argument is deterministic on the common
H\"older event and already holds for all admissible $\theta,s$.
\end{proof}

\subsection{Verification of the two width laws}

\begin{proof}[Proof of Corollary~\ref{cor:relative-width-transfer}]
By Theorem~\ref{thm:general-cluster-transfer}, it suffices to verify
(\ref{ass:cluster-count})--(\ref{ass:cluster-exponents}) with
\[
 \mathcal L_r(M)=M^{-\rho},\qquad
 \mathcal W_r(M)=r^\beta M^{-\rho},\qquad
 \ell(u)=\rho u,\qquad w(u)=\beta+\rho u.
\]
The cardinality bounds in (\ref{ass:cluster-count}) are hypotheses.
Summing adjacent gaps and using the cluster containment gives, also
for partial blocks,
\begin{align*}
 \operatorname{diam}\mathcal B_r(M)
 &\leq a_M-a_{2M-1}+Cr^\beta M^{-\rho}\leq CM^{-\rho},\\
 \operatorname{diam}C_{r,m}&\leq Cr^\beta M^{-\rho}.
\end{align*}
This proves (\ref{ass:cluster-spans}). The within-cluster part of
(\ref{ass:cluster-separation}) is \eqref{eq:relative-subset} with $m\sim M$.
For distinct bands $M\leq i<k<2M$ in the retained window,
\eqref{eq:relative-centers} gives
\[
 a_i-a_k\geq c(k-i)M^{-\rho-1},\qquad
 |x-a_i|+|y-a_k|\leq2Cr^\beta M^{-\rho}
 \quad(x\in C_{r,i},\ y\in C_{r,k}).
\]
Since $\beta>\omega$,
\[
 r^\beta M\leq r^{\beta-\omega}\longrightarrow0.
\]
Thus, for all sufficiently small $r$, uniformly in the retained window,
\begin{align*}
 |x-y|
 &\geq c(k-i)M^{-\rho-1}-2Cr^\beta M^{-\rho}\\
 &\geq\frac c2(k-i)M^{-\rho-1}
 =\frac c2\mathcal L_r(M)\frac{k-i}{M}.
\end{align*}
Finally, $M=r^{-u_M}$ gives (\ref{ass:cluster-exponents}). For small $r$,
all nonexceptional clusters lie in a fixed nonnegative interval, since
their deviations are $O(r^\beta)$ times the positive centres.
Exceptional initial bands are handled as in Theorem~\ref{thm:general-cluster-transfer}.
\end{proof}

\begin{proof}[Proof of Corollary~\ref{cor:gap-width-transfer}]
Set
\[
 \mathcal L_r(M)=M^{-q}+r^\kappa M^{-p},\qquad
 \mathcal W_r(M)=r^\chi\mathcal L_r(M)/M.
\]
The selected subset can be the full band, which has
$L_m\sim m^{b-1}$ distinct points, establishing (\ref{ass:cluster-count}). Equation
\eqref{eq:res-values} bounds the block span by $C\mathcal L_r(M)$,
and \eqref{eq:res-gaps} bounds each cluster span by
$C r^\chi\mathcal L_r(M)/M=C\mathcal W_r(M)$, proving
(\ref{ass:cluster-spans}). For (\ref{ass:cluster-separation}) and
$M\leq i<k<2M$,
\begin{align*}
 \lambda_{i,j}(r)-\lambda_{k,l}(r)
 &\geq a_i(r^\kappa)-a_k(r^\kappa)-c_0g_i(r^\kappa)\\
 &\geq(1-c_0)\bigl(a_i(r^\kappa)-a_k(r^\kappa)\bigr)
 \geq c\mathcal L_r(M)\frac{k-i}{M}.
\end{align*}
Within one band, \eqref{eq:vs-quasiuniform} gives
\begin{equation*}
 |\lambda_{m,j}(r)-\lambda_{m,l}(r)|
 \geq c_0c_{\rm sep}r^\chi g_m(r^\kappa)
                            \frac{|j-l|}{L_m}.
\end{equation*}
Reversing the index order within each cluster gives the increasing
enumeration required in (\ref{ass:cluster-separation}). The scale
relations in (\ref{ass:cluster-exponents}) are uniform because
\[
 M^{-q}+r^\kappa M^{-p}
 =r^{qu_M}+r^{\kappa+pu_M}
 \sim r^{\min\{qu_M,\kappa+pu_M\}},
\]
with
$\ell(u)=\ell_\kappa(u)$ and
$w(u)=\chi+\ell_\kappa(u)+u$. If $\omega<\infty$, the corollary
follows directly from Theorem~\ref{thm:general-cluster-transfer}.

For $\omega=\infty$, fix $U>4/(pH)$ and define
\begin{align*}
 F_r^{(U)}&=\{0\}\cup\{\lambda_{m,j}(r):
 1\leq m\leq\lfloor r^{-U}\rfloor,\ 0\leq j<L_m\},\\
 T_r^{(U)}&=F_r\setminus F_r^{(U)}.
\end{align*}
Theorem~\ref{thm:general-cluster-transfer} gives the threshold
\[
 D_U(\theta)=\max_{0\leq u\leq U}d_\theta(u)
\]
for the truncated family. Since $\ell_\kappa(u)\geq pu$,
\[
 u\geq\frac{\theta}{pH}
 \quad\Longrightarrow\quad
 \left(1-\frac{H\ell_\kappa(u)}{\theta}\right)_+=0
 \quad\Longrightarrow\quad d_\theta(u)=0.
\]
Our choice of $U$ therefore gives
\[
 D_U(\theta)
 =\max_{0\leq u\leq\theta/(pH)}d_\theta(u)
 =D_{\kappa,\chi,\infty,b}(\theta).
\]
For $m>r^{-U}$, \eqref{eq:res-values} implies
\[
 0<\lambda_{m,j}(r)\leq a_m(r^\kappa)
 \lesssim m^{-q}+r^\kappa m^{-p}
 \lesssim m^{-p}\leq r^{pU}.
\]
Hence $T_r^{(U)}\subset[0,Cr^{pU}]$. Fix $\gamma\in(H/2,H)$.
On the common H\"older event,
\[
 \operatorname{diam}B_H(T_r^{(U)})
 \lesssim_\gamma r^{\gamma pU}=o(r),
 \qquad \gamma pU>\frac H2\,p\,\frac4{pH}=2.
\]
For all sufficiently small $r$, one interval of diameter $r$ covers
the tail image. Therefore,
\[
 S^s_{r,\theta}(B_H(F_r^{(U)}))
 \leq S^s_{r,\theta}(B_H(F_r))
 \leq S^s_{r,\theta}(B_H(F_r^{(U)}))+r^s.
\]
The finite-window conclusions and this pathwise estimate give the
same threshold and all stated probability qualifications for the full family.
The vanishing of $d_\theta(u)$ for $u\geq\theta/(pH)$ also justifies
the restricted maximum in \eqref{eq:vs-threshold} for finite $\omega$.
\end{proof}

\subsection{Consequences for splitting and multiplicity}

The parameters $\chi$ and $b$ govern the splitting scale and the growth
of the band multiplicity. We consider a fixed splitting fraction
($\chi=0$), splitting below the smallest image covering scale
($\chi\geq1/H$), and bounded multiplicities ($b=1$).

\begin{lemma}[Two-constraint variational formula]
\label{lem:two-constraint-formula}
Let $q>p>0$, $b>0$, $H\in(0,1)$, $\kappa\geq0$,
$0<\theta\leq1$, and $\omega\in(0,\infty]$. Set
$\ell_\kappa(u)=\min\{qu,\kappa+pu\}$. Then
\[
 \max_{0\leq u\leq\omega}
 \min\left\{bu,\left(1-\frac{H\ell_\kappa(u)}{\theta}\right)_+\right\}
 =\min\left\{b\omega,
 \max\left\{\frac{\theta}{\theta+qH/b},
             \frac{\theta-\kappa H}{\theta+pH/b}\right\}\right\}.
\]
If $\omega=\infty$, the maximum is over $u\geq0$ and the term
$b\omega$ is omitted.
\end{lemma}

\begin{proof}
Write
\[
 h(u)=\left(1-\frac{H\ell_\kappa(u)}{\theta}\right)_+
 =\max\left\{\left(1-\frac{Hqu}{\theta}\right)_+,
              \left(1-\frac{H(\kappa+pu)}{\theta}\right)_+\right\}.
\]
For $a>0$ and $c\geq0$, intersection of the increasing and decreasing
branches gives
\[
 \max_{u\geq0}\min\left\{bu,
       \left(1-\frac{H(c+au)}{\theta}\right)_+\right\}
 =\frac{(\theta-Hc)_+}{\theta+Ha/b}.
\]
Since $\min\{x,\max\{y,z\}\}=\max\{\min\{x,y\},\min\{x,z\}\}$,
\begin{align*}
 z_*:=\max_{u\geq0}\min\{bu,h(u)\}
 &=\max\left\{\frac{\theta}{\theta+qH/b},
              \frac{(\theta-\kappa H)_+}{\theta+pH/b}\right\}\\
 &=\max\left\{\frac{\theta}{\theta+qH/b},
              \frac{\theta-\kappa H}{\theta+pH/b}\right\}.
\end{align*}
The continuous nonincreasing function $h$ satisfies $h(0)=1$ and
vanishes for large $u$. Thus $bu=h(u)$ has a unique solution
$u_*=z_*/b$, and
\[
 \min\{bu,h(u)\}
 =\begin{cases}
 bu,&0\leq u\leq u_*,\\
 h(u),&u\geq u_*.
 \end{cases}
\]
Restriction to $[0,\omega]$ gives $\min\{b\omega,z_*\}$.
\end{proof}

\begin{corollary}[A fixed splitting fraction]\label{cor:vs-fixed-fraction}
When $\chi=0$, the threshold simplifies to
\begin{equation*}
 D_{\kappa,0,\omega,b}(\theta)
 =\min\left\{b\omega,
 \max\left\{\frac{\theta}{qH/b+\theta},
             \frac{\theta-\kappa H}{pH/b+\theta}\right\}\right\}.
\end{equation*}
The term $b\omega$ is omitted if $\omega=\infty$.
\end{corollary}

\begin{proof}
When $\chi=0$, put $t=H\ell_\kappa(u)$ and $v=t+Hu$.
We show that
\[
 Q_\theta(u,v)\geq\left(1-\frac t\theta\right)_+.
\]
This holds immediately for $t\geq\theta$, since $Q_\theta\geq0$.
For $t\leq\theta$, consider the three branches. If $v\leq\theta$,
\[
 Q_\theta(u,v)
 =1-\frac t\theta+\frac{(1-H)u}{\theta}
 \geq1-\frac t\theta.
\]
If $\theta\leq v\leq1$, then
\[
 \theta u-(\theta-t)v
 =\theta(1-H)u+t(v-\theta)\geq0,
 \qquad Q_\theta(u,v)=\frac uv\geq1-\frac t\theta.
\]
If $v\geq1$, then
\[
 Q_\theta(u,v)=u\geq Hu=v-t\geq1-t\geq1-\frac t\theta.
\]
Consequently,
\[
 d_\theta(u)=\min\left\{bu,
             \left(1-\frac{H\ell_\kappa(u)}\theta\right)_+\right\},
\]
and Lemma~\ref{lem:two-constraint-formula} gives the result.
\end{proof}

\begin{corollary}[Splitting below the smallest cover scale]
\label{cor:vs-small-splitting}
If $\chi\geq1/H$, then
\begin{equation*}
 D_{\kappa,\chi,\omega,b}(\theta)
 =\min\left\{\omega,\max\left\{
 \frac{\theta}{qH+\theta},
 \frac{\theta-\kappa H}{pH+\theta}\right\}\right\}.
\end{equation*}
The term $\omega$ is omitted if $\omega=\infty$. Thus this regime
has the same threshold as the truncated unsplit sequence.
\end{corollary}

\begin{proof}
Here $v(u)\geq H\chi\geq1$, so $Q_\theta(u,v(u))=u$.
Since $b\geq1$,
\[
 d_\theta(u)=\min\left\{u,
             \left(1-\frac{H\ell_\kappa(u)}\theta\right)_+\right\}.
\]
Apply Lemma~\ref{lem:two-constraint-formula} with $b=1$.
\end{proof}

\begin{corollary}[Bounded multiplicities]
\label{cor:vs-bounded}\label{cor:vs-bounded-multiplicities}
If $b=1$, then for every $\chi\geq0$,
\[
 D_{\kappa,\chi,\omega,1}(\theta)
 =\min\left\{\omega,\max\left\{
 \frac{\theta}{qH+\theta},
 \frac{\theta-\kappa H}{pH+\theta}\right\}\right\}.
\]
The term $\omega$ is omitted if $\omega=\infty$.
In particular, the $d=2$ ReLU case has no multiplicity-driven change of exponent.
\end{corollary}

\begin{proof}
In all three branches of \eqref{eq:vs-Q},
\[
 Q_\theta(u,v)-u
 =\begin{cases}
 \bigl(u(1-\theta)+\theta-v\bigr)/\theta\geq0,&0\leq v\leq\theta,\\
 u(1/v-1)\geq0,&\theta\leq v\leq1,\\
 0,&v\geq1.
 \end{cases}
\]
Thus the third constraint is redundant when $b=1$, and
Lemma~\ref{lem:two-constraint-formula} applies with $b=1$.
\end{proof}

\section{Spectral clusters from an anisotropic input distribution}
\label{sec:natural-input}

Perturbing the uniform input measure on $S^2$ produces clusters of
width proportional to the eigenvalues of a fixed ReLU tangent kernel.

\subsection{The population operator}

Let $\sigma$ be uniform probability measure on $S^2$, and set
\begin{equation}\label{nat:eq:kernel}
 K(x,y)=\frac{\sqrt{1-t^2}+2t(\pi-\arccos t)}{\pi},
 \qquad t=x\cdot y.
\end{equation}
This is the population two-layer ReLU neural tangent kernel in the
normalization of \eqref{eq:weighted-kernel} with $\eps=1$
\cite{Jacot2018,BM2019}. Write $T$ for its integral operator on
$L^2(\sigma)$. Its eigenvalue on the degree-$2m$ harmonic space is
\begin{equation}\label{nat:eq:even-spectrum}
 a_m=\frac{1}{32\pi}
 \left(\frac{\Gamma(m-1/2)}{\Gamma(m+2)}\right)^2
 (4m^2+2m+4),\qquad m\geq1.
\end{equation}
These spaces have dimension $4m+1$. In particular,
\begin{equation}\label{nat:eq:even-estimates}
 \lim_{m\to\infty}m^3a_m=\frac{1}{8\pi},\qquad
 a_m-a_{m+1}\sim m^{-4}.
\end{equation}
The degree-zero and degree-one eigenvalues are $1/2$ and $1/3$,
and odd degrees $\ell\geq3$ lie in the nullspace. See
\cite{BM2019,BB2021} for the harmonic decay and
Proposition~\ref{prop:eigenvalues} for the exact values and gaps.

For $0<\delta<1/2$, let
\begin{equation*}
 h(x)=P_2(x_3)=\frac{3x_3^2-1}{2},\qquad
 w_\delta(x)=1+\delta h(x),\qquad
 \dd\sigma_\delta=w_\delta\,\dd\sigma.
\end{equation*}
Then $1-\delta/2\leq w_\delta\leq1+\delta$ and
$\int w_\delta\,\dd\sigma=1$. Define
\[
 (T^{\mathrm{in}}_\delta f)(x)
   =\int_{S^2}K(x,y)f(y)\,\dd\sigma_\delta(y),
 \qquad f\in L^2(\sigma_\delta).
\]
Write $M_f$ for multiplication by $f$. For density-change comparisons,
see \cite[Section~2.2, Proposition~1 and Lemma~2]{LYCL2024}.

\begin{proposition}[An exact operator on a fixed Hilbert space]
\label{nat:prop:representation}
The nonzero eigenvalues of $T^{\mathrm{in}}_\delta$, counted with
multiplicity, are exactly those of
\begin{equation*}
 C_\delta=T^{1/2}M_{w_\delta}T^{1/2}
          =T+\delta V,\qquad V=T^{1/2}M_hT^{1/2}.
\end{equation*}
Moreover,
\begin{equation}\label{nat:eq:relative}
 (1-\delta)T\leq C_\delta\leq(1+\delta)T,
 \qquad \ker C_\delta=\ker T.
\end{equation}
If $(\lambda_j)$ and $(\mu_j(\delta))$ are the decreasing lists
of positive eigenvalues of $T$ and $C_\delta$, with multiplicities,
then
\begin{equation}\label{nat:eq:ordered}
 (1-\delta)\lambda_j\leq\mu_j(\delta)
                 \leq(1+\delta)\lambda_j.
\end{equation}
Consequently, if $\mathcal E_\delta$ consists of zero and all distinct
positive eigenvalues of $T^{\mathrm{in}}_\delta$, then
\begin{equation}\label{nat:eq:hausdorff}
 d_{\mathrm H}(\mathcal E_\delta,\mathcal E_0)
                  \leq\delta\|T\|.
\end{equation}
\end{proposition}

\begin{proof}
The unitary map $U_\delta f=\sqrt{w_\delta}f$ from
$L^2(\sigma_\delta)$ to $L^2(\sigma)$ gives
\[
 U_\delta T^{\mathrm{in}}_\delta U_\delta^{-1}
  =M_{\sqrt{w_\delta}}TM_{\sqrt{w_\delta}}=AA^*,
 \qquad A=M_{\sqrt{w_\delta}}T^{1/2}.
\]
For $\lambda>0$, the maps $A^*$ and $\lambda^{-1}A$ identify the
$\lambda$-eigenspaces of $AA^*$ and $A^*A=C_\delta$. Moreover,
\[
 \langle C_\delta f,f\rangle
 =\int w_\delta|T^{1/2}f|^2\dd\sigma
 \in[(1-\delta)\langle Tf,f\rangle,
      (1+\delta)\langle Tf,f\rangle].
\]
This gives \eqref{nat:eq:relative}; min--max gives
\eqref{nat:eq:ordered}. Pairing the ordered lists gives
$|\mu_j(\delta)-\lambda_j|\leq\delta\lambda_j\leq\delta\|T\|$
in both directions, proving \eqref{nat:eq:hausdorff}.
\end{proof}

\subsection{Uniform splitting and a separated subset}

Let $(Y_\ell^k)_{|k|\leq\ell}$ be the $L^2(\sigma)$-normalized complex
harmonics and $P_\ell$ their degree-$\ell$ projection \cite{DaiXu2013}.
Multiplication by $h$ preserves $k$, couples degrees $\ell,\ell\pm2$,
and satisfies
\begin{equation}\label{nat:eq:compression}
 P_\ell M_hP_\ell Y_\ell^k=q_{\ell,k}Y_\ell^k,
 \qquad
 q_{\ell,k}=\frac{\ell(\ell+1)-3k^2}{(2\ell-1)(2\ell+3)}.
\end{equation}
See Appendix~\ref{app:natural-splitting} for the recurrence calculation.

\begin{theorem}[Uniform splitting of harmonic bands]
\label{nat:thm:uniform-splitting}
There are constants $c,C>0$ such that, for $m\geq1$ and
$\delta(m+1)\leq c$, the band originating from $a_m$ consists of
$4m+1$ eigenvalues with multiplicities, indexed by $|k|\leq2m$,
and
\begin{equation}\label{nat:eq:uniform-remainder}
 \left|\mu_{m,k}(\delta)
       -a_m(1+\delta q_{2m,k})\right|
       \leq C\delta^2a_m(m+1).
\end{equation}
They satisfy $\mu_{m,-k}(\delta)=\mu_{m,k}(\delta)$ and lie in
$[(1-\delta)a_m,(1+\delta)a_m]$. The band is separated from the
remaining even spectrum by at least $c'a_m/(m+1)$, for a fixed
$c'>0$.
\end{theorem}

Appendix~\ref{app:natural-splitting} proves the uniform remainder by a
sectorwise Schur complement. For covering estimates we select distinct
values as follows.

\begin{corollary}[Separated values in growing spectral windows]
\label{nat:cor:growing-clusters}
Fix $\omega>0$ and $\beta>2\omega$, and put
$\delta_r=r^\beta$ and $N_r=\lfloor r^{-\omega}\rfloor$.
For all sufficiently small $r$, each $1\leq m\leq N_r$ has an
isolated cluster $C_{r,m}$ of distinct spectral values satisfying
\begin{equation}\label{nat:eq:cluster-upper}
 \#C_{r,m}\leq4m+1,\qquad
 C_{r,m}\subset[a_m-r^\beta a_m,a_m+r^\beta a_m].
\end{equation}
For all sufficiently large $m$, the cluster contains
$J_m\sim m$ points $x_{m,1}<\cdots<x_{m,J_m}$ such that
\begin{equation}\label{nat:eq:cluster-lower}
 c r^\beta m^{-3}\frac{j-i}{m}
 \leq x_{m,j}-x_{m,i}
 \leq C r^\beta m^{-3}\frac{j-i}{m},\qquad 1\leq i<j\leq J_m.
\end{equation}
All comparison constants are independent of $r,m,i,j$. In particular,
the cluster width is comparable to $r^\beta m^{-3}$ and the number
of distinct values is comparable to $m$.
\end{corollary}

\begin{proof}
Set $\ell=2m$ and select
$\lceil\ell/3\rceil\leq k\leq\lfloor2\ell/3\rfloor$.
For $i<j$ in this range,
\[
 q_{\ell,i}-q_{\ell,j}
 =\frac{3(j-i)(j+i)}{(2\ell-1)(2\ell+3)}
 \sim\frac{j-i}{\ell}.
\]
Since $\delta_r m^2\leq r^{\beta-2\omega}\to0$ uniformly for $m\leq N_r$,
\eqref{nat:eq:uniform-remainder} gives
\begin{align*}
 \mu_{m,i}(\delta_r)-\mu_{m,j}(\delta_r)
 &=\delta_r a_m(q_{\ell,i}-q_{\ell,j})+O(\delta_r^2a_m m)\\
 &=\delta_r a_m(q_{\ell,i}-q_{\ell,j})
       \left(1+O\left(\frac{\delta_r m^2}{j-i}\right)\right)
 \sim\delta_r m^{-3}\frac{j-i}{m}.
\end{align*}
Reverse the selected indices to obtain \eqref{nat:eq:cluster-lower}.
Their span gives the lower width bound; the band count and
\eqref{nat:eq:ordered} give \eqref{nat:eq:cluster-upper}.
Finitely many initial bands have bounded cardinality and at least
two distinct first-order values; treat them separately.
\end{proof}

Here $q_{\ell,0}-q_{\ell,1}\sim\ell^{-2}$ and selected bulk gaps are
$\sim\ell^{-1}$, as allowed by the separated-subset hypothesis of
Corollary~\ref{cor:relative-width-transfer}. The width is
$\delta_r a_m\sim\delta_r m^{-3}$. The two conditions are
\[
 \delta_r N_r\to0\quad\text{for band isolation},\qquad
 \delta_r N_r^2\to0\quad\text{for selected spacings}.
\]

\subsection{Control of the complete spectral tail}

For sufficiently small $\delta$, parity and $T^{1/2}Y_\ell^k=0$ for odd
$\ell\geq3$ leave at most three odd positive eigenvalues and one isolated
value originating from degree zero.
Their covering cost is $O(r^s)$. Ordered comparison controls the tail.

\begin{proposition}[Tail bound with multiplicities]
\label{nat:prop:tail}
For an integer $N\geq1$, let
\[
 D_N=4+\sum_{m=1}^N(4m+1)=2N^2+3N+4.
\]
Every eigenvalue beyond the first $D_N$ ordered positive eigenvalues
of $C_\delta$ lies in $[0,CN^{-3}]$, uniformly for
$0<\delta<1/2$. If $\delta N$ is sufficiently small, these first
$D_N$ eigenvalues are precisely the first $N$ isolated even bands and
the four low-degree values, counted with multiplicity.
\end{proposition}

\begin{proof}
For $T$, $\lambda_{D_N+1}=a_{N+1}$ since $1/2,1/3>a_1$. By
\eqref{nat:eq:ordered},
\[
 \mu_{D_N+1}(\delta)\leq(1+\delta)a_{N+1}\lesssim N^{-3}.
\]
For $\delta N$ small, \eqref{nat:eq:even-estimates} separates the first
$N$ relative band intervals and the low-degree intervals.
Theorem~\ref{nat:thm:uniform-splitting} counts their $D_N$ eigenvalues.
\end{proof}

\section{Random images of the complete population spectra}
\label{sec:natural-images}

For $\beta>0$, let $\mathcal E_r^{(\beta)}$ consist of zero and
all distinct positive eigenvalues of $T^{\mathrm{in}}_{r^\beta}$
from Section~\ref{sec:natural-input}, with the fixed kernel \eqref{nat:eq:kernel}.

\begin{theorem}[Complete threshold curves for the natural families]
\label{thm:natural-separation}
For $\beta=1$ and $\beta=2$, the source joint covering thresholds
agree at every $0<\theta\leq1$ and equal
\begin{equation*}
 D_{\mathrm{src}}^{(\beta)}(\theta)=\frac{\theta}{3+\theta}.
\end{equation*}
The source covering costs diverge for $0<s<D_{\mathrm{src}}^{(\beta)}$
and tend to zero for $D_{\mathrm{src}}^{(\beta)}<s<1$ as
$r\downarrow0$ through all real scales.

Let $B=B_{1/2}$ be a common Brownian path. For each prescribed $a>1$,
almost surely simultaneously for every $0<\theta\leq1$, the image
joint thresholds along $r_n=a^{-n}$ are
\begin{align}
 D_B^{(2)}(\theta)&=\frac{2\theta}{3+2\theta},
 \notag\\
 D_B^{(1)}(\theta)&=
 \begin{cases}
 \displaystyle\frac{3+10\theta-\sqrt{9+60\theta+4\theta^2}}
                         {12\theta},&0<\theta\leq3/4,\\[7pt]
 \displaystyle\frac{2\theta}{3+2\theta},&3/4\leq\theta\leq1.
 \end{cases}
 \label{nat:eq:beta-one}
\end{align}
In particular, $D_B^{(1)}(\theta)>D_B^{(2)}(\theta)$ for
$0<\theta<3/4$, and the thresholds agree for
$3/4\leq\theta\leq1$. Supercritical convergence holds almost
surely at all real scales. For each fixed $\theta$ and subcritical
$s$, divergence holds in probability at arbitrary scales and almost
surely along the prescribed geometric sequence. These statements concern
strict exponents; the probability-one event may depend on $a$.
\end{theorem}

\begin{proof}
\noindent\emph{Source comparison.}
Proposition~\ref{nat:prop:representation} gives
\[
 d_{\mathrm H}(\mathcal E_r^{(\beta)},\mathcal E_0)
 \leq\|T\|r^\beta\leq Cr,
 \qquad \beta\in\{1,2\}.
\]
Enlarge each covering interval by $Cr$ at each end, then split at
length $r^\theta$. Since $|U|+2Cr\leq(1+2C)|U|$, the cost comparison
in both directions is
\[
 S^s_{r,\theta}(\mathcal E_r^{(\beta)})
 \sim S^s_{r,\theta}(\mathcal E_0),\qquad 0<\theta\leq1.
\]

Put $d=\theta/(3+\theta)$ and $M=\lfloor r^{-d}\rfloor$.
The head and tail covers of~\cite{FFK2020}, at scales $r$ and $r^\theta$,
include the finitely many low-degree values; using $a_m\sim m^{-3}$ gives
\[
 S^s_{r,\theta}(\mathcal E_0)
 \leq C\{r^{s-d}+r^{\theta(s-d)}+r^{\theta s}\}.
\]
Conversely, $a_m-a_{m+1}\sim m^{-4}$ gives
$\mu_M(U)\lesssim M^{-1}+|U|M^3$ for uniform measure on
$\{a_m:M\leq m<2M\}$. Thus, for $r\leq|U|\leq r^\theta$,
\[
 \mu_M(U)\leq C\{r^{d-s}+r^{\theta(d-s)}\}|U|^s,
 \qquad
 S^s_{r,\theta}(\mathcal E_0)
 \geq\frac{c}{r^{d-s}+r^{\theta(d-s)}}.
\]
This proves the source limits at all real scales.

\smallskip\noindent\emph{A growing head and its baseline threshold.}
Set $\omega=2/5$ and $N_r=\lfloor r^{-2/5}\rfloor$.
For $\beta\in\{1,2\}$, $\beta>2\omega$, so
Corollaries~\ref{nat:cor:growing-clusters}
and~\ref{cor:relative-width-transfer}, with $(\rho,b,H)=(3,2,1/2)$,
give the head threshold
\begin{equation}\label{nat:eq:head-threshold}
 \max_{0\leq u\leq2/5}
 \min\left\{2u,\left(1-\frac{3u}{2\theta}\right)_+,
     Q_\theta\left(u,\frac\beta2+\frac{3u}{2}\right)\right\}.
\end{equation}
By \eqref{eq:vs-Q}, $Q_\theta(u,v)\geq u$ for $u,v\geq0$, so
\eqref{nat:eq:head-threshold} is at least
\begin{equation*}
 d_\theta:=\max_{0\leq u\leq2/5}
          \min\{u,(1-3u/(2\theta))_+\}
          =\frac{2\theta}{3+2\theta}.
\end{equation*}
The maximum occurs at $u=d_\theta\leq2/5$.
For $\beta=2$, $1+3u/2\geq1$ gives $Q_\theta=u$, hence equality.

\smallskip\noindent\emph{Evaluation of the head when $\beta=1$.}
Put $v(u)=1/2+3u/2$ and $G_\theta(u)=(1-3u/(2\theta))_+$.

For $v(u)\leq\theta$ (hence $\theta\geq1/2$),
$Q_\theta(u,v(u))=1-(1+u)/(2\theta)$; its contribution is at most
\begin{equation*}
 \max_{u\geq0}\min\{2u,1-(1+u)/(2\theta)\}
 =\frac{4\theta-2}{4\theta+1}\leq d_\theta.
\end{equation*}
The last inequality is equivalent to $6(\theta-1)\leq0$.

If $\theta\leq v(u)\leq1$, then $u\leq1/3$ and
\[
 Q_\theta(u,v(u))=R(u):=\frac{u}{1/2+3u/2}\leq2u.
\]
On $u\geq0$, $R$ increases and $G_\theta$ decreases before vanishing;
their unique intersection $u_*>0$ maximizes $\min\{G_\theta,R\}$. Moreover,
\[
 u_*\leq\frac13
 \quad\Longleftrightarrow\quad
 1-\frac1{2\theta}\leq R(1/3)=\frac13
 \quad\Longleftrightarrow\quad\theta\leq\frac34.
\]
If $\theta\geq1/2$, then $u_0=(2\theta-1)/3$ satisfies
\[
 G_\theta(u_0)=\frac1{2\theta}\geq\frac{u_0}{\theta}=R(u_0)
 \quad\Longrightarrow\quad u_*\geq u_0.
\]
Thus $v(u_*)\geq\theta$ for all $0<\theta\leq1$.

For $\theta\leq3/4$, write
$P_\theta=R(u_*)=G_\theta(u_*)$, so $u_*=(2\theta/3)(1-P_\theta)$ and
\[
 6\theta P_\theta^2-(3+10\theta)P_\theta+4\theta=0.
\]
The smaller root is
\begin{equation*}
 P_\theta=
 \frac{3+10\theta-\sqrt{9+60\theta+4\theta^2}}{12\theta}.
\end{equation*}
Here $u_*\leq1/3<2/5$. For $\theta<3/4$, $d_\theta$ lies in this branch,
and
\[
 R(d_\theta)>d_\theta=G_\theta(d_\theta)
 \quad\Longrightarrow\quad u_*<d_\theta,
 \quad P_\theta>d_\theta.
\]
At $\theta=3/4$, $P_\theta=d_\theta=1/3$.
For $\theta\geq3/4$, this branch contributes at most
$R(1/3)=1/3\leq d_\theta$.

For $v(u)\geq1$, $Q_\theta=u$ gives a contribution at most $d_\theta$,
attained at $u=d_\theta\geq1/3$ for $\theta\geq3/4$.
This proves \eqref{nat:eq:beta-one} for the head.

\smallskip\noindent\emph{Passage to the complete spectrum.}
The low degrees cost $O(r^s)$. By Proposition~\ref{nat:prop:tail},
\[
 \text{tail}\subset[0,CN_r^{-3}]\subset[0,Cr^{6/5}],\qquad
 \operatorname{diam}B(\text{tail}\cup\{0\})
 \leq C_\gamma r^{6\gamma/5}
\]
on the common H\"older event, for every $0<\gamma<1/2$.
The $r^\theta$-cover therefore costs
\begin{equation}\label{nat:eq:tail-cost}
 S^s_{r,\theta}(B(\text{tail}))
 \leq C_\gamma\left[
 r^{6\gamma/5+\theta(s-1)}+r^{\theta s}\right].
\end{equation}
For all $0<\theta\leq1$,
\begin{equation*}
 \left(1-\frac{3}{5\theta}\right)_+\leq d_\theta.
\end{equation*}
For a positive left side this is equivalent to $\theta\leq1$.
Above the head threshold,
\[
 s>d_\theta\geq\left(1-\frac{3}{5\theta}\right)_+
 \quad\Longrightarrow\quad
 \frac35+\theta(s-1)>0.
\]
Choose $\gamma<1/2$ so that $6\gamma/5+\theta(s-1)>0$.
Equation~\eqref{nat:eq:tail-cost} gives the all-scale upper bound;
inclusion of the head gives both lower bounds.
For each prescribed $a>1$, intersect the two simultaneous-$\theta$
events of Corollary~\ref{cor:relative-width-transfer} with the common
H\"older event.
\end{proof}

\begin{figure}[tbp]
\centering
\includegraphics[width=.76\linewidth]{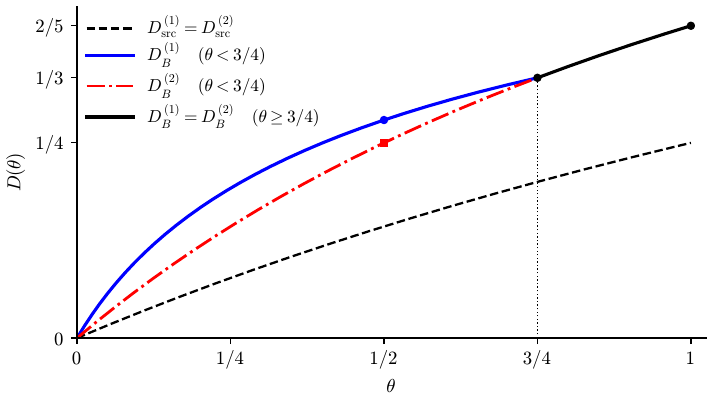}
\caption{Joint covering thresholds for two input-distribution families
of the fixed ReLU tangent kernel, plotted from the exact formulas in
Theorem~\ref{thm:natural-separation}. The source curves coincide.
The Brownian-image curves are distinct on $0<\theta<3/4$ and coincide
on $3/4\leq\theta\leq1$; the marked parameters are
$\theta=1/2,3/4,1$, with values in
\eqref{nat:eq:representative-values}.}
\label{nat:fig:phase}
\end{figure}

At three representative parameters, the image thresholds are
\begin{equation}\label{nat:eq:representative-values}
 \begin{array}{c|ccc}
 &\theta=1&\theta=3/4&\theta=1/2\\\hline
 D_B^{(1)}(\theta)&2/5&1/3&(4-\sqrt{10})/3\\[2pt]
 D_B^{(2)}(\theta)&2/5&1/3&1/4.
 \end{array}
\end{equation}

\part{Cantor sets with delayed contraction histories}
\label{part:cantor}
\section{Entropy formulas for Cantor sets and their random images}
\label{sec:tree-entropy}

We study the Cantor constructions of
Baker, Banaji, Feng, Lai and Xiong~\cite{BBFLX2025} through
cumulative digit frequencies and entropy.

\subsection{The construction, frequency profiles, and physical scales}

Fix a delay parameter $0<\beta<1$ and $M\ge100$, and put
\[
 \lambda=\log M,\qquad c=\log2,\qquad
 a=\lambda-\beta c>0,\qquad b=\lambda+\beta c.
\]
For a binary word $\omega=\omega_1\cdots\omega_n$, put
$S_k(\omega)=\sum_{j=1}^k\omega_j$ and
\begin{equation}\label{eq:tree-length}
 \ell_\omega=M^{-n}2^{S_{\lfloor\beta n\rfloor}(\omega)-\beta S_n(\omega)}.
\end{equation}
Starting from $[0,1]$, place intervals of these lengths at the two ends
of their parent intervals, and let $E_{\beta,M}$ be the resulting Cantor
set. Directly from \eqref{eq:tree-length},
\begin{equation}\label{eq:tree-ratios}
 (2M)^{-1}\le\ell_{\omega i}/\ell_\omega\le2/M,
 \qquad
 \operatorname{dist}(I_{\omega0},I_{\omega1})
 \ge(1-4/M)\ell_\omega.
\end{equation}

Let $h(z)=-z\log z-(1-z)\log(1-z)$, with $h(0)=h(1)=0$.
A digit-frequency profile is a measurable function $f:[0,T]\to[0,1]$,
identified up to equality almost everywhere. Define
\begin{align}
 F_f(t)&=\int_0^t f(v)\,dv,\qquad
 A_f(t)=\int_0^t h(f(v))\,dv,\label{eq:tree-entropy}\\
 L_f(t)&=\lambda t+c\{\beta F_f(t)-F_f(\beta t)\}.
 \label{eq:tree-scale}
\end{align}
For $N=-\log r$ and depth $n=\lfloor Nt\rfloor$, the logarithmic
source scale is approximated by $NL_f(t)$; write $u=HL_f(t)$ for
the normalized image scale. Lemma~\ref{lem:tree-frequency-types}
quantifies the approximation.
Moreover,
\[
 L_f'(t)=\lambda+c\beta\{f(t)-f(\beta t)\},\qquad
 a\le L_f'(t)\le b
\]
almost everywhere. For $H>0$,
take $T=2/(Ha)$ and let $t_f^H(u)$ solve $HL_f(t)=u$,
$0\le u\le1$. All these times lie in $[0,1/(Ha)]$, so any larger
horizon gives the same variational values. We use $H=1$ for the source
and $0<H<1$ for fractional Brownian images. Write
$A_f^H(u)=A_f(t_f^H(u))$ and set
\begin{equation}\label{eq:tree-payoff}
 G_{\theta,H}(f,u)=
 \begin{cases}
  1+\bigl(A_f^H(u)-u\bigr)/\theta,&0\le u\le\theta,\\[2pt]
  A_f^H(u)/u,&\theta\le u\le1.
 \end{cases}
\end{equation}
The two expressions agree at $u=\theta$, and $G_{\theta,H}(f,0)=1$.
Define
\begin{equation}\label{eq:tree-random-value}
 \mathcal D_{\theta,H}
 =\sup_{0\le f\le1}\ \min_{0\le u\le1}G_{\theta,H}(f,u).
\end{equation}

\subsection{Source and random-image variational formulas}

\begin{theorem}[Source entropy formula]\label{thm:tree-dimensions}
For every $0<\theta\le1$, the two intermediate dimensions agree and
\begin{equation}\label{eq:tree-source-value}
 \underline{\dim}_\theta E_{\beta,M}
 =\overline{\dim}_\theta E_{\beta,M}
 =d_\theta
 :=\max_{0\le f\le1}\ \min_{\theta\le u\le1}
       \frac{A_f^1(u)}{u}
 =\mathcal D_{\theta,1}.
\end{equation}
The maximum is attained by a digit-frequency profile.
\end{theorem}

\begin{theorem}[Scalar random-image entropy formula]
\label{thm:tree-random-dimensions}\label{thm:tree-random-images}
For a scalar fractional Brownian motion $B_H$, $0<H<1$, almost surely,
simultaneously for every $0<\theta\le1$,
\begin{equation}\label{eq:tree-random-dimensions}
 \underline{\dim}_\theta B_H(E_{\beta,M})
 =\overline{\dim}_\theta B_H(E_{\beta,M})
 =\mathcal D_{\theta,H}.
\end{equation}
The common probability-one event is for each fixed $E_{\beta,M}$ and $H$.
\end{theorem}

The branches in \eqref{eq:tree-payoff} cover prefix images respectively
by intervals of diameter $r^\theta$ and by a single admissible interval.

\begin{lemma}[Compactness and stability of frequency profiles]
\label{lem:tree-profile-stability}
Both variational maxima above are attained. Their values are continuous
in $(\theta,H)\in(0,1]\times(0,\infty)$, with the source value using
$H=1$. For fixed parameters the functionals are continuous for strong
$L^1$ convergence of profiles and upper semicontinuous for weak-*
convergence on a common bounded horizon. Their suprema can be
approximated by step profiles with values in
$[\varepsilon,1-\varepsilon]$ for some $\varepsilon>0$.
\end{lemma}

\begin{proof}
Work on a common horizon and compact ranges of positive $\theta,H$.
The bounds $a\le L_f'\le b$ and $0\le A_f'\le c$ imply
\[
 \operatorname{Lip}(A_f^H)\le c/(Ha),\qquad
 |t_f^{H'}(u)-t_f^H(u)|\le C|H'-H|.
\]
Thus $\min_uG_{\theta,H}(f,u)$ and
$J_\theta(f):=\min_{\theta\le u\le1}A_f^1(u)/u$ are equicontinuous
in their parameters, uniformly in $f$. Since $h$ is uniformly
continuous, strong $L^1$ convergence gives uniform convergence of
$F_f,A_f,L_f$ and hence convergence of both functionals, including
approximation by clipped step profiles.

If $f_j\overset{*}{\rightharpoonup}f$, testing against interval
indicators gives $F_{f_j}\to F_f$ uniformly by equi-Lipschitz
continuity, and hence $L_{f_j}\to L_f$ uniformly. Thus
\[
 \sup_{0\le u\le1}|t_{f_j}^H(u)-t_f^H(u)|
 \le a^{-1}\|L_{f_j}-L_f\|_\infty\longrightarrow0.
\]
Let $\mathcal P_m$ divide $[0,t]$ into $m$ equal intervals.
Jensen's inequality and $L^1$ convergence of cell averages give
\[
 \limsup_j A_{f_j}(t)
 \le\sum_{I\in\mathcal P_m}|I|
        h\left(|I|^{-1}\int_I f\right)
 \xrightarrow[m\to\infty]{} A_f(t).
\]
For $t_j\to t$, therefore,
\[
 \limsup_j A_{f_j}(t_j)
 \le\limsup_j\{A_{f_j}(t)+c|t_j-t|\}\le A_f(t).
\]
Apply this at the inverse times and use
$\limsup_j\inf_u g_j(u)\le\inf_u\limsup_jg_j(u)$.
Since $L^1[0,T]$ is separable, the weak-* compact set
$\{0\le f\le1\}\subset L^\infty[0,T]$ is metrizable.
Upper semicontinuity therefore gives attainment.
\end{proof}

\subsection{Type counting and covering upper bounds}

\begin{lemma}[Uniform estimates for frequency types]
\label{lem:tree-frequency-types}
Fix a horizon $T>0$ and an integer $K\ge1$, and put $\eta=T/K$.
For each sufficiently large $N$ (depending on $T,K$), divide the first
$\lfloor NT\rfloor$ digits into
$K$ consecutive blocks corresponding to the equal partition of $[0,T]$.
All blocks are then nonempty. A frequency record specifies the number
of ones in every block and
defines a step profile $f$ by its block frequencies. There are at most
$(NT+2)^K$ records. Uniformly over records, words in a record, and
$0\le t\le T$, with $n=\lfloor Nt\rfloor$,
\begin{align}
 -\log\ell_{\omega|n}
     &=NL_f(t)+O(N\eta+K),\label{eq:tree-type-length}\\
 \log\#\{\text{length-$n$ prefixes in the record}\}
     &\le NA_f(t)+O(N\eta+K\log(N+2)).
 \label{eq:tree-type-count}
\end{align}
The constants depend only on $\beta,M,T$, and are independent of the
record, prefix and observation scale $N$.
\end{lemma}

\begin{proof}
There are at most $NT+2$ counts per block. Complete blocks determine
both digit sums in \eqref{eq:tree-length}; the two incomplete blocks
at $n,\lfloor\beta n\rfloor$ and rounding contribute $O(N\eta+K)$.
This proves \eqref{eq:tree-type-length}. For prefix counts, use
$\binom{m}{k}\le e^{mh(k/m)}$ on complete blocks and the bound
$2^{N\eta+O(K)}$ on the final incomplete block, giving
\eqref{eq:tree-type-count}.
\end{proof}

Fix the exponent margin, then $\eta$, and finally let $N\to\infty$.
Thus $(NT+2)^K=e^{o(N)}$; covering depths may depend on the record.

\begin{proof}[Upper bounds in Theorems~\ref{thm:tree-dimensions}
and~\ref{thm:tree-random-dimensions}]
For the source upper bound in
\eqref{eq:tree-source-value}, suppose $\theta<1$ and fix $s>d_\theta$.
By time dilation and continuity, choose $0<\delta<(1-\theta)/2$ with
$d_{(\theta+\delta)/(1-\delta)}<s$. Every record then has a time $t$
with $L_f(t)\in[\theta+\delta,1-\delta]$ and
\[
 A_f(t)-sL_f(t)\le-\kappa,
\]
for a uniform $\kappa>0$. For small $\eta$, the corresponding
prefix cover is admissible by \eqref{eq:tree-type-length}, with
total cost bounded using \eqref{eq:tree-type-count} by
\[
 (NT+2)^K\exp(-\kappa N/2)\longrightarrow0.
\]
At $\theta=1$, use the box formula \eqref{eq:tree-box}
from \cite[Theorem~4.2]{BBFLX2025}.

For the image upper bound, suppose $g:[0,1]\to\mathbb R$ is
$\gamma$-H\"older and fix
$\mathcal D_{\theta,\gamma}<s<1$. For each record choose $u\in[0,1]$
where $G_{\theta,\gamma}(f,u)\le\mathcal D_{\theta,\gamma}$ and put
$t=t_f^\gamma(u)$. If $u\le\theta$, cover $g(E\cap I_{\omega|n})$
by intervals of diameter $r^\theta$; their total $s$-cost, after
summing over prefixes, is bounded by
\begin{equation}\label{eq:tree-coarse-cover}
 C\exp\{N[A_f(t)-u+\theta(1-s)+O(\eta)]
             +O(K\log(N+2))\}.
\end{equation}
For $u\ge\theta$, enlarge a single image interval to $r$ if needed
and subdivide at $r^\theta$. These corrections have cost
$\exp\{O(N\eta)+O(K\log(N+2))\}$, giving
\begin{equation}\label{eq:tree-middle-cover}
 C\exp\{N[A_f(t)-su+O(\eta)]
             +O(K\log(N+2))\}.
\end{equation}
Both leading exponents are at most
$-\theta(s-\mathcal D_{\theta,\gamma})<0$; all errors are
$O(\eta)+o(1)$ uniformly in $u$. Also $G(f,0)=1>s$ and
equi-Lipschitz continuity keep the chosen times uniformly positive.
Choose $\eta$ small and sum over records:
\[
 \overline{\dim}_\theta g(E_{\beta,M})\le\mathcal D_{\theta,\gamma}.
\]
The ambient bound covers $\mathcal D_{\theta,\gamma}=1$.
On the common H\"older event, let $\gamma\uparrow H$ and use
continuity to obtain \eqref{eq:tree-random-dimensions}, including $\theta=1$.

\end{proof}

\subsection{Uniformly typical measures and energy lower bounds}

\begin{lemma}[Uniformly typical cylinders]
\label{lem:tree-typical-cylinders}
Fix $T>0$, $\varepsilon>0$, and a step profile
$f:[0,T]\to[\varepsilon,1-\varepsilon]$ with finitely many steps.
For every real $N\ge1$, define on binary sequences the product measure
$\nu_N$ whose $j$th digit has success probability $p_j=f(j/N)$ up to
depth $\lfloor NT\rfloor$, and continue it arbitrarily thereafter.
For each fixed $\delta>0$ there is a set $\mathcal G_N$ such that
\begin{equation}\label{eq:tree-typical-probability}
 \nu_N(\mathcal G_N^c)\le C_f(NT+1)e^{-c_f\delta^2N}.
\end{equation}
For every $\omega\in\mathcal G_N$, simultaneously at all
$0\le n\le NT$,
\begin{align}
 -\log\ell_{\omega|n}&=NL_f(n/N)+O(\delta N+1),
 \label{eq:tree-typical-length}\\
 -\log\nu_N([\omega|n])&=NA_f(n/N)+O(\delta N+1).
 \label{eq:tree-typical-mass}
\end{align}
The constants may depend on the fixed profile and $\varepsilon,T,\beta,M$,
and are independent of $N,n$ and $\omega$. For sufficiently large $N$,
conditioning
on $\mathcal G_N$ and applying the coding map gives a probability
measure $\mu_N$ on $E_{\beta,M}$ such that every cylinder meeting its
support satisfies
\begin{equation}\label{eq:tree-conditioned-mass}
 \mu_N(I_{\omega|n})\le
 2\exp\{-NA_f(n/N)+O(\delta N+1)\}.
\end{equation}
\end{lemma}

\begin{proof}
Let $X_j$ be the independent Bernoulli digits and set
\[
 S_n=\sum_{j\le n}X_j,\qquad
 Y_n=-\sum_{j\le n}\{X_j\log p_j+(1-X_j)\log(1-p_j)\}.
\]
Their summands are uniformly bounded, so Hoeffding's
inequality~\cite{Hoeffding1963} gives, for $1\le n\le NT$,
\[
 \nu_N\{|S_n-\mathbb ES_n|>\delta N\}
 +\nu_N\{|Y_n-\mathbb EY_n|>\delta N\}
 \le4\exp(-c_\varepsilon\delta^2N^2/n)
 \le4\exp(-c_\varepsilon\delta^2N/T).
\]
A union bound over $1\le n\le NT$ proves
\eqref{eq:tree-typical-probability}; at $n=0$ both deviations vanish.
Since $f$ is a step profile,
\[
 \mathbb ES_n=NF_f(n/N)+O_f(1),\qquad
 \mathbb EY_n=NA_f(n/N)+O_f(1).
\]
Apply these at $n,\lfloor\beta n\rfloor$ for
\eqref{eq:tree-typical-length}--\eqref{eq:tree-typical-mass}.
Finally $\nu_N(\mathcal G_N)\ge1/2$ for large $N$, and
\[
 \mu_N(I_\omega)
 =\frac{\nu_N([\omega]\cap\mathcal G_N)}{\nu_N(\mathcal G_N)}
 \le2\nu_N([\omega])
\]
gives \eqref{eq:tree-conditioned-mass}.
\end{proof}

\begin{proof}[Lower bounds in Theorems~\ref{thm:tree-dimensions}
and~\ref{thm:tree-random-dimensions}]
Fix a clipped step profile as in Lemma~\ref{lem:tree-typical-cylinders}.
For the source lower bound, take $s<J_\theta(f)$. If
$r\le\rho\le r^\theta$, stop cylinders when their lengths first
become at most $\rho$. Their lengths lie in $[(2M)^{-1}\rho,\rho]$.
An interval of diameter $\rho$ meets at most $2+2M$ of them.
For a stopping cylinder meeting $\operatorname{supp}\mu_N$,
\eqref{eq:tree-typical-length} gives
$n/N=t_f^1(-\log\rho/N)+O(\delta)+o(1)$.
Thus $s<J_\theta(f)$ and \eqref{eq:tree-conditioned-mass} give,
after reducing $\delta$,
\[
 \mu_N(U)\le C e^{-\kappa N}|U|^s
 \quad(r\le|U|\le r^\theta)
\]
for some $\kappa>0$. Hence
$S^s_{r,\theta}(E_{\beta,M})\ge C^{-1}e^{\kappa N}$ for all large
real $N$; approximation by clipped profiles proves the source bound.

For random images, take $0<s<\min_uG_{\theta,H}(f,u)$ and set
$\bar\mu_N=(B_H)_*\mu_N$. Pairs with longest common prefix $\omega$
have distance comparable to $\ell_\omega$ by \eqref{eq:tree-ratios};
their total mass at depth $n$ is at most
\[
 \sum_{|\omega|=n}\mu_N(I_\omega)^2
 \le\max_{|\omega|=n}\mu_N(I_\omega).
\]
By \eqref{eq:vs-gaussian-three}, \eqref{eq:tree-typical-length} and
\eqref{eq:tree-conditioned-mass}, the expected level energy is at most
$\exp(O(\delta N))$ times, with $u=HL_f(n/N)$,
\begin{equation}\label{eq:tree-level-energy}
 \begin{cases}
 \exp\{-N[s+A_f(n/N)+\theta(1-s)-u]\},&u\le\theta,\\
 \exp\{-N[s+A_f(n/N)-su]\},&\theta\le u\le1.
 \end{cases}
\end{equation}
Sum up to $n_*=\lfloor Nt_f^H(1)\rfloor$; the remaining pairs
contribute at most $\exp\{-NA_f^H(1)+O(\delta N+1)\}$ since $K\le1$.

Put $\xi=\min_uG_{\theta,H}(f,u)-s>0$ and abbreviate
$G=G_{\theta,H}(f,u)$. The excess exponents in
\eqref{eq:tree-level-energy} are respectively
$\theta(G-s)\ge\theta\xi$ and $u(G-s)\ge\theta\xi$; the terminal
excess is $A_f^H(1)-s\ge\xi$. Choose $\delta$ small and then $N$ large
to absorb $O(\delta N)$ and the $O(N)$ levels. For some $\kappa>0$,
\begin{equation*}
 \mathbb E\iint K_{r,\theta}^s(x-y)
      \,d\bar\mu_N(x)d\bar\mu_N(y)
 \le C e^{-N(s+\kappa)}
\end{equation*}
For integer $N$, Markov's inequality gives
\[
 \mathbb P\{I^s_{r,\theta}(\bar\mu_N)>e^{-N(s+\kappa/2)}\}
 \le Ce^{-\kappa N/2}.
\]
Borel--Cantelli and \eqref{eq:cover-energy} yield divergent covering
cost along $r=e^{-N}$ almost surely. The fixed-set comparison
\eqref{eq:fixed-scale-comparison} extends this to all real scales,
with factor at most $e^s$. Take a countable approximating sequence
of clipped profiles and rational strict exponents, then intersect over
rational $\theta\in(0,1]$. Monotonicity and continuity of
$\mathcal D_{\theta,H}$ give equality simultaneously for all
$\theta\in(0,1]$.

\end{proof}

\subsection{Interpolation, basic regimes, and finite-depth stability}

\begin{proposition}[Strict interpolation and basic endpoint regimes]
\label{prop:tree-basic-regimes}
The source value satisfies
\begin{equation*}
 \frac{\log2}{\log M}<d_\theta<D(\log M)
 \quad(0<\theta<1),\qquad
 \lim_{\theta\downarrow0}d_\theta=\frac{\log2}{\log M},
\end{equation*}
and $d_1=D(\log M)$, where
\begin{equation}\label{eq:tree-box}
 D(\lambda)=
 \max_{p,q\in[0,1]}
 \frac{\beta h(p)+(1-\beta)h(q)}
      {\lambda+c\beta(1-\beta)(q-p)}.
\end{equation}
Moreover, the scalar random-image variational value obeys
\begin{equation*}
 \mathcal D_{\theta,H}=
 \begin{cases}
  1,&0<H\le c/\lambda,\\
  d_\theta/H,&H\ge c/(\lambda-\beta c).
 \end{cases}
\end{equation*}
In the intervening range the complete formula remains
\eqref{eq:tree-random-value}.
\end{proposition}

The endpoint dimensions were computed in
\cite[Theorem~4.2]{BBFLX2025}.

\begin{proof}

Since $(A_f-L_f)'\le c-a<0$, the coarse minimum in
\eqref{eq:tree-payoff} is at $u=\theta$ when $H=1$.
Thus $\mathcal D_{\theta,1}=d_\theta$.

Let $d_0=c/\lambda$. On the fixed horizon, choose a profile
$f(t)=1/2-\varepsilon t$, with $\varepsilon>0$ sufficiently small.
Then
\[
 L_f(t)=\lambda t-
       \frac{c\beta(1-\beta)}2\varepsilon t^2,
 \qquad
 A_f(t)\ge ct-\frac43\varepsilon^2t^3.
\]
Here $c-h(1/2-z)\le4z^2$ for $|z|\le1/4$.
For $\varepsilon T<\min\{1/4,3d_0c\beta(1-\beta)/8\}$,
\[
 A_f(t)-d_0L_f(t)
 \ge\frac{d_0c\beta(1-\beta)}2\varepsilon t^2
       -\frac43\varepsilon^2t^3>0\qquad(0<t\le T).
\]
Since $t_f^1(u)\ge\theta/b$ for $u\ge\theta$, we obtain $J_\theta(f)>d_0$.

For every profile and $t>0$, Jensen's inequality on $[0,\beta t]$
and $[\beta t,t]$ gives
\[
 \frac{A_f(t)}{L_f(t)}\le D(\lambda).
\]
The entropy identity $\max_x\{h(x)+tx\}=\log(1+e^t)$ makes
$D=D(\lambda)$ the positive zero of
\[
 \beta\log(1+2^{(1-\beta)D})
 +(1-\beta)\log(1+2^{-\beta D})-\lambda D.
\]
This function equals $c$ at zero and has derivative at most
$c\beta(1-\beta)-\lambda<0$. Its unique maximizing frequencies are
\[
 p_*=(1+2^{-(1-\beta)D})^{-1},\qquad
 q_*=(1+2^{\beta D})^{-1}<p_*.
\]
Equality at time $t$ requires $f=p_*$ almost everywhere on
$[0,\beta t]$ and $f=q_*$ almost everywhere on $[\beta t,t]$.
If $d_\theta=D$ with $\theta<1$, an attaining profile has equality
at $t_1=t_f^1(\theta)<t_2=t_f^1(1)$, forcing $f=p_*=q_*$ on the
positive-length interval $(\beta t_1,\min\{t_1,\beta t_2\})$.
This contradiction proves the strict upper bound.

To prove continuity at the Hausdorff endpoint, suppose
$J_\theta(f)\ge s>d_0$ and let a time $t$
satisfy $\theta\le L_f(t)\le1$. Then $A_f(t)\le ct$ implies, with
$m_f(t)=F_f(t)/t$,
\begin{equation}\label{eq:tree-average-increment}
 m_f(\beta t)-m_f(t)
 \ge\frac{\lambda-c/s}{\beta c}>0.
\end{equation}
Take $t_j=\beta^j/b$. For every $j$ with
$a\beta^j/b\ge\theta$, its physical scale lies in $[\theta,1]$.
Summing \eqref{eq:tree-average-increment} over these consecutive times
gives
\[
 k\frac{\lambda-c/s}{\beta c}\le1.
\]
As $\theta\downarrow0$, $k\to\infty$, excluding every fixed $s>d_0$.

For $H\le c/\lambda$, the profile $f=1/2$ gives
$A_f^H(u)=cu/(H\lambda)\ge u$ and $\min_uG=G(f,0)=1$.
For $Ha\ge c$, $A_f-HL_f$ is decreasing, so the coarse minimum is
at $u=\theta$. Time dilation of the middle branch gives
$\mathcal D_{\theta,H}=d_\theta/H$.
\end{proof}

For $\lambda\ge\log100$, write $d_\theta(\lambda)$,
$\mathcal D_{\theta,H}(\lambda)$, and $L_{\lambda,f}$ for the preceding
values and length functional with $\log M$ replaced by $\lambda$.

\begin{lemma}[Finite-depth stability of the covering estimates]
\label{lem:oscillating-prefix-comparison}
Let $F$ be a binary Cantor set whose child-to-parent length ratios lie
in a fixed interval $[a_0,b_0]\subset(0,1/2)$ and whose sibling gap
is at least a fixed positive multiple of the parent length.
Fix $\lambda\ge\log100$, $\beta\in(0,1)$,
$0<\theta\le1$, and $0<H<1$.
Let $r=e^{-N}$ and choose a sufficiently large fixed horizon $T$.
Write
\[
 \ell^\lambda_\omega=e^{-\lambda n}
       2^{S_{\lfloor\beta n\rfloor}(\omega)-\beta S_n(\omega)}.
\]
The following assertions hold as $N\to\infty$ along any sequence for which
the indicated comparisons are valid for every $n\le TN$ and every
word of length $n$, with $\varepsilon_N\to0$.
\begin{enumerate}
\item If $\ell^F_\omega\le e^{\varepsilon_N N}\ell^\lambda_\omega$,
then, for every $s>\mathcal D_{\theta,H}(\lambda)$, the
$s$-cost of a cover of $B_H(F)$ with diameters in $[r,r^\theta]$
tends to zero almost surely along these scales.
\item If $\ell^F_\omega\ge e^{-\varepsilon_N N}\ell^\lambda_\omega$,
then, for every $0<s<\mathcal D_{\theta,H}(\lambda)$, there are
deterministic probability measures $\mu_N$ on $F$ and constants
$C,\kappa>0$ such that
\begin{equation}\label{eq:oscillating-energy-margin}
 \mathbb E\iint K_{r,\theta}^s(B_H(x)-B_H(y))
             \,d\mu_N(x)d\mu_N(y)
 \le C e^{-N(s+\kappa)}.
\end{equation}
Here $K_{r,\theta}^s$ is the truncated kernel used in the proof of
Theorem~\ref{thm:tree-random-images}.
\end{enumerate}
The corresponding deterministic assertions hold for $F$ itself with
$H=1$ and $\mathcal D_{\theta,1}=d_\theta$; in the second assertion
the expectation is omitted.
\end{lemma}

\begin{proof}
Inverse times satisfy $t\le1/[H(\lambda-\beta c)]$.
For $\mathcal D_{\theta,H}(\lambda)<s<1$, choose $\gamma<H$ with
$\mathcal D_{\theta,\gamma}(\lambda)<s$, enlarge $T$ to contain
the $\gamma$-inverse times, and work on the probability-one
$\gamma$-H\"older event.
Larger $s$ follow by monotonicity; if $\mathcal D_{\theta,H}=1$,
interval covers suffice for $s>1$.

Lemma~\ref{lem:tree-frequency-types} has unchanged prefix counts,
while the first comparison gives image diameters at most
\[
 C\exp\{-N\gamma L_{\lambda,f}(t)
               +O(N\eta)+\gamma\varepsilon_NN+O(K)\}.
\]
Thus \eqref{eq:tree-coarse-cover}--\eqref{eq:tree-middle-cover}
gain only $O(\varepsilon_NN)$ in their exponents, including enlargement
to $r$ and subdivision at $r^\theta$. Fix a fine mesh, then take $N$
large to absorb this error and the polynomial record count into the
strict margin. For the source use $\gamma=1$ and the identity map.

For the lower bound, push the conditioned measure of
Lemma~\ref{lem:tree-typical-cylinders} onto $F$, using a clipped
profile with minimum payoff above $s$. Cylinder masses are unchanged;
for a longest common prefix $\omega$ of length at most $NT$,
\[
 |x-y|\ge C^{-1}e^{-\varepsilon_NN}\ell^\lambda_\omega.
\]
The Gaussian bound \eqref{eq:vs-gaussian-three} gains at most
$Ce^{H\varepsilon_NN}$ since $s<1$. Sum
\eqref{eq:tree-level-energy} up to $\lfloor Nt_f^H(1)\rfloor$,
then use the tail prefix mass and $K\le1$. Absorbing
$O(\varepsilon_NN)$ into the strict margin proves
\eqref{eq:oscillating-energy-margin}.

For the deterministic assertion observe that
\[
 K_{r,\theta}^s(x-y)
 \le\min\{1,r^s|x-y|^{-s},
                  r^{s+\theta(1-s)}|x-y|^{-1}\}.
\]
The same prefix calculation therefore applies with $H=1$.
\end{proof}

\section{Exact finite reductions and pressure certificates}
\label{sec:ifs-finite-pressure}

\subsection{Normalized windows and delay-orbit partitions}
\label{sec:tree-finite-exact}

\begin{lemma}[Normalization of the terminal time]
\label{lem:tree-time-normalization}
For $H>0$, let $t_*=t_f^H(1)$ and set $g(v)=f(t_*v)$ on $[0,1]$.
Then
\[
 A_g(v)=t_*^{-1}A_f(t_*v),\qquad
 L_g(v)=t_*^{-1}L_f(t_*v),\qquad HL_g(1)=t_*^{-1}.
\]
In particular, the original physical scale and entropy satisfy
\begin{equation}\label{eq:tree-normalized-physical-scale}
 u=\frac{L_g(v)}{L_g(1)},\qquad
 A_f^H(u)=\frac{A_g(v)}{HL_g(1)}.
\end{equation}
For the source window the normalized initial time obeys
$L_g(\tau)=\theta L_g(1)$ and $\tau\ge\theta a/b$.
\end{lemma}

\begin{proof}
Substitute $t=t_*v$; conversely take $t_*=(HL_g(1))^{-1}$.
The time bound follows from
$\theta a\le\theta L_g(1)=L_g(\tau)\le b\tau$.
\end{proof}

For $0<\tau\le1$, let
\begin{equation*}
 \mathcal B_\tau=\{0\}\cup
 \{\beta^j v:v\in\{\tau,1\},\ j\in\mathbb Z,
                  \ \beta\tau\le\beta^jv\le1\}.
\end{equation*}
Write the distinct points as $0=b_0<\cdots<b_N=1$, and let
$f=p_i\in[0,1]$ on $(b_{i-1},b_i)$.

For example, if $\beta=1/2$ and the normalized starting time is
$\tau=3/4$, then
\[
 \mathcal B_\tau=\{0,3/8,1/2,3/4,1\}.
\]
\begin{lemma}[Averaging on a delay-orbit partition]
\label{lem:tree-orbit-averaging}
Let a finite partition $\mathcal B$ of $[0,1]$ contain the endpoints
of a constraint window $[\alpha,1]$. Suppose that it contains
$\beta v$ for every $v\in\mathcal B\cap[\alpha,1]$, and contains
$v/\beta$ whenever $v\in\mathcal B$ and $v/\beta\in[\alpha,1]$.
Average a profile $f$ on each partition cell, obtaining $\bar f$.
At every partition point, $F_{\bar f}=F_f$ and $A_{\bar f}\ge A_f$.
At each constraint node, also $L_{\bar f}=L_f$.
For every $q\ge0$, $A_{\bar f}-qL_{\bar f}$ is affine between
successive constraint nodes. Consequently its nonnegativity, including
any fixed additive constant, is equivalent to nonnegativity at those nodes.
\end{lemma}

\begin{proof}
Averaging preserves $\int_I f$ on each cell, and Jensen gives
$|I|h(|I|^{-1}\int_I f)\ge\int_I h(f)$.
Closure under $v\mapsto\beta v$ preserves $L_f$ at constraint nodes;
closure under $v\mapsto v/\beta$ makes both $\bar f(t)$ and
$\bar f(\beta t)$ constant between them. Hence so is
\[
 (A_{\bar f}-qL_{\bar f})'(t)
 =h(\bar f(t))-q\lambda-qc\beta\{\bar f(t)-\bar f(\beta t)\}.
\]
\end{proof}

\subsection{Finite convex feasibility for source dimensions}

\begin{proposition}[Exact finite-frequency reduction]
\label{prop:tree-finite-exact}
For $0<\theta\le1$,
\begin{equation}\label{eq:tree-finite-source}
 d_\theta=
 \max_{\substack{0<\tau\le1,\ p_i\in[0,1]\\
                  L_f(\tau)=\theta L_f(1)}}
 \ \min_{b\in\mathcal B_\tau\cap[\tau,1]}
                    \frac{A_f(b)}{L_f(b)}.
\end{equation}
One may restrict $\tau$ to $\theta a/b\le\tau\le1$.
The number of frequencies is bounded by a constant
multiple of $1+\log(1/\theta)/|\log\beta|$, with constants depending
only on $\beta,M$.

For any fixed test exponent $s\ge0$, the condition $d_\theta\ge s$
is equivalent to the feasibility of at least one member of a finite
collection of convex problems. The variables are the time endpoint
$\tau$ and the digit masses in the partition cells. The constraints
consist of affine equalities and inequalities, together with concave
entropy inequalities of the form $A_f(b)-sL_f(b)\ge0$.
\end{proposition}

\begin{proof}
Normalize by Lemma~\ref{lem:tree-time-normalization}, so
$L_f(\tau)=\theta L_f(1)$ and $\tau\ge\theta a/b$.
Lemma~\ref{lem:tree-orbit-averaging} averages an attaining profile
without decreasing its value and extends nodal constraints across the
window, proving \eqref{eq:tree-finite-source}. Each orbit has
$O(1+|\log\tau|/|\log\beta|)$ points; $\tau\ge\theta a/b$
gives the frequency bound.

The finitely many candidates $\beta^j,\beta^j\tau$ give finitely
many affine ordering regions. On each, put $l_i=b_i-b_{i-1}$,
$z_i=l_ip_i$. Nodal lengths are affine in $(\tau,z)$; nodal entropies
are sums of concave perspectives $l_i h(z_i/l_i)$, extended by zero
at collapsed cells since
\[
 0\le l_i h(z_i/l_i)\le c l_i\longrightarrow0
 \qquad(0\le z_i\le l_i,\ l_i\downarrow0).
\]
Thus $0\le z_i\le l_i$, the physical equality and the entropy
inequalities define a convex set, including coincident boundary nodes.
\end{proof}

\subsection{Finite convex feasibility for random images}

Fix $0<s<1$ and put
\[
 \sigma_0=\theta(1-s),\qquad \sigma_1=\theta.
\]
Choose normalized times $0<\tau_0\le\tau_1\le1$, and define
\begin{equation*}
 \mathcal B_{\tau_0,\tau_1}
 =\{0\}\cup\{\beta^jv:v\in\{\tau_0,\tau_1,1\},\ j\in\mathbb Z,
                       \ \beta\tau_0\le\beta^jv\le1\}.
\end{equation*}
As before, frequencies are constant on its successive cells.

\begin{proposition}[Finite feasibility for random images]
\label{prop:tree-random-finite}
The inequality $\mathcal D_{\theta,H}\ge s$ holds if and only if
there exist times and frequencies as above satisfying
\begin{align}
 L_f(\tau_i)&=\sigma_iL_f(1),\qquad i=0,1,
                    \notag\\
 A_f(v)-HL_f(v)+H\theta(1-s)L_f(1)&\ge0,
       \qquad v\in\mathcal B_{\tau_0,\tau_1}\cap[\tau_0,\tau_1],
                    \label{eq:tree-finite-coarse}\\
 A_f(v)-HsL_f(v)&\ge0,
       \qquad v\in\mathcal B_{\tau_0,\tau_1}\cap[\tau_1,1].
                    \label{eq:tree-finite-middle}
\end{align}
One may impose $\tau_0\ge\theta(1-s)a/b$. For fixed
$(\theta,H,s)$ this gives a finite collection of convex feasibility
problems with
$O(1+\log(1/[\theta(1-s)])/|\log\beta|)$ frequencies per problem.
\end{proposition}

\begin{proof}
By attainment and \eqref{eq:tree-normalized-physical-scale},
$\mathcal D_{\theta,H}\ge s$ is equivalent to
\eqref{eq:tree-finite-coarse}--\eqref{eq:tree-finite-middle} on the
full intervals: for $u\le\sigma_0$,
$A_f^H(u)-u+\theta(1-s)\ge\sigma_0-u\ge0$.
Lemma~\ref{lem:tree-orbit-averaging} preserves the physical levels
and the coarse additive constant, so the nodal constraints are equivalent.
Now $\tau_0\ge\sigma_0a/b$ bounds the geometric powers; each affine
ordering region in $(\tau_0,\tau_1)$ has the same concave perspective
constraints as above.
\end{proof}

\subsection{Regional pressures and the role of boundary faces}
\label{sec:ifs-exact-pressure-duality}

Finite-dimensional separation gives the pressure dual below;
compare the minimax theorem of~\cite{Sion1958}.

Fix $0<\theta\le1$. Partition $[\theta a/b,1]$ into finitely many
closed intervals $I=[\alpha,\gamma]$ on which the membership and
ordering of the orbit grid are fixed. Boundary coincidences are
represented by cells of length zero. Denote the cell endpoints by
$b_i(\tau)$, their affine lengths by $l_i(\tau)$, and the window
times by $t_0(\tau)=\tau,\ldots,t_J(\tau)=1$. The times and their
delays are grid points. Put
\[
 \mathcal K_I=\{(\tau,z):\tau\in I,\quad
                              0\le z_i\le l_i(\tau)\}.
\]
This compact convex polytope parametrizes profiles with frequencies
$z_i/l_i(\tau)$ on nonempty cells. For $s\ge0$, put
\[
 \mathcal R(\tau,z)=L_f(\tau)-\theta L_f(1),\qquad
 g_j^s(\tau,z)=A_f(t_j)-sL_f(t_j),
\]
and define the regional margin
\begin{equation*}
 \mathcal V_I(s)=
 \max_{\substack{x\in\mathcal K_I\\\mathcal R(x)=0}}
                  \min_{0\le j\le J}g_j^s(x).
\end{equation*}
An empty physical slice has value $-\infty$.
The function $\mathcal R$ is affine, and every $g_j^s$ is continuous and
concave in the digit-mass variables and $\tau$.

For $w\in\Delta_J$ and $\eta\in\mathbb R$, put
\[
 \rho_j=sw_j-\eta\mathbf1_{j=0}
                    +\eta\theta\mathbf1_{j=J}.
\]
On the $i$th open cell let
\[
 W_i=\sum_{j:y\le t_j}w_j,\qquad
 K_i=c\left(\sum_{j:y\le\beta t_j}\rho_j
                       -\beta\sum_{j:y\le t_j}\rho_j\right),
\]
where $y$ lies in that cell. The coefficients are constant on each
ordering region. At a collapsing cell use their interior limits;
empty cells contribute zero. For singleton regions use the nonempty
cells directly. Define
\[
 \mathcal L(W,K)=
 \begin{cases}
 W\log(1+e^{K/W}),&W>0,\\
 \max\{0,K\},&W=0.
 \end{cases}
\]
Thus $\mathcal L(W,K)=\max_{0\le p\le1}\{Wh(p)+Kp\}$, including $W=0$.
Set
\begin{equation*}
 \mathcal P_{I,s}(\tau;w,\eta)
 =-\lambda\sum_{j=0}^J\rho_jt_j(\tau)
        +\sum_i l_i(\tau)\mathcal L(W_i,K_i).
\end{equation*}
For fixed $(s,w,\eta)$ this function is affine in $\tau$.

\begin{proposition}[Exact regional pressure duality]
\label{prop:ifs-pressure-duality}
For every ordering interval $I$ and $s\ge0$,
\begin{equation}\label{eq:ifs-exact-pressure-dual}
 \mathcal V_I(s)=
 \inf_{w\in\Delta_J,\ \eta\in\mathbb R}
 \max\{\mathcal P_{I,s}(\alpha;w,\eta),
                 \mathcal P_{I,s}(\gamma;w,\eta)\}.
\end{equation}
The identity includes the value $-\infty$ when the physical slice is
empty. If $\mathcal R$ takes both positive and negative values on
$\mathcal K_I$, the infimum is attained. More precisely, let
\[
 \delta=\min\{\max_{\mathcal K_I}\mathcal R,
                          -\min_{\mathcal K_I}\mathcal R\}>0.
\]
Then it suffices to minimize over
$|\eta|\le2(c+sb)/\delta$. The infimum is also attained when
$\mathcal R$ vanishes identically.
\end{proposition}

The proof is given in Appendix~\ref{app:ifs-phases}.

At a one-sided physical constraint, eliminate the forced digit masses
before taking the dual. This gives an attained formula on the boundary face.

\begin{proposition}[Pressure on a boundary face]
\label{prop:ifs-boundary-pressure}
Suppose the physical slice is nonempty and choose
$\sigma\in\{-1,1\}$ such that $\sigma \mathcal R\ge0$ on
$\mathcal K_I$. This includes $\mathcal R\equiv0$. Write
\[
 \sigma \mathcal R(\tau,z)=h_0(\tau)+\sum_i e_i z_i,
 \qquad
 m(\tau)=h_0(\tau)+\sum_i\min\{0,e_i\}l_i(\tau).
\]
Then $m$ is affine and nonnegative on $I$, and
$I_0:=\{\tau\in I:m(\tau)=0\}$ is either $I$ or a single
endpoint of $I$. The physical slice is exactly
\[
 \tau\in I_0,\qquad
 z_i=0\ (e_i>0),\qquad z_i=l_i(\tau)\ (e_i<0),\qquad
 0\le z_i\le l_i(\tau)\ (e_i=0).
\]
Let $K_i^0$ be $K_i$ at $\eta=0$, and define
\[
 Q_i(w)=
 \begin{cases}
 0,&e_i>0,\\
 K_i^0,&e_i<0,\\
 \mathcal L(W_i,K_i^0),&e_i=0.
 \end{cases}
\]
The reduced pressure
\[
 \mathcal P^0_{I,s}(\tau;w)
 =-\lambda s\sum_jw_jt_j(\tau)+\sum_i l_i(\tau)Q_i(w)
\]
satisfies the attained compact dual formula
\begin{equation}\label{eq:ifs-boundary-pressure-dual}
 \mathcal V_I(s)=\min_{w\in\Delta_J}
       \max_{\tau\in\partial I_0}\mathcal P^0_{I,s}(\tau;w).
\end{equation}
For a singleton interval its boundary means that singleton.
\end{proposition}

The proof is given in Appendix~\ref{app:ifs-phases}.

The coefficient of $z_i$ in $\mathcal R$ is
\[
 c\left(\beta\mathbf1_{y\le\tau}
       -\mathbf1_{y\le\beta\tau}-\theta\beta
       +\theta\mathbf1_{y\le\beta}\right),
\]
with constant term $\lambda(\tau-\theta)$. Its extrema on $\mathcal K_I$
are found by choosing each $z_i\in\{0,l_i\}$ according to this sign,
then evaluating the resulting affine function at $\alpha,\gamma$.
These tests distinguish empty, two-sided, and boundary slices.

\begin{corollary}[A complete finite pressure characterization]
\label{cor:ifs-global-pressure-root}
Let $\mathcal I_\theta$ be the finite family of ordering intervals
and put
\[
 \mathcal V_\theta(s)=\max_{I\in\mathcal I_\theta}\mathcal V_I(s).
\]
The number $d_\theta$ is the unique zero of
$\mathcal V_\theta$ on $[0,\infty)$. This function is continuous
and strictly decreasing. For $s_2>s_1\ge0$,
\begin{equation}\label{eq:ifs-pressure-monotonicity}
 -b(s_2-s_1)
 \le\mathcal V_\theta(s_2)-\mathcal V_\theta(s_1)
 \le-\frac{\theta a^2}{b}(s_2-s_1).
\end{equation}
Every upper bound $s\ge d_\theta$ therefore admits a finite
collection of nonpositive endpoint-pressure certificates: use
\eqref{eq:ifs-exact-pressure-dual} in each two-sided region and
\eqref{eq:ifs-boundary-pressure-dual} on every boundary face.
Empty physical regions are discarded.
\end{corollary}

The proof is given in Appendix~\ref{app:ifs-phases}.

For a primal maximizer $x_*$ and dual minimizer $(w,\eta)$ on a
two-sided region, equality in the dual bound gives
\[
 w_j>0\quad\Longrightarrow\quad
                  g_j^s(x_*)=\mathcal V_I(s).
\]
On every nonempty cell with $W_i>0$, the maximizing frequency is
\[
 p_i=\frac1{1+e^{-K_i/W_i}}.
\]
If $W_i=0$, it is one when $K_i>0$, zero when $K_i<0$, and
unrestricted when $K_i=0$. If $\tau_*$ lies in the interior of
$I$, the affine pressure is constant on $I$; if $\tau_*$ is an
endpoint, it maximizes the pressure.
Conversely, suppose $x_*$ is physically feasible, $g_j^s(x_*)\ge v$
for every $j$, and $w_j(g_j^s(x_*)-v)=0$ for every $j$.
If the entropy maximization conditions and the endpoint maximum
condition also hold, then $\mathcal V_I(s)=v$. The analogous
conditions apply on a boundary face with its forced frequencies.

\section{Endpoint asymptotics and analytic branches}
\label{sec:ifs-regularity}

\subsection{An analytic optimizer near the box endpoint}
\label{sec:ifs-analytic-optimizer}

Put $D=d_1$. The pressure equation is
\begin{equation}\label{eq:ifs-box-pressure}
 \lambda D=\beta P+(1-\beta)Q,\qquad
 P=\log(1+e^{c(1-\beta)D}),\quad
 Q=\log(1+e^{-c\beta D}).
\end{equation}
Put $\sigma(x)=(1+e^{-x})^{-1}$ and
\[
 p=\sigma(c(1-\beta)D),\qquad q=\sigma(-c\beta D).
\]
Define
\begin{equation*}
 \theta_{\mathrm{safe}}
 =\frac{\beta\lambda+c\beta^2(1-\beta)}
        {\lambda-c\beta(1-\beta)^2}\in(\beta,1).
\end{equation*}

\begin{lemma}[Localization of the initial time]
\label{lem:ifs-initial-time}
If $\theta\ge\theta_{\mathrm{safe}}$, every normalized profile with
$L_f(\tau)=\theta L_f(1)$ has $\tau\ge\beta$. Averaging on the four
intervals with endpoints $0,\beta\tau,\beta,\tau,1$ preserves the
window and does not decrease its minimum entropy-to-length ratio.
\end{lemma}

\begin{proof}
Normalize the terminal time to one and let $x,y,z_0$ be the averages
of $f$ on $[0,\beta^2]$, $[\beta^2,\beta]$, $[\beta,1]$. Then
\[
 \frac{L_f(\beta)}{L_f(1)}
 =\frac{\beta\lambda+c\beta^2(1-\beta)(y-x)}
 {\lambda+c\beta(1-\beta)(-\beta x-(1-\beta)y+z_0)}.
\]
The $x$ coefficients coincide and $L_f(\beta)<L_f(1)$, so the
fraction decreases in $x,z_0$ and increases in $y$. Hence
$L_f(\beta)/L_f(1)\le\theta_{\mathrm{safe}}\le\theta$ gives $\tau\ge\beta$.

For $\tau\ge\beta$, averaging on the four cells preserves
$L_f(\tau),L_f(1)$ and gives $A_{\bar f}(\tau)\ge A_f(\tau)$,
$A_{\bar f}(1)\ge A_f(1)$ by Jensen. Since $A_{\bar f},L_{\bar f}$
are affine on $[\tau,1]$ and $L_{\bar f}>0$,
\[
 \min_{[\tau,1]}\frac{A_{\bar f}}{L_{\bar f}}
 =\min_{t\in\{\tau,1\}}\frac{A_{\bar f}(t)}{L_f(t)}
 \ge\min_{[\tau,1]}\frac{A_f}{L_f}.
\]
\end{proof}

For $0\le w<1$ and real $k$, define
\begin{align}
 \mathcal C(k,w)={}&\lambda k-\beta P
 +\beta\log\bigl(1+e^{c((1-\beta)D-k)}\bigr)-Q\notag\\
 &+(1-w)\log\bigl(1+e^{-c\beta(D-k)/(1-w)}\bigr).
 \label{eq:ifs-scalar-pressure}
\end{align}
At $w=1$, the last term is interpreted by its continuous perspective
extension, namely $\max\{0,-c\beta(D-k)\}$.

\begin{lemma}[The scalar pressure curve]
\label{lem:ifs-scalar-pressure}
For every $w\in[0,1]$, the equation $\mathcal C(k,w)=0$ has exactly one
solution $k=k(w)$. It satisfies $k(0)=0$, $k(1)=D$, and $k$ is
strictly increasing and strictly concave. On $[0,1)$ it is real
analytic.
\end{lemma}

\begin{proof}
For $w<1$, write
\[
 z=\sigma(c((1-\beta)D-k)),\qquad
 v=\sigma(-c\beta(D-k)/(1-w)).
\]
Differentiating \eqref{eq:ifs-scalar-pressure} gives
\begin{equation*}
 \mathcal C_k=\lambda+c\beta(v-z)\ge a>0,
 \qquad \mathcal C_w=-h(v)<0.
\end{equation*}
Also $\mathcal C(0,0)=\mathcal C(D,1)=0$ and
$\mathcal C(0,w)<0<\mathcal C(D,w)=(1-w)c$ for $0<w<1$.
Thus $\mathcal C_k\ge a$ gives a unique root in $[0,D]$, analytic
for $w<1$, with
\begin{equation}\label{eq:ifs-k-prime}
 k'(w)=\frac{h(v)}{\lambda+c\beta(v-z)}>0.
\end{equation}

Differentiating once more gives
\begin{align*}
 -\mathcal C_k k''(w)
 ={}&\beta c^2z(1-z)(k'(w))^2\\
 &+\frac{(c\beta)^2v(1-v)}{1-w}
       \left(k'(w)-\frac{D-k(w)}{1-w}\right)^2>0.
\end{align*}
Hence $k''<0$. The bound $0<k'\le c/a$ extends $k$ continuously
to $[0,1]$, with strict concavity and
\[
 k(w)>Dw\quad(0<w<1),\qquad k'(0)>D>k'(1-).
\]
\end{proof}

\begin{theorem}[Scalar formula and optimizer near $\theta=1$]
\label{thm:ifs-analytic-optimizer}
For $\theta_{\mathrm{safe}}\le\theta<1$, set $\delta=1-\theta$.
There is a unique $w=w_\delta\in(0,1)$ solving
\begin{equation}\label{eq:ifs-scalar-stationarity}
 D-\delta k(w)-(1-\delta w)k'(w)=0.
\end{equation}
The intermediate dimension is then
\begin{equation*}
 d_\theta=k'(w_\delta)
 =\frac{D-\delta k(w_\delta)}{1-\delta w_\delta}
 =\min_{0\le w\le1}\frac{D-\delta k(w)}{1-\delta w}.
\end{equation*}
In particular, $d_\theta$ is real analytic, strictly increasing, and
strictly concave on $(\theta_{\mathrm{safe}},1)$.

On the normalized terminal interval $[0,1]$, the maximizing profile
is unique up to null sets and equals
\begin{equation}\label{eq:ifs-four-frequencies}
 \begin{gathered}
 f(t)=
 \begin{cases}
  p,&0<t<\beta\tau,\\
  z,&\beta\tau<t<\beta,\\
  q,&\beta<t<\tau,\\
  v,&\tau<t<1,
 \end{cases}\\[3pt]
 \begin{aligned}
 z&=\sigma(c((1-\beta)D-k(w_\delta))),\\
 v&=\sigma\!\left(-\frac{c\beta(D-k(w_\delta))}{1-w_\delta}\right).
 \end{aligned}
 \end{gathered}
\end{equation}
The constraint $L_f(\tau)=\theta L_f(1)$ uniquely determines
$\tau\in(\beta,1)$. Put
\[
 L_B=\lambda+c\beta(1-\beta)(q-p),\qquad
 R=c\beta\{(1-\beta)(p-z)+v-q\},\qquad
 \Lambda=\lambda+c\beta(v-z).
\]
Then
\begin{equation}\label{eq:ifs-optimal-start}
 \tau=1-\frac{\delta L_B}{\Lambda-\delta R}.
\end{equation}
The optimizer satisfies equality at every scale in its window:
\begin{equation}\label{eq:ifs-equalization}
 A_f(t)=d_\theta L_f(t)\qquad(\tau\le t\le1).
\end{equation}

Let $w_0\in(0,1)$ be the unique solution of $k'(w_0)=D$.
\begin{equation}\label{eq:ifs-endpoint-expansion}
 d_\theta
 =D-\bigl(k(w_0)-Dw_0\bigr)(1-\theta)
      +O((1-\theta)^2),\qquad k(w_0)-Dw_0>0.
\end{equation}
\end{theorem}

\begin{proof}
By Lemma~\ref{lem:ifs-initial-time}, every feasible profile has
$\tau\ge\beta$. For $0\le\delta<1$, define
$F_\delta(w)=D-\delta k(w)-(1-\delta w)k'(w)$.
Since
\[
 F_\delta'(w)=-(1-\delta w)k''(w)>0,
\]
and $F_\delta(0)<0<F_\delta(1-)$, the zero $w_\delta$ is unique.
For $\delta>0$,
\[
 \frac{d}{dw}\frac{D-\delta k(w)}{1-\delta w}
 =\frac{\delta F_\delta(w)}{(1-\delta w)^2},
\]
so it is the unique minimizer. Put $s=k'(w_\delta)$, $k=k(w_\delta)$.

For $\delta>0$, set $\eta=(s-D)/\delta$. Stationarity gives
\[
 k=sw_\delta-\eta,\qquad
 D-k=s(1-w_\delta)+\eta\theta.
\]
For an arbitrary profile whose initial time satisfies $\tau\ge\beta$,
consider the weighted endpoint expression
\begin{align*}
 &w_\delta\{A_f(\tau)-sL_f(\tau)\}
 +(1-w_\delta)\{A_f(1)-sL_f(1)\}
 +\eta\{L_f(\tau)-\theta L_f(1)\}\\
 &\hspace{1cm}=
 w_\delta A_f(\tau)+(1-w_\delta)A_f(1)
       -kL_f(\tau)-(D-k)L_f(1).
\end{align*}
Its coefficients on the four cells are:
\begin{center}
\small
\begin{tabular}{@{}cccc@{}}
\toprule
Cell & Length & Entropy weight & Frequency coefficient\\
\midrule
$(0,\beta\tau)$ & $\beta\tau$ & $1$ & $c(1-\beta)D$\\
$(\beta\tau,\beta)$ & $\beta(1-\tau)$ & $1$
 & $c((1-\beta)D-k)$\\
$(\beta,\tau)$ & $\tau-\beta$ & $1$ & $-c\beta D$\\
$(\tau,1)$ & $1-\tau$ & $1-w_\delta$ & $-c\beta(D-k)$\\
\bottomrule
\end{tabular}
\end{center}
Maximize cellwise using
\[
 h(x)+tx\le\log(1+e^t),
 \qquad\text{with equality precisely at }x=\sigma(t).
\]
The resulting upper bound is
\begin{align*}
 &-\lambda\{k\tau+D-k\}+\beta\tau P
 +\beta(1-\tau)\log(1+e^{c((1-\beta)D-k)})
 +(\tau-\beta)Q\\
 &\hspace{2cm}
 +(1-\tau)(1-w_\delta)
       \log(1+e^{-c\beta(D-k)/(1-w_\delta)})\\
 &\hspace{1cm}=(1-\tau)\mathcal C(k,w_\delta)=0,
\end{align*}
by \eqref{eq:ifs-box-pressure} and $\mathcal C(k,w_\delta)=0$.
For a feasible profile the residual vanishes; the positive endpoint
weights therefore give $d_\theta\le s$.

For attainment, use \eqref{eq:ifs-four-frequencies} with free
$\tau\in[\beta,1]$. Equality holds in the certificate, and on the window
\[
 A_f'(t)=h(v),\qquad L_f'(t)=\lambda+c\beta(v-z)=\Lambda.
\]
Thus $(A_f-sL_f)'=0$ by \eqref{eq:ifs-k-prime}.

For this profile, $L_f(1)=L_B+(1-\tau)R$ and
$L_f(1)-L_f(\tau)=(1-\tau)\Lambda$.
Consequently
\[
 L_f(\tau)-\theta L_f(1)
 =\delta L_B-(1-\tau)(\Lambda-\delta R),
\]
and \eqref{eq:ifs-optimal-start} follows from
\[
 \Lambda-\delta R\ge\Lambda-R
 =\lambda+c\beta\{q-(1-\beta)p-\beta z\}\ge a.
\]
Here $R>0$, since $k(w)>Dw$ gives $z<p$ and $v>q$.
The residual is positive at $1$ and negative at $\beta$: the interior
frequencies cannot attain the localization maximum $(0,1,0)$, even
when $\theta=\theta_{\mathrm{safe}}$. Thus $\beta<\tau<1$.
The constant deficit has weighted sum zero, proving
\eqref{eq:ifs-equalization} and $d_\theta\ge s$.

For any maximizer, nonnegative endpoint deficits force equality in
the certificate. Strict entropy concavity fixes all four frequencies;
the physical constraint then fixes $\tau$.

The implicit function theorem gives analytic $w_\delta$, with
\[
 \frac{dw_\delta}{d\delta}
 =\frac{k(w_\delta)-w_\delta k'(w_\delta)}
       {-(1-\delta w_\delta)k''(w_\delta)}>0.
\]
Set $s(\delta)=d_{1-\delta}$ and
$T_\delta=k(w_\delta)-w_\delta s(\delta)>0$.
Since $D=s+\delta T_\delta$ and $T_\delta'=-w_\delta s'$, we obtain
\[
 s'(\delta)=-\frac{T_\delta}{1-\delta w_\delta}<0,\qquad
 s''(\delta)=
 -\frac{T_\delta\{2w_\delta+\delta w_\delta'\}}
        {(1-\delta w_\delta)^2}<0.
\]
These signs give the asserted monotonicity and concavity.
At $\delta=0$, $s(0)=D$ and $w_\delta=w_0$, giving
\eqref{eq:ifs-endpoint-expansion}.
\end{proof}

\subsection{Entropy localization and a larger validity interval}
\label{sec:ifs-entropy-localization}

Put $d_0=c/\lambda$ and $D=d_1$. For the pressure curve in
Lemma~\ref{lem:ifs-scalar-pressure}, let $w_\theta\in(0,1)$ solve
\[
 D-(1-\theta)k(w_\theta)
   -(1-(1-\theta)w_\theta)k'(w_\theta)=0,
\]
and define
\begin{equation*}
 \widehat d_\theta=k'(w_\theta)
 =\frac{D-(1-\theta)k(w_\theta)}
        {1-(1-\theta)w_\theta}.
\end{equation*}
This continuation is real analytic and strictly increasing on $(0,1)$.

Put $B=\beta(1-\beta)$ and
\begin{equation}\label{eq:ifs-entropy-threshold}
 \epsilon=\frac{c^3B}{4\lambda^2},\qquad
 \theta_{\mathrm{ent}}
 =\frac{\beta(1-\epsilon)
          +\epsilon\sqrt{\beta(1+\beta-2\epsilon)}}{1-2\epsilon},
 \qquad
 \theta_{\mathrm{loc}}
 =\min\{\theta_{\mathrm{safe}},\theta_{\mathrm{ent}}\}.
\end{equation}
Both sufficient thresholds belong to $(\beta,1)$.

\begin{proposition}[A larger interval for the analytic formula]
\label{prop:ifs-entropy-localization}
The scalar dimension formula, unique four-frequency optimizer on the
normalized terminal interval $[0,1]$, and
strict increase and strict concavity in
Theorem~\ref{thm:ifs-analytic-optimizer} remain valid for
$\theta_{\mathrm{loc}}\le\theta<1$.
\end{proposition}

The proof is given in Appendix~\ref{app:ifs-phases}.

For $\beta=1/2$ and $M=100$, the thresholds are approximately
$\theta_{\mathrm{safe}}=0.52876$ and
$\theta_{\mathrm{ent}}=0.50134$, giving $\theta_{\mathrm{loc}}=0.50134$.

\subsection{Quantitative convergence at the Hausdorff endpoint}
\label{sec:tree-endpoint-rate}

Put $d_0=c/\lambda$. The endpoint rate below has the same order as
for Bedford--McMullen carpets~\cite[Corollary~2.3(iv)]{BanajiKolossvary2024}.

\begin{theorem}[Endpoint rate and saturation]
\label{thm:tree-endpoint-rate}
There exist $C_-,C_+,\theta_0>0$, depending only on $\beta$ and $M$,
such that
\begin{equation}\label{eq:tree-endpoint-two-sided}
 \frac{C_-}{(\log(1/\theta))^2}
 \leq d_\theta-d_0
 \leq\frac{C_+}{(\log(1/\theta))^2},
 \qquad 0<\theta<\theta_0.
\end{equation}
One may take
\[
 C_-=
 \frac{c^4\beta^2(\log(1/\beta))^2}{256\lambda^3}
\]
after decreasing $\theta_0$.

For every $0<\theta\leq1$ and $0<H<1$,
\begin{equation}\label{eq:tree-exact-saturation}
 \mathcal D_{\theta,H}=1
 \quad\Longleftrightarrow\quad H\leq d_0.
\end{equation}
For every fixed $H>d_0$, there is $\theta_H>0$ such that
\begin{equation*}
 \frac{C_-}{H(\log(1/\theta))^2}
 \leq\mathcal D_{\theta,H}-\frac{d_0}{H}
 \leq\frac{C_+}{H(\log(1/\theta))^2},
 \qquad 0<\theta<\theta_H.
\end{equation*}
\end{theorem}

The saturation criterion also follows from
\cite[Corollary~3.5]{AngeliniMolter2025} and continuity of $d_\theta$ at zero.

\begin{lemma}[Length of a quadratic drift trajectory]
\label{lem:tree-quadratic-drift}
Fix $\kappa>0$. There is $C_\kappa>0$ such that, if
$0<\alpha\leq1/16$ and $z_0,\ldots,z_k\in[-1/2,1/2]$ satisfy
\[
 z_{j+1}-z_j\geq\kappa(\alpha+2z_j^2),
 \qquad 0\leq j<k,
\]
then $k\leq C_\kappa/\sqrt\alpha$.
\end{lemma}

\begin{proof}
Put $v=\sqrt\alpha$. Monotonicity bounds the number of starting
indices in $[-v,v]$ by $1+2/(\kappa v)$. Each signed band
$2^iv\le|z|\le2^{i+1}v$, truncated at $1/2$, has length at most
$2^iv$ and increments at least $2\kappa4^iv^2$, hence at most
$1+(2\kappa2^iv)^{-1}$ starting indices. Thus
\[
 k\leq C\left(1+\log(1/v)+\frac1{\kappa v}
                         \sum_{i\geq0}2^{-i}\right)
 \leq\frac{C_\kappa}{v}.
\]
\end{proof}

\begin{proof}[Proof of Theorem~\ref{thm:tree-endpoint-rate}]
\emph{Upper endpoint bound.}
Let $f$ attain the source variational maximum, and set
$d=d_\theta>d_0$, $\varepsilon=d-d_0$ and
$m_f(t)=F_f(t)/t$ for $t>0$.
For each $t$ with $\theta\leq L_f(t)\leq1$,
Jensen's inequality and the variational constraint give
\[
 h(m_f(t))\geq A_f(t)/t
 \geq d\{\lambda+c\beta(m_f(t)-m_f(\beta t))\}.
\]
Since $h(1/2+z)\le c-2z^2$, putting $z(t)=m_f(t)-1/2$ gives
\begin{equation*}
 z(\beta t)-z(t)
 \geq\frac{\lambda\varepsilon+2z(t)^2}{d c\beta}.
\end{equation*}

Put $t_j=\beta^j/b$. Since $at_j\leq L_f(t_j)\leq bt_j$,
the recurrence holds for all $0\leq j<k_\theta$, where
\[
 k_\theta=
 1+\left\lfloor
      \frac{\log(a/(b\theta))}{\log(1/\beta)}
     \right\rfloor
 \sim\log(1/\theta)
\]
for sufficiently small $\theta$.
For small $\theta$, Proposition~\ref{prop:tree-basic-regimes} gives
$\alpha=\lambda\varepsilon\le1/16$. Apply
Lemma~\ref{lem:tree-quadratic-drift} with
$\kappa=(D(\lambda)c\beta)^{-1}\le(dc\beta)^{-1}$:
\[
 k_\theta\le\frac{C_\kappa}{\sqrt{\lambda\varepsilon}}
 \quad\Longrightarrow\quad
 \varepsilon\le\frac{C_\kappa^2}{\lambda k_\theta^2}.
\]

\emph{Lower endpoint bound.}
Set $L=\log(1/\theta)$, $\ell=\log(1/\beta)$, and $T=2/a$.
Choose the fixed amplitude
\[
 q=\frac{c^2\beta\ell}{64\lambda}>0,
 \qquad t_c=T\theta^2.
\]
On $0<t\leq T$, define the clipped logarithmic profile
\begin{equation*}
 f_\theta(t)=\frac12+
 \frac{q}{L^2}\min\{\log(T/t),2L\}.
\end{equation*}
Extend continuously at zero. For large $L$, $1/2\le f_\theta\le3/4$;
$c-h(1/2+z)\le4z^2$ then gives
\begin{equation}\label{eq:tree-log-entropy}
 A_{f_\theta}(t)\geq
 \left(c-\frac{16q^2}{L^2}\right)t
 \qquad(0\leq t\leq T).
\end{equation}
For $t\geq t_c$, direct integration yields
\[
 m_{f_\theta}(t)=\frac12+
 \frac{q}{L^2}\left(\log(T/t)+1-\frac{t_c}{t}\right).
\]
Consequently, if $\beta t\geq t_c$,
\begin{equation}\label{eq:tree-log-drift}
 m_{f_\theta}(\beta t)-m_{f_\theta}(t)
 =\frac{q}{L^2}
   \left(\ell-\frac{(1-\beta)t_c}{\beta t}\right).
\end{equation}
On the physical window, $t\ge\theta/b$ gives
$t_c/t\le Tb\theta=o(1)$. Thus \eqref{eq:tree-log-drift} applies
and is at least $q\ell/(2L^2)$ for small $\theta$, so
\[
 L_{f_\theta}(t)
 \leq\left(\lambda-\frac{c\beta q\ell}{2L^2}\right)t.
\]
By \eqref{eq:tree-log-entropy},
\begin{align*}
 J_\theta(f_\theta)-\frac c\lambda
 &\geq
 \frac{c-16q^2/L^2}
      {\lambda-c\beta q\ell/(2L^2)}-\frac c\lambda\\
 &\geq\frac{c^2\beta q\ell/2-16\lambda q^2}
                 {\lambda^2L^2}
 =\frac{c^4\beta^2\ell^2}{256\lambda^3L^2}.
\end{align*}
Here the denominator is positive for small $\theta$.

\emph{The saturation threshold.}
The constant profile $f=1/2$ proves
$\mathcal D_{\theta,H}=1$ when $H\leq d_0$.
Conversely, attainment gives $A_f(t)\ge HL_f(t)$ for
$0<t\le t_f^H(1)$. If $H>d_0$, then $A_f(t)\le ct$ implies
\[
 1\ge m_f(\beta^nt)-m_f(t)
 \ge n\frac{\lambda-c/H}{c\beta}\longrightarrow\infty,
\]
a contradiction. This proves \eqref{eq:tree-exact-saturation}.

\emph{Random-image endpoint bound.}
For $H>0$ and $g(t)=f(t/H)$, we have
$L_g(t)=HL_f(t/H)$ and $A_g(t)=HA_f(t/H)$.
Thus the middle branch gives
\begin{equation}\label{eq:tree-random-source-upper}
 \mathcal D_{\theta,H}\leq\min\{1,d_\theta/H\}.
\end{equation}
Apply \eqref{eq:tree-endpoint-two-sided} for the upper estimate.

For the lower estimate, fix $H>d_0$ and use
$g_\theta(t)=f_\theta(Ht)$ on the random-image horizon
$[0,2/(Ha)]$. Then
$A_{g_\theta}^H(u)=A_{f_\theta}^1(u)/H$.
Since $-q\ell/L^2\le f_\theta(t)-f_\theta(\beta t)\le0$,
\begin{align*}
 \frac{d}{dt}\{A_{g_\theta}(t)-HL_{g_\theta}(t)\}
 &\leq c-H\lambda+\frac{Hc\beta q\ell}{L^2}<0
\end{align*}
almost everywhere for small $\theta$, since $H\lambda>c$.
The coarse minimum is therefore at $u=\theta$, and
\[
 \mathcal D_{\theta,H}
 \geq\frac1H J_\theta(f_\theta)
 \geq\frac{d_0}{H}+\frac{C_-}{HL^2}.
\]
\end{proof}

\begin{remark}
The quadratic recurrence acts on cumulative digit frequencies;
the clipped logarithmic profile gives the matching endpoint order.
\end{remark}

\subsection{Local finiteness of the analytic branches}
\label{sec:tree-tame-regularity}

We apply analytic cell decomposition in $\mathbb R_{\mathrm{an},\exp}$
\cite{vandenDriesMiller1994} to the finite entropy constraints.
This structure allows real parameters, logarithms and real quantification.
For general attainable dimension curves see~\cite{BanajiRutar2022}.

\begin{proposition}[Locally finite analytic decomposition]
\label{prop:ifs-local-analyticity}
Fix $\beta\in(0,1)$ and $M\ge100$. For every $\varepsilon>0$, the
restriction of $\theta\mapsto d_\theta$ to $[\varepsilon,1]$ is
definable in $\mathbb R_{\mathrm{an},\exp}$. It is real analytic on
each open interval of a finite partition of $[\varepsilon,1]$.

The same statements hold for
$\theta\mapsto\mathcal D_{\theta,H}$ for every fixed $H\in(0,1)$.
Consequently, for either curve the set of nonanalytic points in $(0,1)$
is locally finite, and its only possible accumulation point in $[0,1]$
is $0$.
\end{proposition}

\begin{proof}
Fix $0<\varepsilon<1$. The bound $\tau\ge\varepsilon a/b$ in
Proposition~\ref{prop:tree-finite-exact} restricts the grid points
$\beta^j,\beta^j\tau$ to a fixed finite set of exponents.
On each ordering region the cell lengths are affine, and the entropy
contributions are definable continuous perspectives
\[
 \Psi(l,z)=l h(z/l)\quad(l>0,\ 0\le z\le l),
 \qquad \Psi(0,0)=0,
\]
with $h(0)=h(1)=0$. Thus all nodal constraints are definable,
including coincident endpoints. Projection in $(\tau,z)$ and finite
union give the definable set
\[
 \mathcal F=\{(\theta,s):\varepsilon\le\theta\le1,
                         \ 0\le s\le1,\ d_\theta\ge s\}.
\]
By attainment its upper fibre endpoints form the graph:
\[
 (\theta,s)\in\mathcal F,
 \qquad
 \text{there is no }t\in(s,1]\text{ with }(\theta,t)\in\mathcal F.
\]
Analytic cell decomposition gives the claimed finite partition.

For $H\le d_0$, \eqref{eq:tree-exact-saturation} gives the constant
curve $\mathcal D_{\theta,H}=1$. For $H>d_0$, continuity and strict
saturation give
\[
 \sup_{\varepsilon\le\theta\le1}
                  \mathcal D_{\theta,H}<S<1
\]
for some $S$. For $0<s\le S$, Proposition~\ref{prop:tree-random-finite}
gives
\[
 \tau_0\ge\theta(1-s)a/b
          \ge\varepsilon(1-S)a/b>0.
\]
The grid again uses finitely many powers. Projection in
$(\tau_0,\tau_1,z)$ and upper fibre endpoints give the graph as above,
using positivity and attainment. Analytic cell decomposition on each
$[\varepsilon,1]$ proves both assertions.
\end{proof}

\subsection{Existence of a nonanalytic transition}

\begin{proposition}[Existence of a nonanalytic transition]
\label{prop:ifs-genuine-transition}
Let $q=\sigma(-c\beta D)$ be the lower-frequency box optimizer.
Then
\begin{equation}\label{eq:ifs-continued-limit}
 \lim_{\theta\downarrow0}\widehat d_\theta=s_*,
 \qquad 0<s_*<d_0,
\end{equation}
where $s_*$ is the unique positive root of
\begin{equation}\label{eq:ifs-limit-pressure}
 \log(1+e^{-c\beta s})+c\beta s q-\lambda s=0.
\end{equation}
Consequently there is a unique $\theta_{\mathrm{cross}}\in(0,1)$ with
$\widehat d_{\theta_{\mathrm{cross}}}=d_0$.
There exists a parameter $\theta_{\mathrm{tr}}$ satisfying
\begin{equation}\label{eq:ifs-transition-bracket}
 \theta_{\mathrm{cross}}<\theta_{\mathrm{tr}}\le\theta_{\mathrm{loc}}
\end{equation}
at which $d_\theta$ is nonanalytic. One such parameter is the first
departure from the continued branch:
\begin{equation}\label{eq:ifs-first-departure}
 \theta_{\mathrm{tr}}
 =\inf\{a\in(0,1):d_\theta=\widehat d_\theta
                          \text{ for all }a\le\theta<1\}.
\end{equation}
\end{proposition}

The proof is in Appendix~\ref{app:ifs-phases};
Section~\ref{sec:ifs-first-transition} gives the large-$M$ separation
from the continued profile collision.

\section{The first transition: relative entropy and sharp separation}
\label{sec:ifs-first-transition}

\subsection{Collision and the quantitative transition theorem}

Retain $\lambda=\log M$, $c=\log2$, and $D=d_1$, and put
\[
 \varepsilon_M=\frac{c}{\lambda},\qquad G(y)=\log(1+e^y).
\]
By Theorem~\ref{thm:tree-dimensions} and
Lemma~\ref{lem:tree-time-normalization}, the normalized source formula is
\begin{equation}\label{eq:ifs-transition-variational}
 d_\theta=
 \max_{\substack{f:[0,1]\to[0,1]\ \mathrm{measurable},\ 0<\tau<1\\
                  L_f(\tau)=\theta L_f(1)}}
       \min_{\tau\le t\le1}\frac{A_f(t)}{L_f(t)},
 \qquad 0<\theta<1.
\end{equation}
Here $A_f,L_f$ are given by
\eqref{eq:tree-entropy}--\eqref{eq:tree-scale};
$L_f'\ge\lambda-c\beta>0$ makes $\tau$ unique for each $f$.

For each $0<\theta<1$, let $\delta=1-\theta$ and let $w=w_\delta$ be
the stationary point in \eqref{eq:ifs-scalar-stationarity}. Write
$k=k(w)$, $s=\widehat d_\theta=k'(w)$, and use
$p,q,z,v$ from \eqref{eq:ifs-four-frequencies}. Put
\[
 \begin{aligned}
 L_B&=\lambda+c\beta(1-\beta)(q-p),&
 \Lambda&=\lambda+c\beta(v-z),\\
 R_*&=c\beta\{(1-\beta)(p-z)+v-q\},&
 \Gamma&=\Lambda-\delta R_*.
 \end{aligned}
\]
Thus $R_*$ is the quantity denoted by $R$ in
Theorem~\ref{thm:ifs-analytic-optimizer}, and
\begin{equation}\label{eq:ifs-transition-time}
 \tau_*(\theta)=1-\frac{\delta L_B}{\Gamma}.
\end{equation}
The denominator obeys $\Gamma\ge\lambda-c\beta>0$.
Whenever $\tau_*\ge\beta$, the four frequencies on the partition
$0,\beta\tau_*,\beta,\tau_*,1$ define a feasible profile $f_*$.
The attainment calculation in Theorem~\ref{thm:ifs-analytic-optimizer}
then gives
\[
 A_{f_*}(t)=sL_{f_*}(t)\quad(\tau_*\le t\le1),\qquad
 d_\theta\ge\widehat d_\theta.
\]
Call $\tau_*(\theta)=\beta$ a \emph{collision}.
Retain $\theta_{\mathrm{tr}}$ from \eqref{eq:ifs-first-departure}.

For $w_\beta=(1+\sqrt\beta)^{-1}$, define
\begin{align}
 B_\beta&=\frac{c\beta^2(1-\beta)}4
            [1-(2-\beta)w_\beta],&
 \mathcal H_\beta&=\frac{c^2\beta w_\beta^3(4-w_\beta)}{32},
 \label{eq:ifs-transition-BH}\\
 J_\beta&=\beta^2(1-\beta)^4+
       \frac{\beta^4(1-\beta)}{1-w_\beta},&
 \kappa_\beta&=\frac{\mathcal H_\beta J_\beta}{8}>0.
 \notag
\end{align}

\begin{theorem}[Sharp separation of departure and collision]
\label{thm:ifs-first-transition}
For every nonempty compact $K\subset(0,1)$, there is $M_K\ge100$
such that, for $\beta\in K$ and $M\ge M_K$, the continued profile has
a unique collision parameter $\theta_{\mathrm{col}}\in(0,1)$. Uniformly
for $\beta\in K$, as $M\to\infty$,
\begin{align}
 \theta_{\mathrm{col}}
   &=\beta+B_\beta\varepsilon_M^2+O_K(\varepsilon_M^4),
   \label{eq:ifs-transition-collision}\\
 \theta_{\mathrm{tr}}-\theta_{\mathrm{col}}
   &=\kappa_\beta\varepsilon_M^4+O_K(\varepsilon_M^6).
   \label{eq:ifs-transition-sharp-gap}
\end{align}
Moreover,
$d_{\theta_{\mathrm{col}}}>\widehat d_{\theta_{\mathrm{col}}}$.
In particular, $\theta_{\mathrm{tr}}>\theta_{\mathrm{col}}$ and
$\theta_{\mathrm{tr}}=\beta+B_\beta\varepsilon_M^2+O_K(\varepsilon_M^4)$.
\end{theorem}

The coefficient of $(\log M)^{-4}$ in the separation is $c^4\kappa_\beta$.
Nonanalyticity at $\theta_{\mathrm{tr}}$ follows from
Proposition~\ref{prop:ifs-genuine-transition}.

\subsection{Equalization along delay orbits}

We use the following entropy recursion; compare
\cite{BanajiKolossvary2024} for Bedford--McMullen carpets.

\begin{lemma}[Equalizing frequency chains]
\label{lem:ifs-equalizing-recursion}
For $s>0$, put
\begin{equation}\label{eq:ifs-frequency-transfer}
 T_s(x)=x+\frac{\lambda-h(x)/s}{c\beta},\qquad 0\le x\le1.
\end{equation}
The deficit $A_f-sL_f$ is constant on an interval if and only if
\begin{equation}\label{eq:ifs-frequency-recursion}
 f(\beta t)=T_s(f(t))
\end{equation}
almost everywhere there. In particular, this recursion and equality
$A_f=sL_f$ at one point imply equality throughout the interval.
The map $T_s$ is strictly convex on $(0,1)$, has its minimum at
$(1+e^{sc\beta})^{-1}$, and satisfies $T_s(x)>x$ whenever
$s>d_0=c/\lambda$.
\end{lemma}

\begin{proof}
Absolute continuity gives
\[
 \begin{aligned}
 (A_f-sL_f)'(t)
 &=h(f(t))-s\{\lambda+c\beta(f(t)-f(\beta t))\}\\
 &=sc\beta\{f(\beta t)-T_s(f(t))\}
 \end{aligned}
\]
almost everywhere, proving \eqref{eq:ifs-frequency-recursion}.
The remaining assertions follow from
\[
 T_s'(x)=1-\frac{\log((1-x)/x)}{sc\beta},\qquad
 T_s''(x)=\frac1{sc\beta x(1-x)}>0,
\]
and $h(x)\le c$.
\end{proof}

\subsection{Relative-entropy bounds and global exclusion}
\label{sec:ifs-transition-exclusion}

For $0\le p\le1$ and $0<q<1$, write
\[
 \KL(p\Vert q)=p\log(p/q)+(1-p)\log((1-p)/(1-q)),
\]
with the continuous values at $p=0,1$, and put
$\operatorname{logit}(r)=\log(r/(1-r))$. The entropy identity is
\begin{equation}\label{eq:ifs-transition-tangent}
 h(p)=h(q)+h'(q)(p-q)-\KL(p\Vert q).
\end{equation}
For $X\sim\operatorname{Ber}(q)$, put
\[
 \psi_q(t)=\log\mathbb E\exp(t(X-q)),\qquad
 q_t=\sigma\bigl(\log(q/(1-q))+t\bigr).
\]
Since $\psi_q''(t)=q_t(1-q_t)\le1/4$ and
$\psi_q(0)=\psi_q'(0)=0$, the centered Bernoulli exponential-moment
bound~\cite{Hoeffding1963} gives $\psi_q(t)\le t^2/8$.
The tilted likelihood ratio therefore
gives, for every $t\in\mathbb R$,
\begin{equation}\label{eq:ifs-transition-moment-bound}
 \KL(p\Vert q)
 =\KL(p\Vert q_t)+t(p-q)-\psi_q(t)
 \ge t(p-q)-\frac{t^2}{8}.
\end{equation}
\begin{lemma}[A weighted relative-entropy bound]
\label{lem:ifs-transition-weighted-cost}
Let $f:[0,1]\to[0,1]$ and $g:[0,1]\to(0,1)$ be measurable.
Let $W$ be bounded, measurable and positive almost everywhere, and let
$\mathfrak a$ be a real measurable function such that
\[
 0<\int_0^1\frac{\mathfrak a(\xi)^2}{W(\xi)}\dd\xi<\infty.
\]
Then
\begin{equation}\label{eq:ifs-transition-weighted-cost}
 \int_0^1W(\xi)\KL(f(\xi)\Vert g(\xi))\dd\xi
 \ge
 \frac{2\left(\int_0^1\mathfrak a(\xi)
                    (f(\xi)-g(\xi))\dd\xi\right)^2}
      {\int_0^1\mathfrak a(\xi)^2/W(\xi)\dd\xi}.
\end{equation}
The left-hand side is an extended integral.
\end{lemma}

\begin{proof}
Take $t=4(p-q)$ in \eqref{eq:ifs-transition-moment-bound} to obtain
$\KL(p\Vert q)\ge2(p-q)^2$. Cauchy--Schwarz gives
\[
 \begin{aligned}
 \left|\int_0^1\mathfrak a(f-g)\dd\xi\right|^2
 &\le\left(\int_0^1\frac{\mathfrak a^2}{W}\dd\xi\right)
       \left(\int_0^1W(f-g)^2\dd\xi\right)\\
 &\le\frac12\left(\int_0^1\frac{\mathfrak a^2}{W}\dd\xi\right)
       \left(\int_0^1W\KL(f\Vert g)\dd\xi\right).
 \end{aligned}
\]
The first bound also makes the numerator finite.
\end{proof}

Fix $0<\theta<1$, with the continued-branch quantities defined above,
and put
\[
 \begin{aligned}
 E_f(t)&=A_f(t)-sL_f(t),&
 R_f&=L_f(\tau)-\theta L_f(1),\\
 \eta&=\frac{s-D}{\delta}=sw-k,&
 \varphi(r)&=-\log(1-r).
 \end{aligned}
\]
For $\tau\in(0,1)$, feasibility means $R_f=0$. Define
\begin{equation*}
\begin{aligned}
 y&=\sigma\left(\frac{c(1-\beta)(D-k)}{1-w}\right),\qquad
 b=\Gamma(\tau_*-\beta),\\
 m&=\lambda+c(1-\beta)\{\beta(z-p)
                -\theta[y-\beta p-(1-\beta)z]\},\\
 \chi&=\varphi(q)-\varphi(z)+(1-w)\{\varphi(y)-\varphi(v)\}.
\end{aligned}
\end{equation*}
The bounds
$m\ge\lambda-c(1-\beta)>0$ and
$\Gamma\ge\lambda-\beta c>0$ hold throughout the parameter range.
Here $b$ denotes the physical mismatch.

\begin{proposition}[A sufficient certificate for global optimality]
\label{prop:ifs-transition-exclusion}
For any $0<\beta<1$, $M\ge100$ and $0<\theta<1$, if
\begin{equation}\label{eq:ifs-transition-exclusion-test}
 \tau_*\ge\beta,\qquad 8bm(1-w)\ge c^2\chi,
\end{equation}
then $d_\theta=s=\widehat d_\theta$. If the second inequality is
strict, no maximizing profile has initial time $\tau<\beta$.
\end{proposition}

\begin{proof}
For an arbitrary measurable profile $f:[0,1]\to[0,1]$, the stationary
equation \eqref{eq:ifs-scalar-stationarity} gives
\begin{equation*}
\begin{aligned}
 S_f&:=wE_f(\tau)+(1-w)E_f(1)+\eta R_f\\
 &=wA_f(\tau)+(1-w)A_f(1)-kL_f(\tau)-(D-k)L_f(1).
\end{aligned}
\end{equation*}
Let $W=\mathbf1_{[0,\tau]}+(1-w)\mathbf1_{(\tau,1]}$.
Define $g_\tau$ by its frequencies and partition endpoints:
\[
 \begin{array}{c|c|c}
 &\text{frequencies}&\text{endpoints}\\
 \tau\ge\beta &(p,z,q,v)&(0,\beta\tau,\beta,\tau,1)\\
 \tau<\beta &(p,z,y,v)&(0,\beta\tau,\tau,\beta,1).
 \end{array}
\]
Omit zero-length cells.

On each cell, the frequency coefficient equals
$W\operatorname{logit}(g_\tau)$.
Since $G(\operatorname{logit}(r))=\varphi(r)$,
\eqref{eq:ifs-transition-tangent} gives
\begin{equation}\label{eq:ifs-transition-low-certificate}
 S_f=J_\tau-\int_0^1W\KL(f\Vert g_\tau)\dd\xi,
\end{equation}
where, writing $x_+=\max\{x,0\}$,
\[
\begin{aligned}
 J_\tau={}&-\lambda(k\tau+D-k)+\beta\tau\varphi(p)
       +(\min\{\tau,\beta\}-\beta\tau)\varphi(z)\\
 &+(\tau-\beta)_+\varphi(q)
 +(1-w)\{(\beta-\tau)_+\varphi(y)
                    +(1-\max\{\tau,\beta\})\varphi(v)\}\\
 ={}&(1-\tau)\mathcal C(k,w)+(\beta-\tau)_+\chi
       =(\beta-\tau)_+\chi.
\end{aligned}
\]
Here $\mathcal C(k,w)=0$ by \eqref{eq:ifs-scalar-pressure}.
Thus $S_f\le0$ for $\tau\ge\beta$.

For $\tau<\beta$, put $u=\beta-\tau$ and
$\bar p=\beta p+(1-\beta)z$. Direct integration and
\eqref{eq:ifs-transition-time} yield
\[
\begin{aligned}
 R_{g_\tau}
 &=\tau\{\lambda+c\beta(1-\beta)(z-p)\}\\
 &\quad-\theta\{\lambda+c(1-\beta)
                       [\tau(y-\bar p)+\beta(v-y)]\}\\
 &=m(\tau-\beta)+\delta L_B-(1-\beta)\Gamma=-mu-b.
\end{aligned}
\]
The coefficient of $f$ in the physical constraint $R_f$ is
\begin{equation*}
\begin{aligned}
 \mathfrak a(\xi)
 &=c\{\beta\mathbf1_{[0,\tau]}(\xi)
      -\mathbf1_{[0,\beta\tau]}(\xi)
      -\theta\beta+\theta\mathbf1_{[0,\beta]}(\xi)\},\\
 &\hspace{35mm}|\mathfrak a(\xi)|\le c.
\end{aligned}
\end{equation*}
If $R_f=0$, then
\[
 \int_0^1\mathfrak a(f-g_\tau)\dd\xi
       =R_f-R_{g_\tau}=b+mu,\qquad
 0<\int_0^1\frac{\mathfrak a^2}{W}\dd\xi
       \le\frac{c^2}{1-w}.
\]
Since $\mathfrak a=-c\theta\beta$ on $(\beta,1)$ and $b\ge0$,
\eqref{eq:ifs-transition-weighted-cost} and
$(b+mu)^2\ge4bmu$ give
\begin{equation}\label{eq:ifs-transition-global-cost}
 S_f\le\chi u-\frac{2(1-w)}{c^2}(b+mu)^2
 \le u\left[\chi-\frac{8bm(1-w)}{c^2}\right]\le0.
\end{equation}
Since $R_f=0$ and $S_f\le0$, every feasible pair satisfies
\[
 \min_{\tau\le t\le1}\frac{A_f(t)}{L_f(t)}
 \le\min\left\{\frac{A_f(\tau)}{L_f(\tau)},
                  \frac{A_f(1)}{L_f(1)}\right\}
 \le s.
\]
Equality is attained by $f_*$. If the second inequality in
\eqref{eq:ifs-transition-exclusion-test} is strict, $\tau<\beta$ gives
$S_f<0$, excluding a maximizer.
\end{proof}

\subsection{A four-node certificate}
\label{sec:ifs-transition-four-node}

For $\beta^2<\tau<\beta$, the times $\beta,\tau/\beta$ lie in the
covering window. Set
\[
 \alpha=\frac{w_\beta^2}{2},\qquad
 \gamma=(1-w_\beta)(w_\beta-\alpha),\qquad x=T_s(z),
\]
where $T_s$ is the transfer map from
\eqref{eq:ifs-frequency-transfer}. Differentiating
\eqref{eq:ifs-scalar-pressure} gives $T_s(v)=z$.
For $w>\alpha$ and $1-w>\gamma$, define
\begin{equation}\label{eq:ifs-transition-new-tilts}
\begin{aligned}
 \widetilde x&=\sigma(\operatorname{logit}(z)+\alpha sc),\\
 \widetilde z&=\sigma\left(
 \frac{(1-w)\operatorname{logit}(y)-sc(\beta\alpha+\gamma)}
      {1-w+\alpha}\right),\\
 \widetilde v&=\sigma\left(
 \frac{(1-w)\operatorname{logit}(v)+\gamma sc\beta}
      {1-w-\gamma}\right).
\end{aligned}
\end{equation}

\begin{proposition}[A four-node relative-entropy identity]
\label{prop:ifs-transition-four-certificate}
Suppose $\beta^2<\tau<\beta$, $0<x<1$, $w>\alpha$ and
$1-w>\gamma$. Let $g^\sharp$ have frequencies
$(p,\widetilde x,z,\widetilde z,v,\widetilde v)$ on the partition
\begin{equation*}
 0,\ \beta\tau,\ \beta^2,\ \tau,\ \beta,\ \tau/\beta,\ 1,
\end{equation*}
and let $W^\sharp$ have the corresponding values
$(1,1,1,1-w+\alpha,1-w,1-w-\gamma)$.
For every measurable $f:[0,1]\to[0,1]$,
\begin{align*}
 S_f^\sharp&:=(w-\alpha)E_f(\tau)+\alpha E_f(\beta)
       +\gamma E_f(\tau/\beta)+(1-w-\gamma)E_f(1)+\eta R_f
       \\
   &=(\beta-\tau)\Xi-I_f^\sharp,\qquad
 I_f^\sharp=\int_0^1W^\sharp\KL(f\Vert g^\sharp)\dd\xi,
\end{align*}
where
\begin{equation}\label{eq:ifs-transition-sharp-pressure}
\begin{aligned}
 \Xi={}&\chi-\beta\KL(x\Vert z)-(1-w)\KL(z\Vert y)
       +\beta\KL(x\Vert\widetilde x)\\
 &+(1-w+\alpha)\KL(z\Vert\widetilde z)
       +\frac{1-w-\gamma}{\beta}\KL(v\Vert\widetilde v).
\end{aligned}
\end{equation}
\end{proposition}

\begin{proof}
By definition,
\[
 S_f^\sharp-S_f=
 \alpha[E_f(\beta)-E_f(\tau)]
       -\gamma[E_f(1)-E_f(\tau/\beta)].
\]
On $(\beta\tau,\beta^2)$, $(\tau,\beta)$ and $(\tau/\beta,1)$,
the frequency coefficients change by
$\alpha sc$, $-sc(\beta\alpha+\gamma)$ and $\gamma sc\beta$,
respectively; the entropy coefficients become $W^\sharp$. Hence
\eqref{eq:ifs-transition-tangent} and
\eqref{eq:ifs-transition-new-tilts} make
$S_f^\sharp+I_f^\sharp$ independent of $f$.

For $\bar g=(p,x,z,v)$ on $0,\beta\tau,\beta^2,\beta,1$,
Lemma~\ref{lem:ifs-equalizing-recursion} gives $E_{\bar g}'=0$ on
$(\tau,1)$ and $S_{\bar g}^\sharp=S_{\bar g}$, since
$T_s(z)=x$ and $T_s(v)=z$. With $u=\beta-\tau$,
\[
 \begin{aligned}
 \int_0^1W\KL(\bar g\Vert g_\tau)\dd\xi
   &=u\{\beta\KL(x\Vert z)+(1-w)\KL(z\Vert y)\},\\
 I_{\bar g}^\sharp&=u\bigg\{
 \beta\KL(x\Vert\widetilde x)
 +(1-w+\alpha)\KL(z\Vert\widetilde z)\\
 &\hspace{15mm}+\frac{1-w-\gamma}{\beta}\KL(v\Vert\widetilde v)\bigg\}.
 \end{aligned}
\]
By \eqref{eq:ifs-transition-low-certificate}, therefore,
$S_f^\sharp+I_f^\sharp
 =S_{\bar g}^\sharp+I_{\bar g}^\sharp=u\Xi$ for every measurable $f$.
\end{proof}

\subsection{Proof of the sharp first-departure asymptotics}
\label{sec:ifs-transition-fourth-order-proof}

\begin{proof}[Proof of Theorem~\ref{thm:ifs-first-transition}]
Fix a nonempty compact $K\subset(0,1)$, $\beta\in K$, and retain
$s=\widehat d_\theta$. All implicit constants may depend on $K$.
Substitute $(D,k,\lambda)=(\varepsilon_M\overline D,
\varepsilon_M\overline k,c/\varepsilon_M)$ into
\eqref{eq:ifs-box-pressure} and \eqref{eq:ifs-scalar-pressure}
and divide by $c$. The identity $G(y)=c+y/2+\log\cosh(y/2)$
cancels the linear terms; the residuals are analytic in
$\varepsilon_M^2$, with value $(\overline D-1,\overline k-w)$ at zero.
Put $\theta_0=\min K/2$. The implicit-function theorem, uniformly near
$\beta\in K$, $1/2\le w\le(1+\sqrt{\theta_0})^{-1}$, gives
\begin{align}
 D&=\varepsilon_M+
       \frac{c\beta(1-\beta)}8\varepsilon_M^3
       +O_K(\varepsilon_M^5),\notag\\
 k(w)&=\varepsilon_M w+
       \frac{c\beta w(2-\beta-w)}8\varepsilon_M^3
       +O_K(\varepsilon_M^5),
       \label{eq:ifs-transition-expansion}\\
 s=k'(w)&=\varepsilon_M+
       \frac{c\beta(2-\beta-2w)}8\varepsilon_M^3
       +O_K(\varepsilon_M^5).\notag
\end{align}
The remainders are uniform with parameter derivatives.

After division by $c\beta\varepsilon_M^3/8$, the stationary residual in
\eqref{eq:ifs-scalar-stationarity} extends analytically to
$-1+2w-(1-\theta)w^2$ at $\varepsilon_M=0$.
Its derivative at $w_0(\theta)=(1+\sqrt\theta)^{-1}$ is
$2\sqrt\theta$, uniformly positive for $\theta_0\le\theta\le1$.
The implicit-function theorem and \eqref{eq:ifs-transition-time} give
\begin{equation}\label{eq:ifs-transition-time-estimates}
 \begin{gathered}
 w=w_0(\theta)+O_K(\varepsilon_M^2),\qquad
 \tau_*(\theta)=\theta+O_K(\varepsilon_M^2),\\
 \partial_\theta\tau_*(\theta)=1+O_K(\varepsilon_M^2),
 \qquad \theta_0\le\theta\le1.
 \end{gathered}
\end{equation}
Thus $\tau_*(\theta)=\beta$ has a local analytic root near $\beta$.
For every $0<\theta<1$, concavity gives
\[
 0\le\frac{D-k}{1-w}\le D=O_K(\varepsilon_M).
\]
Hence $p,q,z,v=1/2+O_K(\varepsilon_M)$ and
$\sup_{0<\theta<1}|\tau_*(\theta)-\theta|=O_K(\varepsilon_M^2)$.
Every collision lies in the positive-derivative interval of
\eqref{eq:ifs-transition-time-estimates}, proving uniqueness.
At collision the exact physical constraint reduces to
\[
 \theta_{\rm col}
 =\frac{\lambda\beta+c\beta^2(1-\beta)(z-p)}
 {\lambda+c\beta(1-\beta)[-\beta p-(1-\beta)z+v]}.
\]
Substituting \eqref{eq:ifs-transition-expansion} and the frequencies from
\eqref{eq:ifs-four-frequencies} proves
\eqref{eq:ifs-transition-collision}, with $B_\beta$ as in
\eqref{eq:ifs-transition-BH}.

On $[\theta_{\rm col},1]$,
\[
 \begin{gathered}
 \Gamma,m=\lambda+O_K(\varepsilon_M),\qquad
 \chi=O_K(\varepsilon_M^2),\qquad 1-w\ge c_K>0,\\
 \tau_*-\beta\ge\tfrac12(\theta-\theta_{\rm col}),\qquad
 bm\gtrsim_K(\theta-\theta_{\rm col})\varepsilon_M^{-2}.
 \end{gathered}
\]
Proposition~\ref{prop:ifs-transition-exclusion} therefore gives
\begin{equation}\label{eq:ifs-transition-coarse-tail}
 d_\theta=s\qquad\text{whenever}\qquad
 \theta_{\rm col}+C_K^{\rm tail}\varepsilon_M^4\le\theta<1.
\end{equation}

Fix
$Z_0>C_K^{\rm tail}+1+\max_{\beta\in K}\kappa_\beta$ and write
\begin{equation}\label{eq:ifs-transition-boundary-parameters}
 \theta=\theta_{\rm col}+\zeta\varepsilon_M^4,
 \qquad 0\le\zeta\le Z_0.
\end{equation}
Fix also a bounded $t$-interval containing the localization bound and
recovery values below. Choose all error constants on these fixed
$(\zeta,t)$ ranges before shrinking $\varepsilon_M$.
In the range \eqref{eq:ifs-transition-boundary-parameters},
$w=w_\beta+O_K(\varepsilon_M^2)$. Equations
\eqref{eq:ifs-four-frequencies} and \eqref{eq:ifs-transition-expansion}
give
\begin{align*}
 p&=\tfrac12+\tfrac c4(1-\beta)\varepsilon_M
                 +O_K(\varepsilon_M^3),&
 z&=\tfrac12+\tfrac c4(1-\beta-w_\beta)\varepsilon_M
                 +O_K(\varepsilon_M^3),\\
 v&=\tfrac12-\tfrac c4\beta\varepsilon_M
                 +O_K(\varepsilon_M^3),&
 y&=\tfrac12+\tfrac c4(1-\beta)\varepsilon_M
                 +O_K(\varepsilon_M^3).
\end{align*}
Also $x=T_s(z)=z+cw_\beta^2\varepsilon_M/8+O_K(\varepsilon_M^3)\in(0,1)$.
The node weights in Proposition~\ref{prop:ifs-transition-four-certificate}
converge to
\[
 w_\beta-\alpha,\quad \alpha,\quad \gamma,\quad
 (1-w_\beta)(1-w_\beta+\alpha),
\]
where $\alpha=w_\beta^2/2$ and
$\gamma=(1-w_\beta)(w_\beta-\alpha)$.
These limits are positive on $K$, so the node weights and $W^\sharp$
are uniformly bounded below.
By \eqref{eq:ifs-transition-new-tilts},
\[
 \widetilde x-x=O_K(\varepsilon_M^3),\qquad
 \widetilde z-z=O_K(\varepsilon_M^3),\qquad
 \widetilde v-v=O_K(\varepsilon_M^3).
\]

For $r_1,r_2=1/2+O_K(\varepsilon_M)$, Taylor's formula gives
\begin{equation}\label{eq:ifs-transition-KL-Taylor}
 \begin{aligned}
 \KL(r_1\Vert r_2)
 &=(r_1-r_2)^2\int_0^1
 \frac{1-\rho}
 {[r_2+\rho(r_1-r_2)][1-r_2-\rho(r_1-r_2)]}\,d\rho\\
 &=2(r_1-r_2)^2+
          O_K\bigl(\varepsilon_M^2(r_1-r_2)^2\bigr).
 \end{aligned}
\end{equation}
Here the reciprocal denominator is $4+O_K(\varepsilon_M^2)$.
Since
$\chi=c^2w_\beta(1-w_\beta)\varepsilon_M^2/8+O_K(\varepsilon_M^4)$
and $(1-w_\beta)^2=\beta w_\beta^2$,
\eqref{eq:ifs-transition-sharp-pressure} gives
\begin{equation}\label{eq:ifs-transition-Xi-expansion}
 \begin{aligned}
 \Xi
 &=c^2\left\{\frac{w_\beta(1-w_\beta)}8
       -\frac{\beta w_\beta^4}{32}
       -\frac{(1-w_\beta)w_\beta^2}{8}\right\}
          \varepsilon_M^2+O_K(\varepsilon_M^4)\\
 &=\mathcal H_\beta\varepsilon_M^2+O_K(\varepsilon_M^4).
 \end{aligned}
\end{equation}
The last three KL terms in \eqref{eq:ifs-transition-sharp-pressure}
are $O_K(\varepsilon_M^6)$.

Suppose $d_\theta>s$. The normalized variational formula
\eqref{eq:ifs-transition-variational} supplies a feasible measurable pair
$(f,\tau)$ with $E_f(r)>0$ for every $r\in[\tau,1]$.
The two-node certificate \eqref{eq:ifs-transition-low-certificate}
excludes $\tau\ge\beta$. Put $u=\beta-\tau>0$.
Since $b\ge0$, $m\gtrsim_K\varepsilon_M^{-1}$ and $1-w\gtrsim_K1$,
\eqref{eq:ifs-transition-global-cost} gives
\[
 0<S_f\le C_K\varepsilon_M^2u-c_K\varepsilon_M^{-2}u^2
 \quad\Longrightarrow\quad
 u=t\varepsilon_M^4,\qquad 0<t\le T_K.
\]
For small $\varepsilon_M$, this implies $\beta^2<\tau<\beta$.

The four-node identity and Pinsker inequality give
\[
 2c_K\|f-g^\sharp\|_2^2\le I_f^\sharp
 =u\Xi-S_f^\sharp<u\Xi=O_K(\varepsilon_M^6),\qquad
 c_KE_f(\tau)\le S_f^\sharp,
\]
using positivity of every node deficit and weight. Hence, uniformly
over all feasible improving profiles,
\begin{equation}\label{eq:ifs-transition-localize}
 \begin{gathered}
 0<S_f^\sharp\le C_K\varepsilon_M^6,\qquad
 I_f^\sharp\le C_K\varepsilon_M^6,\qquad
 \|f-g^\sharp\|_2\le C_K\varepsilon_M^3,\\
 0<E_f(\tau)\le C_K\varepsilon_M^6.
 \end{gathered}
\end{equation}

Put $\Delta=f-g^\sharp$ and
\[
 \mathcal I=(0,\beta\tau),\qquad
 \mathcal Z=(\beta^2,\tau),\qquad
 \mathcal V=(\beta,\tau/\beta).
\]
Put $\mathcal T=(0,1)\setminus(\mathcal I\cup\mathcal Z\cup\mathcal V)$.
Then
\[
 |\mathcal T|=O_K(\varepsilon_M^4),\qquad
 \|\Delta\|_{L^1(\mathcal T)}
 \le |\mathcal T|^{1/2}\|\Delta\|_2=O_K(\varepsilon_M^5).
\]
At $\tau=\beta$, the reference $g_\beta=(p,z,v)$ satisfies
$R_{g_\beta}=-b$, $S_{g_\beta}=0$, and $E'_{g_\beta}=0$ on $(\beta,1)$;
hence $E_{g_\beta}\equiv\eta b$ on $[\beta,1]$. Direct integration gives
\begin{equation}\label{eq:ifs-transition-baseline}
 E_{g^\sharp}(\tau)=\frac{\tau}{\beta}\eta b
   +\beta u\{h(\widetilde x)-h(z)
                    -sc\beta(\widetilde x-z)\}
       =O_K(\varepsilon_M^6).
\end{equation}
Here $\eta,b=O_K(\varepsilon_M^3)$ and the braces are
$O_K(\varepsilon_M^2)$. The entropy tangent identity gives
\begin{align*}
 E_f(\tau)-E_{g^\sharp}(\tau)
 ={}&\int_0^\tau
   \{h'(g^\sharp)-sc\beta
               +sc\mathbf1_{[0,\beta\tau]}\}\Delta\,d\xi\\
 &-\int_0^\tau\KL(f\Vert g^\sharp)\,d\xi.
\end{align*}
The coefficient equals $c(1-\beta)(s-D)=O_K(\varepsilon_M^3)$ on
$\mathcal I$, $-c(1-w_\beta)\varepsilon_M+O_K(\varepsilon_M^3)$
on $\mathcal Z$, and $O_K(\varepsilon_M)$ on $\mathcal T\cap(0,\tau)$.
By \eqref{eq:ifs-transition-localize}--\eqref{eq:ifs-transition-baseline},
\[
 \begin{aligned}
 -c(1-w_\beta)\varepsilon_M\int_{\mathcal Z}\Delta\,d\xi
 &=E_f(\tau)-E_{g^\sharp}(\tau)\\
 &\quad+O_K\bigl(\varepsilon_M^3\|\Delta\|_2
       +\varepsilon_M\|\Delta\|_{L^1(\mathcal T)}+I_f^\sharp\bigr)
 =O_K(\varepsilon_M^6).
 \end{aligned}
\]
Since $1-w_\beta\gtrsim_K1$,
\begin{equation}\label{eq:ifs-transition-middle-moment}
 \int_{\mathcal Z}\Delta\,d\xi=O_K(\varepsilon_M^5).
\end{equation}

The references differ by $O_K(\varepsilon_M)$ only on the thin cells,
so $R_{g^\sharp}=-b-mu+O_K(\varepsilon_M u)$.
Since $R_f=0$, \eqref{eq:ifs-transition-middle-moment} gives
\begin{equation*}
 \int_{\mathcal I\cup\mathcal V}\mathfrak a\Delta\,d\xi
 =b+mu+O_K(\varepsilon_M^5)
 =c(\zeta+t)\varepsilon_M^3+O_K(\varepsilon_M^5).
\end{equation*}
Here $b=c\zeta\varepsilon_M^3+O_K(\varepsilon_M^5)$ by
\eqref{eq:ifs-transition-time-estimates}. The weighted variance is
\begin{equation*}
 \begin{aligned}
 V_{\mathcal I,\mathcal V}
 &:=\int_{\mathcal I\cup\mathcal V}
                    \frac{\mathfrak a^2}{W^\sharp}\,d\xi\\
 &=c^2\left\{\beta\tau(1-\beta)^2(1-\theta)^2
       +\frac{(\tau/\beta-\beta)\theta^2\beta^2}{1-w}\right\}
 =c^2J_\beta+O_K(\varepsilon_M^2).
 \end{aligned}
\end{equation*}
Pinsker and weighted Cauchy--Schwarz on $\mathcal I\cup\mathcal V$ give
\[
 I_f^\sharp\ge
 \frac{2\bigl(\int_{\mathcal I\cup\mathcal V}
                         \mathfrak a\Delta\,d\xi\bigr)^2}
      {V_{\mathcal I,\mathcal V}}
 \ge\frac{2(\zeta+t)^2}{J_\beta}\varepsilon_M^6
                  -C_K\varepsilon_M^8.
\]
The four-node certificate now yields
\begin{equation*}
 \begin{aligned}
 0<S_f^\sharp
 &\le\varepsilon_M^6\left[
     \mathcal H_\beta t-\frac{2(\zeta+t)^2}{J_\beta}\right]
                   +C_K\varepsilon_M^8\\
 &=\varepsilon_M^6\left[
     \mathcal H_\beta(\kappa_\beta-\zeta)
     -\frac{2}{J_\beta}(t+\zeta-2\kappa_\beta)^2\right]
                   +C_K\varepsilon_M^8.
 \end{aligned}
\end{equation*}
Since $\min_K\mathcal H_\beta>0$, no improvement is possible for
$\zeta\ge\kappa_\beta+C'_K\varepsilon_M^2$ with $C'_K$ large.
The continued profile attains $s$ because $\tau_*\ge\beta$.
Together with the overlapping interval \eqref{eq:ifs-transition-coarse-tail},
this gives $d_\theta=s$ for
\[
 \theta_{\rm col}+\kappa_\beta\varepsilon_M^4
                      +C'_K\varepsilon_M^6\le\theta<1,
\]
and hence the upper bound on $\theta_{\rm tr}$.

For bounded $\zeta\ge0$ and $t>0$, let
$u=t\varepsilon_M^4$, $\tau=\beta-u$, and define
\begin{equation*}
 \begin{gathered}
 b_u=\frac{(\zeta+t)\beta^2}{(1-w_\beta)J_\beta},\qquad
 V=v-b_u\varepsilon_M^3,\\
 Z_u=T_s(V),\qquad X_u=T_s(Z_u).
 \end{gathered}
\end{equation*}
Use frequencies $(p_u,X_u,Z_u,V)$ on the intervals with endpoints
$0,\beta\tau,\beta^2,\beta,1$.
Writing $d=1-\beta$ and $\delta=1-\theta$, choose the first frequency
by the exact equation
\begin{equation*}
 p_u=\frac{\lambda(\tau-\theta)
           +c\beta u(\beta+\theta d)X_u
           +c\beta(\tau-\beta^2+\theta d^2)Z_u
           -\theta c\beta dV}
          {\delta c\beta d\tau}.
\end{equation*}
Then $L_f(\tau)=\theta L_f(1)$; the denominator is bounded below uniformly.

The scalar identities imply
\[
 T_s(v)=z,\qquad
 T_s'(v)=-\frac{\theta\eta}{(1-w)s}
                  =O_K(\varepsilon_M^2),\qquad
 T_s''=O_K(\varepsilon_M^{-1}),\qquad T_s'(z)=O_K(1).
\]
The bound on $T_s''$ holds on fixed neighborhoods of $v,z$.
Taylor's formula gives
$Z_u-z=O_K(\varepsilon_M^2\varepsilon_M^3+
\varepsilon_M^{-1}\varepsilon_M^6)$, hence
\begin{equation*}
 \begin{gathered}
 Z_u-z=O_K(\varepsilon_M^5),\qquad
 X_u-x=O_K(\varepsilon_M^5),\\
 p_u-p=-\frac{(\zeta+t)(1-\beta)^2}{J_\beta}
             \varepsilon_M^3+O_K(\varepsilon_M^5).
 \end{gathered}
\end{equation*}
Indeed, the thin and middle changes contribute $O_K(\varepsilon_M^5)$
to $R_f-R_{g_\tau}$. Writing
$p_u-p=A_u\varepsilon_M^3+O_K(\varepsilon_M^5)$, feasibility gives
\[
 \beta^2(1-\beta)^2 A_u-\beta^2(1-\beta)b_u=-(\zeta+t).
\]
Thus $A_u=-(\zeta+t)(1-\beta)^2/J_\beta$.
All four frequencies belong to $(0,1)$ for small $\varepsilon_M$.

By the transfer equations, almost everywhere on $(\tau,1)$,
\[
 E_f'(r)=
 \begin{cases}
 h(Z_u)-s[\lambda+c\beta(Z_u-X_u)]=0,&\tau<r<\beta,\\
 h(V)-s[\lambda+c\beta(V-Z_u)]=0,&\beta<r<1.
 \end{cases}
\]
Thus $E_f\equiv C_u$ on the whole window, and the two-node identity gives
\begin{equation*}
 \begin{aligned}
 C_u={}&u\chi-\beta\tau\KL(p_u\Vert p)
                  -\beta u\KL(X_u\Vert z)
                  -(\tau-\beta^2)\KL(Z_u\Vert z)\\
       &-(1-w)u\KL(Z_u\Vert y)
                  -(1-w)(1-\beta)\KL(V\Vert v).
 \end{aligned}
\end{equation*}

The first and terminal costs satisfy
\[
 \beta^2 A_u^2+(1-w_\beta)(1-\beta)b_u^2
       =\frac{(\zeta+t)^2}{J_\beta}.
\]
The thin-cell terms combine with $u\chi$ as in
\eqref{eq:ifs-transition-Xi-expansion}. Hence
\eqref{eq:ifs-transition-KL-Taylor} yields
\begin{equation}\label{eq:ifs-transition-recovery-gain}
 C_u=\varepsilon_M^6\left[
          \mathcal H_\beta t-\frac{2(\zeta+t)^2}{J_\beta}\right]
              +O_K(\varepsilon_M^8).
\end{equation}
The long, thin, and middle KL errors are, respectively,
\[
 \begin{gathered}
 O_K(\varepsilon_M^3\varepsilon_M^5+
       \varepsilon_M^2\varepsilon_M^6)=O_K(\varepsilon_M^8),\\
 O_K\bigl(\varepsilon_M^4[\varepsilon_M\varepsilon_M^3+
       \varepsilon_M^2\varepsilon_M^2]\bigr)=O_K(\varepsilon_M^8),
 \qquad O_K(\varepsilon_M^{10}).
 \end{gathered}
\]

Choose a constant $C''_K$ larger than the recovery remainder bound
divided by $\min_K\mathcal H_\beta$, and set
\[
 \zeta=\kappa_\beta-C''_K\varepsilon_M^2,
 \qquad t=2\kappa_\beta-\zeta.
\]
For small $\varepsilon_M$, these parameters are positive and
\eqref{eq:ifs-transition-recovery-gain} gives $C_u>0$. Therefore
\[
 \min_{\tau\le r\le1}\frac{A_f(r)}{L_f(r)}
       =s+\frac{C_u}{L_f(1)}>s.
\]
Thus $d_\theta>s$ at
$\theta=\theta_{\rm col}+\kappa_\beta\varepsilon_M^4
                         -C''_K\varepsilon_M^6$.
By the definition of $\theta_{\rm tr}$, this gives the lower bound and
proves \eqref{eq:ifs-transition-sharp-gap}.
Taking instead $\zeta=0$ and $t=2\kappa_\beta$ gives
$C_u=(\mathcal H_\beta^2J_\beta/8)\varepsilon_M^6
+O_K(\varepsilon_M^8)>0$, and hence
$d_{\theta_{\rm col}}>\widehat d_{\theta_{\rm col}}$.

Finally, \eqref{eq:ifs-transition-collision} gives
$\theta_{\rm tr}=\beta+B_\beta\varepsilon_M^2+O_K(\varepsilon_M^4)$.
All remainders are uniform on the fixed parameter ranges.
\end{proof}

\section{Oscillating constructions and their Assouad spectra}
\label{sec:ifs-comparison}

\subsection{Intermediate dimensions of the oscillating construction}
\label{sec:oscillating-tree}

Consider the oscillating construction of~\cite[Section~5]{BBFLX2025}.
Fix $0<\beta<1$ and $M\ge100$. Let $N_1=1<N_2<\cdots$ be integers
such that $N_{k+1}/N_k\to\infty$, and set
\[
 M_i=\begin{cases}
 M,&N_{2k-1}\le i<N_{2k},\\
 M+1,&N_{2k}\le i<N_{2k+1}.
 \end{cases}
\]
Let $E^*$ be the Cantor set with endpoint placement and cylinder lengths
\begin{equation*}
 \ell^*_\omega
 =\left(\prod_{i=1}^{n}M_i\right)^{-1}
       2^{S_{\lfloor\beta n\rfloor}(\omega)-\beta S_n(\omega)},
 \qquad |\omega|=n.
\end{equation*}
Write $\lambda_- =\log M$, $\lambda_+=\log(M+1)$, and $c=\log2$.
Denote the entropy values in
Theorems~\ref{thm:tree-dimensions} and~\ref{thm:tree-random-images}
by $d_\theta(\lambda)$ and $\mathcal D_{\theta,H}(\lambda)$.

\begin{theorem}\label{thm:oscillating-intermediate}
For every $0<\theta\le1$,
\begin{equation*}
 \underline{\dim}_\theta E^*=d_\theta(\lambda_+),
 \qquad
 \overline{\dim}_\theta E^*=d_\theta(\lambda_-).
\end{equation*}
In particular, the two dimensions are strictly different for every
$\theta>0$, and
\begin{equation}\label{eq:oscillating-endpoint-jump}
 \lim_{\theta\downarrow0}\underline{\dim}_\theta E^*
 =\frac{c}{\lambda_+}=\dim_{\mathrm H}E^*,
 \qquad
 \lim_{\theta\downarrow0}\overline{\dim}_\theta E^*
 =\frac{c}{\lambda_-}>\dim_{\mathrm H}E^*.
\end{equation}
For each fixed $0<H<1$, almost surely, simultaneously for all
$0<\theta\le1$,
\begin{equation*}
 \underline{\dim}_\theta B_H(E^*)
 =\mathcal D_{\theta,H}(\lambda_+),
 \qquad
 \overline{\dim}_\theta B_H(E^*)
 =\mathcal D_{\theta,H}(\lambda_-).
\end{equation*}
Consequently,
\begin{equation}\label{eq:oscillating-random-endpoints}
 \begin{split}
 \lim_{\theta\downarrow0}\underline{\dim}_\theta B_H(E^*)
 &=\min\{1,c/(H\lambda_+)\},\\
 \lim_{\theta\downarrow0}\overline{\dim}_\theta B_H(E^*)
 &=\min\{1,c/(H\lambda_-)\}.
 \end{split}
\end{equation}
If $H\le c/\lambda_+$, both image dimension functions equal one.
If $H>c/\lambda_+$, they are strictly different for every sufficiently
small positive $\theta$.
\end{theorem}

By Theorem~\ref{thm:tree-endpoint-rate}, each source curve approaches
its limit in~\eqref{eq:oscillating-endpoint-jump} at order
$(\log(1/\theta))^{-2}$. The same rate holds for each unsaturated
random-image branch in~\eqref{eq:oscillating-random-endpoints}.

\begin{proof}[Proof of Theorem~\ref{thm:oscillating-intermediate}]
The ratios in the oscillating construction satisfy
\[
 [2(M+1)]^{-1}\le\ell^*_{\omega i}/\ell^*_\omega\le2/M,
 \qquad |G_\omega|\ge(1-4/M)\ell^*_\omega.
\]
Moreover, for every word,
\begin{equation*}
 \ell^{\lambda_+}_\omega\le\ell^*_\omega
 \le\ell^{\lambda_-}_\omega.
\end{equation*}
Lemma~\ref{lem:oscillating-prefix-comparison} bounds the upper
dimensions from above by the $\lambda_-$ values and the lower
dimensions from below by the $\lambda_+$ values, at all sufficiently
small scales.
For the random lower bound, Markov's inequality and Borel--Cantelli
bound the energy by $e^{-N(s+\kappa/2)}$ for all large integer $N$.
Equation~\eqref{eq:cover-energy} and the interpolation from
$e^{-N}$ to intervening scales in the proof of
Theorem~\ref{thm:tree-random-images} then give divergent covering cost
at every small real scale, including when $\theta=1$.

It remains to attain each endpoint on a sequence of scales.  Put
$Q_k=(N_kN_{k+1})^{1/2}$ and let $\lambda(k)$ be $\lambda_-$ or
$\lambda_+$ according to the parity of $k$.  For every fixed $T$,
eventually $TQ_k<N_{k+1}$.  At every depth $n\le TQ_k$,
\begin{equation*}
 \left|\log\ell^*_\omega-\log\ell^{\lambda(k)}_\omega\right|
 \le (\lambda_+-\lambda_-)N_k=o(Q_k).
\end{equation*}
Only the first $N_k$ positions contribute to the difference.
Thus Lemma~\ref{lem:oscillating-prefix-comparison} applies on each
parity subsequence with $N=Q_k$.

The $M+1$ subsequence supplies covers and the $M$ subsequence supplies
energy bounds, giving the reverse inequalities. For the image,
$\sum_k e^{-c_0Q_k}\le\sum_k e^{-c_0k}<\infty$ for every $c_0>0$,
so Markov and Borel--Cantelli apply. Intersect over rational $\theta$
and rational strict test exponents; monotonicity of dimensions and
continuity of the entropy values extend the equalities to every $\theta>0$.

For strict separation, put
$\Delta=\lambda_+-\lambda_->0$ and
$b_- =\lambda_-+\beta c$.  The coding map from $E_{\beta,M}$
to $E_{\beta,M+1}$ is $\alpha$-H\"older on its domain, where
\[
 \alpha=1+\Delta/b_->1.
\]
Indeed,
\[
 -\log\ell^{\lambda_-}_\omega\le b_-n,
 \qquad \ell^{\lambda_+}_\omega
 =e^{-\Delta n}\ell^{\lambda_-}_\omega
 \le(\ell^{\lambda_-}_\omega)^\alpha.
\]
Sibling separation gives the same bound for distances. Hence
\[
 d_\theta(\lambda_+)\le \alpha^{-1}d_\theta(\lambda_-)
 <d_\theta(\lambda_-).
\]
For $\theta<\theta'\le1$, take source covers $\{U_i\}$ with
$\rho\le\operatorname{diam}U_i\le\rho^{\theta'}$ and
$\rho=(r/C)^{1/\alpha}$. Map $U_i\cap E_{\beta,M}$ and enlarge
image diameters to $r$ as needed.
Their diameters $d_i$ satisfy, for small $r$,
\[
 \begin{gathered}
 r\le d_i\le C(\operatorname{diam}U_i)^\alpha
 \le C^{1-\theta'}r^{\theta'}\le r^\theta,\\
 \sum_i d_i^s\le C^s\sum_i(\operatorname{diam}U_i)^{\alpha s}.
 \end{gathered}
\]
Let $\theta'\downarrow\theta$; at $\theta=1$ use fixed-scale box covers.

The fixed-parameter endpoint limit of
Proposition~\ref{prop:tree-basic-regimes} gives
\eqref{eq:oscillating-endpoint-jump}.  The value of
$\dim_{\mathrm H}E^*$ and the realization by a strongly separated
two-map bi-Lipschitz IFS are proved in
\cite[Theorem~5.1]{BBFLX2025}.

Finally, for every fixed $\lambda$ the constant profile $f=1/2$
gives
\[
 \min\{1,c/(H\lambda)\}
 \le\mathcal D_{\theta,H}(\lambda)
 \le\min\{1,d_\theta(\lambda)/H\}.
\]
The upper bound is \eqref{eq:tree-random-source-upper}.
Letting $\theta\downarrow0$ gives \eqref{eq:oscillating-random-endpoints}.
The same inequalities give saturation for $H\le c/\lambda_+$;
for $H>c/\lambda_+$, the distinct limits give a gap for small $\theta$.
\end{proof}

\begin{remark}
This computes the oscillating example in~\cite[Section~1.5]{BBFLX2025}.
The probability-one event corresponds to each fixed set $E^*$ and fixed $H$.
\end{remark}

\subsection{Local counting and the Assouad spectrum}
\label{sec:tree-assouad-spectrum}

For $0<\vartheta<1$, the Assouad spectrum is defined by
\cite{FraserYu2018Spectra}
\[
 \dim_{\mathrm A}^{\vartheta}F
 =\inf\left\{s:\ \exists C>0\ \forall x\in F\ \forall 0<R<1,
 N_{R^{1/\vartheta}}(F\cap B(x,R))
 \le C(R/R^{1/\vartheta})^s\right\},
\]
where $N_r$ denotes the least number of intervals of diameter at most
$r$ needed to cover a set.

Retain $c=\log2$, $\lambda=\log M$, and define, for $t\ge0$,
\begin{align*}
 U(t)&=\log(1+2^{(1-\beta)t}),&
 V(t)&=\log(1+2^{-\beta t}),\\
 P(t)&=\beta U(t)+(1-\beta)V(t)-\lambda t,&
 F(t)&=V(t)-(\lambda-\beta c)t.
\end{align*}
Let $d_{\mathrm B}$ and $d_{\mathrm A}$ be the unique positive zeros
of $P$ and $F$, respectively. The former is the previously defined
box dimension $D(\lambda)$; the latter is the Assouad dimension.
These endpoints were computed in~\cite{BBFLX2025}; its
Proposition~5.6 gives the Assouad value for $E^*$.
Also put
\begin{equation*}
 q(t)=\frac1{1+2^{\beta t}},\qquad
 a(t)=\lambda-\beta c+\beta c q(t),\qquad
 j_0=\frac{1-\beta}{\beta}.
\end{equation*}
Thus $F'(t)=-a(t)<0$ and
$F''(t)=(\beta c)^2q(t)(1-q(t))>0$.
Both pressures are strictly decreasing from $c$ to $-\infty$, and
\[
 P(t)-F(t)=\beta\{U(t)-V(t)-ct\}<0\qquad(t>0).
\]
Thus $0<d_{\mathrm B}<d_{\mathrm A}$.

\begin{theorem}[Assouad spectrum from local covering growth]
\label{thm:ifs-assouad-spectrum}
For every $0<\vartheta<1$, the fixed Cantor set $E_{\beta,M}$ and
the oscillating set $E^*$ have the same Assouad spectrum, equal to
\begin{equation*}
 S(\vartheta)
 =\min_{d_{\mathrm B}\le t\le d_{\mathrm A}}
 \left\{t+\frac{(1-\beta)\vartheta}
                    {\beta\lambda(1-\vartheta)}F(t)\right\}.
\end{equation*}
\end{theorem}

\begin{lemma}[Uniform local partition-sum bound]
\label{lem:ifs-local-partition}
For either construction, write $Q_u=-\log\ell_u$ and
\[
 Z_{u,j}(t)=\sum_{|\omega|=j}(\ell_{u\omega}/\ell_u)^t.
\]
For every $d_{\mathrm B}\le t\le d_{\mathrm A}$,
\begin{equation}\label{eq:assouad-uniform-partition}
 Z_{u,j}(t)\le C_t\exp\{\kappa(t)Q_u\},
 \qquad \kappa(t)=\frac{j_0F(t)}{\lambda}.
\end{equation}
The constant is independent of the parent word $u$, the depth $j$,
and the contraction sequence satisfying $M\le M_i\le M+1$.
\end{lemma}

\begin{proof}
The cylinder lengths of both sets have the form
\[
 \ell_\omega=
 \left(\prod_{i=1}^{n}M_i\right)^{-1}
 2^{S_{\lfloor\beta n\rfloor}(\omega)-\beta S_n(\omega)},
 \qquad M\le M_i\le M+1.
\]
The fixed construction has $M_i=M$. All child-to-parent
ratios lie in $[(2M+2)^{-1},2/M]$, and every sibling gap is at least
$(1-4/M)$ times its parent length. Constants below are uniform in
the parent word and in the contraction sequence within these bounds.

Fix a parent word $u$ of length $k$. Put
\[
 Q_u=-\log\ell_u,\qquad
 Y=S_k(u)-S_{\lfloor\beta k\rfloor}(u),\qquad
 a_0=\lambda-\beta(1-\beta)c.
\]
Since $S_{\lfloor\beta k\rfloor}(u)\le\beta k$,
\begin{equation*}
 Q_u\ge a_0k+\beta cY.
\end{equation*}
Fix $t\in[d_{\mathrm B},d_{\mathrm A}]$, so that
$P(t)\le0$, $F(t)\ge0$, and $V(t)-\lambda t<0$.
For $k=0$, the digit sum gives
$\log Z_{u,j}(t)\le jP(t)+O(1)\le O(1)$. Hence assume $k\ge1$.

If $\lfloor\beta(k+j)\rfloor\le k$, let $m$ count the newly included
ones of $u$. Since $M_i\ge M$,
\[
 \begin{gathered}
 m\le\min\{Y,\beta j+1,(1-\beta)k+1\},\\
 \begin{aligned}
 \log Z_{u,j}(t)
 &\le j\{V(t)-\lambda t\}+ctm
 \le\frac{mF(t)}\beta+O(1)\\
 &=\kappa(t)(a_0k+\beta cm)
   +\frac{a_0F(t)}{\beta\lambda}\{m-(1-\beta)k\}+O(1)\\
 &\le\kappa(t)Q_u+O(1).
 \end{aligned}
 \end{gathered}
\]

If $e=\lfloor\beta(k+j)\rfloor-k>0$, the digit weights are
$2^{(1-\beta)t}$ on the first $e$ descendants and $2^{-\beta t}$
on the rest. Since $j\ge j_0k-O(1)$ and $P(t)\le0$,
\[
 \begin{aligned}
 \log Z_{u,j}(t)
 &\le jP(t)+(1-\beta)k\{V(t)-U(t)\}+ctY+O(1)\\
 &\le j_0k\{V(t)-\lambda t\}+ctY+O(1)\\
 &=\kappa(t)(a_0k+\beta cY)
   +c\left(t-\frac{(1-\beta)F(t)}\lambda\right)
       \{Y-(1-\beta)k\}+O(1)\\
 &\le\kappa(t)Q_u+O(1).
 \end{aligned}
\]
The last step uses $Y\le(1-\beta)k+1$ and
$t-(1-\beta)F(t)/\lambda\ge a_0t/\lambda>0$,
since $F(t)<\beta ct$. Both cases prove
\eqref{eq:assouad-uniform-partition}.
\end{proof}

\begin{proof}[Proof of Theorem~\ref{thm:ifs-assouad-spectrum}]
\smallskip\noindent\emph{Uniform upper cover.}
Stopping first at $R$, then at $r=R^{1/\vartheta}$, gives cylinder
lengths in $[R/(2M+2),R]$ and $[r/(2M+2),r]$, respectively.
A ball $B(x,R)$ meets $O(1)$ stopped parents; descendant depths
range over $O(1+\log(R/r))$ integers. Thus
\eqref{eq:assouad-uniform-partition} gives
\[
 N_r(E\cap I_u)
 \le C(1+\log(R/r))(\ell_u/r)^t e^{\kappa(t)Q_u}.
\]
With $\nu=(1-\vartheta)/\vartheta$ and $N=-\log R$,
$Q_u=N+O(1)$ and $\log(\ell_u/r)=\nu N+O(1)$, hence
\[
 N_r(E\cap B(x,R))
 \le C(1+N)\exp\{N[\nu t+\kappa(t)]\}.
\]
Absorbing $1+N$ into a strict exponent margin and minimizing in $t$
gives the upper bound for both sets.

\smallskip\noindent\emph{Matching local lower bounds in the fixed construction.}
Let $K=j_0/(\nu\lambda)$ and let $t_*$ minimize $t+KF(t)$ on
$[d_{\mathrm B},d_{\mathrm A}]$. Since $F''>0$, this minimizer is
unique. The derivative conditions are
\[
 \nu\lambda\ge j_0a(d_{\mathrm B})\quad(t_*=d_{\mathrm B}),\qquad
 \nu\lambda=j_0a(t_*)\quad(d_{\mathrm B}<t_*<d_{\mathrm A}),
\]
and $\nu\lambda\le j_0a(d_{\mathrm A})$ when $t_*=d_{\mathrm A}$.

For $u=1^k$, $R_k=\ell_u=\exp(-\lambda k+O(1))$ and
$I_u\subset B(1,R_k)$. Take $j=vk+O(1)$ and, for $v\ge j_0$, put
\[
 p(t)=\frac1{1+2^{-(1-\beta)t}},\qquad q=q(t).
\]
For $v=j_0$, choose $j$ with $\lfloor\beta(k+j)\rfloor=k$.
Prescribe $p(t)$ and $q(t)$ on the first
$e=\lfloor\beta(k+j)\rfloor-k$ digits and the remaining $j-e$ digits,
rounding the counts. Stirling's formula gives
$\exp(kC_t(v)+o(k))$ cylinders of common relative length
$\exp(-kB_t(v)+O(1))$, where
\begin{align*}
 B_t(v)&=\lambda v-c(1-\beta)
     -c(1-\beta)\{\beta v-(1-\beta)\}p(t)
     +c\beta(1-\beta)(v+1)q(t),\\
 C_t(v)&=\{\beta v-(1-\beta)\}h(p(t))
                    +(1-\beta)(v+1)h(q(t)).
\end{align*}
These expressions satisfy
\begin{equation}\label{eq:assouad-entropy-identity}
 B_t(j_0)=j_0a(t),\qquad
 C_t(v)-tB_t(v)=(v-j_0)P(t)+j_0F(t).
\end{equation}
Also $B_t(v)$ is strictly increasing in $v$; its slope is
$\lambda+\beta(1-\beta)c(q(t)-p(t))>0$.

For $t=t_*=d_{\mathrm B}$, choose the unique $v\ge j_0$ with
$B_t(v)=\nu\lambda$. Then \eqref{eq:assouad-entropy-identity} gives
\[
 \frac{C_t(v)}{\nu\lambda}
 =d_{\mathrm B}+\frac{j_0F(d_{\mathrm B})}{\nu\lambda}.
\]
For $d_{\mathrm B}<t=t_*<d_{\mathrm A}$, take $v=j_0$; the first
block is empty and $B_t(v)=\nu\lambda$, so the same identity gives
$C_t(v)/(\nu\lambda)=t_*+KF(t_*)$.

For $t=t_*=d_{\mathrm A}$, take $v=\nu\lambda/a(t)\le j_0$.
All delayed digits are ones. Frequency $q(t)$ gives relative
logarithmic length $kva(t)+O(1)$ and count
$\exp(kv h(q(t))+o(k))$, with
\[
 \frac{vh(q(t))}{\nu\lambda}
 =\frac{h(q(t))}{a(t)}=t+\frac{F(t)}{a(t)}=d_{\mathrm A}.
\]

In every case the selected cylinders have diameter comparable to
$r_k=R_k^{1/\vartheta}$. Their common ancestors have at least this
length, so the gap bound separates them by a fixed multiple of $r_k$.
Take one point from each cylinder. Every interval of diameter $r_k$
contains only a uniformly bounded number of these points, so
\[
 \limsup_{k\to\infty}
 \frac{\log N_{r_k}(E_{\beta,M}\cap B(1,R_k))}
      {\log(R_k/r_k)}\ge S(\vartheta).
\]

\smallskip\noindent\emph{The oscillating construction.}
Put $n_k=\lfloor\sqrt{N_{2k-1}N_{2k}}\rfloor$. Then
\[
 \frac{N_{2k-1}}{n_k}\to0,\qquad \frac{n_k}{N_{2k}}\to0,\qquad
 -\log\ell^*_{1^{n_k}}=\lambda n_k+o(n_k).
\]
The $O_\vartheta(n_k)$ descendant depths remain in one $M$-block,
with the same relative lengths and counts as above. The new target
$r_k=(\ell^*_{1^{n_k}})^{1/\vartheta}$ differs by $\exp(o(n_k))$.
For minimum point separation $d_k$, the gap bound gives
\[
 \begin{gathered}
 r_k/d_k\le\exp(o(n_k)),\\
 \sup_{\operatorname{diam}J\le r_k}\#\{\text{selected points in }J\}
 \le1+r_k/d_k\le\exp(o(n_k)).
 \end{gathered}
\]
Thus the logarithmic covering number loses only $o(n_k)$,
preserving the lower exponent $S(\vartheta)$ for $E^*$.

\end{proof}

\subsection{Explicit branches and transition regularity}

Put
\begin{equation}\label{eq:ifs-assouad-transitions}
 \vartheta_{\mathrm B}
 =\frac{\beta\lambda}{\beta\lambda+(1-\beta)a(d_{\mathrm B})},
 \qquad
 \vartheta_{\mathrm A}
 =\frac{\beta\lambda}{\beta\lambda+(1-\beta)a(d_{\mathrm A})}.
\end{equation}

\begin{corollary}[Three explicit Assouad phases]
\label{cor:ifs-assouad-phases}
The transition parameters satisfy $\beta<\vartheta_{\mathrm B}<\vartheta_{\mathrm A}<1$
and
\begin{equation}\label{eq:assouad-three-phases}
 S(\vartheta)=
 \begin{cases}
 d_{\mathrm B}+\displaystyle
 \frac{(1-\beta)\vartheta F(d_{\mathrm B})}
      {\beta\lambda(1-\vartheta)},
 &0<\vartheta\le\vartheta_{\mathrm B},\\[7pt]
 \displaystyle\frac{h(q_\vartheta)}
                        {\lambda-\beta c+\beta c q_\vartheta},
 &\vartheta_{\mathrm B}\le\vartheta\le\vartheta_{\mathrm A},\\[7pt]
 d_{\mathrm A},&\vartheta_{\mathrm A}\le\vartheta<1,
 \end{cases}
\end{equation}
where, in the middle interval,
\begin{equation}\label{eq:assouad-middle-frequency}
 q_\vartheta=
 \frac{\beta\lambda(1-\vartheta)/[(1-\beta)\vartheta]
                         -\lambda+\beta c}{\beta c}.
\end{equation}
The expressions agree at their common endpoints. The spectrum is
continuously differentiable on $(0,1)$ and real
analytic on each of the three open phases. It is not twice
differentiable at either transition point, and strictly increasing
on $(0,\vartheta_{\mathrm A})$.
\end{corollary}

\begin{proof}
Put $K(\vartheta)=j_0\vartheta/[\lambda(1-\vartheta)]$. The function $t+KF(t)$ is strictly
convex and its derivative is $1-Ka(t)$. Since $a$ is strictly
decreasing, the minimum is attained at $d_{\mathrm B}$ when
$K\le1/a(d_{\mathrm B})$, at $d_{\mathrm A}$ when
$K\ge1/a(d_{\mathrm A})$, and at the unique solution of
$a(t)=1/K$ for intermediate values. These inequalities give exactly
\eqref{eq:ifs-assouad-transitions}. Since $a(t)<\lambda$, both
transitions lie strictly above $\beta$. The identity
\[
 h(q(t))=V(t)+\beta ctq(t)
\]
then gives \eqref{eq:assouad-three-phases} and
\eqref{eq:assouad-middle-frequency}.

On the middle phase the minimizing parameter $t=t(\vartheta)$
is analytic and satisfies $1+KF'(t)=0$. The envelope derivative is
\[
 S'(\vartheta)=K'(\vartheta)F(t(\vartheta)).
\]
This agrees with the derivative of the first phase at
$\vartheta_{\mathrm B}$ and vanishes at
$\vartheta_{\mathrm A}$, proving continuous differentiability.
Within the middle phase,
\[
 S''(\vartheta)=K''F(t)
       -\frac{(K')^2(F'(t))^2}{KF''(t)}.
\]
At each transition the second term has a strictly negative finite
limit, whereas it is absent in the adjacent phase. Thus the second
derivative fails to exist at both transition points. Finally,
$F(t)>0$ before the upper transition, so $S'>0$ there.
\end{proof}

\begin{remark}
At $t=d_{\mathrm A}$ the uniform partition bound is independent
of the parent. Repeating the stopping argument for arbitrary
$0<r<R$ proves $\dim_{\mathrm A}E\le d_{\mathrm A}$; the attained
plateau gives the reverse inequality. The plateau also shows that
the quasi-Assouad dimension equals $d_{\mathrm A}$.
The spectrum formula answers the corresponding question
in~\cite[Section~1.5]{BBFLX2025}.
\end{remark}

\part{Projections of prescribed measures}
\label{part:measures}
\section{Intermediate dimensions of measures under projections}
\label{sec:quasi-measures}

For sets, quasi-Assouad bounds yield projection criteria for box and
packing dimensions~\cite[Corollary~1.3]{FFS2021}, and
Angelini and Molter~\cite[Corollary~3.7]{AngeliniMolter2025}
treated intermediate dimensions.
For a prescribed measure, Angelini, Molter, and Tejada~\cite[Theorem~5.13 and
Remark~5.14]{AMT2026} proved preservation of its upper and lower
intermediate dimensions under the strict hypothesis
$\dim_A\mu<m$, and asked whether quasi-Assouad dimension could replace
Assouad dimension. Theorem~\ref{thm:quasi-measure-projection} establishes
preservation under the condition $\dim_{qA}\mu\leq m$.

\subsection{Measure dimensions and the projection criterion}

Throughout this section, $\mu$ is a nonzero finite Borel measure on
$\R^n$ with compact support $K$, and $B(x,u)$ denotes a closed ball.
The all-centre formulation of intermediate dimensions
in~\cite[Lemma~3.3]{AMT2026} is as follows. For
$0<\theta\leq1$, $\udim_\theta\mu$ is the infimum of those $s\geq0$
for which
\begin{equation*}
 \exists\,c,r_0>0\quad
 \forall\,0<r<r_0\ \forall\,x\in K\ \exists\,u\in[r,r^\theta]:
 \qquad \mu(B(x,u))\geq c u^s.
\end{equation*}
The lower dimension $\ldim_\theta\mu$ is the infimum of those $s\geq0$
for which
\begin{equation*}
 \exists\,c>0\ \exists\,r_j\downarrow0\quad
 \forall\,j\ \forall\,x\in K\ \exists\,u\in[r_j,r_j^\theta]:
 \qquad \mu(B(x,u))\geq c u^s.
\end{equation*}
Here $c$ is uniform in the scale and centre, while $u$ may depend on both.
At $\theta=1$ these are the upper and lower Minkowski dimensions of $\mu$.

Write $q=\dim_{qA}\mu$.  The defining quasi-Assouad growth property
\cite[Definition~2.1]{HHT2020} is that, for every $d>q$ and $\eta>0$,
there exist $C,R_*>0$ such that
\begin{equation}\label{eq:measure-quasi-growth}
 \frac{\mu(B(x,R))}{\mu(B(x,u))}
 \leq C(R/u)^d,
 \qquad x\in K,\quad 0<u\leq R^{1+\eta}\leq R<R_*.
\end{equation}
Thus $q$ is the infimum of the exponents valid for every $\eta>0$;
$C,R_*$ may depend on $d,\eta,\mu$, uniformly in $x,u,R$.

For $V\in G(n,m)$, let $P_V$ be the orthogonal projection onto $V$
and $P_V{}_\#\mu$ its pushforward.  Write $\gamma_{n,m}$ for the
rotation-invariant probability measure on $G(n,m)$.

\begin{theorem}[Quasi-Assouad criterion for measure projections]
\label{thm:quasi-measure-projection}
Let $m\in\{1,\ldots,n-1\}$ and suppose that $\dim_{qA}\mu\leq m$.
There is a set of full $\gamma_{n,m}$-measure in
$G(n,m)$ on which, simultaneously for all $0<\theta\leq1$,
\begin{equation*}
 \udim_\theta(P_V{}_\#\mu)=\udim_\theta\mu,
 \qquad
 \ldim_\theta(P_V{}_\#\mu)=\ldim_\theta\mu.
\end{equation*}
\end{theorem}

This answers the question in \cite[Remark~5.16]{AMT2026} and also
includes the endpoint $\dim_{qA}\mu=m$.

\subsection{Window comparison and small-mass witnesses}

\begin{lemma}[Comparison of measure windows]
\label{lem:measure-window-comparison}
If $D_\theta$ denotes either intermediate dimension of $\mu$, then
\begin{equation}\label{eq:measure-dimension-continuity}
 D_\varphi\leq D_\theta
 \leq\frac{\theta}{\varphi}D_\varphi,
 \qquad 0<\varphi\leq\theta\leq1.
\end{equation}
In particular, $D_\theta$ is continuous on $(0,1]$ when it is finite.
Moreover,
\begin{equation}\label{eq:measure-dimension-qA-bound}
 0\leq\ldim_\theta\mu\leq\udim_\theta\mu
 \leq\udim_B\mu\leq\dim_{qA}\mu.
\end{equation}
\end{lemma}

\begin{proof}
Window inclusion gives $D_\varphi\leq D_\theta$. For an admissible
exponent $a$ at $\varphi$, choose its witness
$u\in[\delta^{1/\varphi},\delta]$ and set $v=\max\{u,\delta^{1/\theta}\}$.
Then $v\in[\delta^{1/\theta},\delta]$ and
\[
 \mu(B(x,v))\geq cu^a\geq
 \begin{cases}
 cu^{a\theta/\varphi},&v=u,\\
 c\delta^{a/\varphi},&v=\delta^{1/\theta},
 \end{cases}
 =cv^{a\theta/\varphi}.
\]
This holds uniformly in $x,\delta$, on the defining scales for either
dimension. Taking the infimum over $a$ proves
\eqref{eq:measure-dimension-continuity}.

For $q<\infty$, fix $d>q$, $\eta>0$, and $R<\min\{R_*,1\}$.
Put $c_R=\inf_{x\in K}\mu(B(x,R))>0$, by compactness. Then
\[
 \mu(B(x,u))\geq C^{-1}c_R(u/R)^d
 \quad(x\in K,\ 0<u\leq R^{1+\eta})
 \quad\Longrightarrow\quad\udim_B\mu\leq d.
\]
Let $d\downarrow q$; the other bounds in
\eqref{eq:measure-dimension-qA-bound} follow from the definitions.
\end{proof}

\begin{lemma}[Small-mass witnesses]
\label{lem:measure-small-mass-witnesses}
Fix $0<\varphi\leq1$ and $t\geq0$.  If $t<\udim_\varphi\mu$,
there are arbitrarily small $r>0$ and centres $x_r\in K$ such that
\begin{equation*}
 \mu(B(x_r,u))<u^t\qquad\text{for every }u\in[r,r^\varphi].
\end{equation*}
If $t<\ldim_\varphi\mu$, such a centre exists at every sufficiently
small scale $r$.  The centre may depend on $r$.
\end{lemma}

\begin{proof}
Negating the conclusion gives the defining mass lower bound with
$c=1$: at all small scales for the upper dimension, or along a sequence
for the lower dimension. In either case $t$ would be admissible.
\end{proof}

\subsection{Point potentials on an enlarged window}

For a real number $m>0$, $0\leq s\leq m$, and $0<r<1$, set
\begin{equation*}
 \phi_{r,\theta}^{s,m}(z)=
 \begin{cases}
 1,&|z|\leq r,\\
 r^s|z|^{-s},&r<|z|\leq r^\theta,\\
 r^{s+\theta(m-s)}|z|^{-m},&|z|>r^\theta.
 \end{cases}
\end{equation*}
Its point capacity for the prescribed measure is
\begin{equation*}
 C_{r,\theta}^{s,m}(\mu)
 =\left(\inf_{x\in K}\int\phi_{r,\theta}^{s,m}(x-y)
          \dd\mu(y)\right)^{-1}.
\end{equation*}
The point potential is continuous on compact $K$, so its infimum
is attained.

\begin{lemma}[Potential estimate on an enlarged window]
\label{lem:quasi-window-potential}
Assume $\dim_{qA}\mu\leq m$.  Fix $0<s<t<m$,
$0<\theta\leq1$, $\eta>0$, and $0<\varepsilon<t-s$, and put
$\varphi=\theta/(1+\eta)$.  There exist constants $A,r_0>0$
such that, if $0<r<r_0$ and $x\in K$ satisfy
\begin{equation*}
 \mu(B(x,u))\leq u^t
 \qquad\text{for every }u\in[r,r^\varphi],
\end{equation*}
then
\begin{equation*}
 \int\phi_{r,\theta}^{s,m}(x-y)\dd\mu(y)
 \leq A r^{s+\theta(t-s-\varepsilon)}.
\end{equation*}
The constants are independent of $r$ and $x$.
\end{lemma}

\begin{proof}
Put $M_x(u)=\mu(B(x,u))$ and
\[
 \rho=r^\theta=R_0^{1+\eta},\qquad R_0=r^\varphi,
 \qquad d=m+\varepsilon>\dim_{qA}\mu.
\]
For $R_0\leq u<R_*<1$, \eqref{eq:measure-quasi-growth} applies since
$\rho\leq u^{1+\eta}$:
\[
 M_x(u)\leq C(u/\rho)^dM_x(\rho)\leq Cu^d\rho^{t-d}.
\]
For $R_*\leq u\leq1$, use
$M_x(u)\leq\mu(K)\leq\mu(K)R_*^{-d}u^d\rho^{t-d}$.
Thus the bound holds throughout $[R_0,1]$.

Layer cake gives
\begin{equation*}
\begin{split}
 \int\phi_{r,\theta}^{s,m}(x-y)\dd\mu(y)
 ={}&s r^s\int_r^\rho M_x(u)u^{-s-1}\dd u\\
 &+m r^{s+\theta(m-s)}
             \int_\rho^\infty M_x(u)u^{-m-1}\dd u.
\end{split}
\end{equation*}
This includes atoms in $B(x,r)$; the first term vanishes at $\theta=1$.
The small-mass hypothesis gives
\begin{align*}
 s r^s\int_r^\rho M_x(u)u^{-s-1}\dd u
 &\leq C r^s\rho^{t-s}
 =C r^{s+\theta(t-s)},\\
 m r^{s+\theta(m-s)}\int_\rho^{R_0}M_x(u)u^{-m-1}\dd u
 &\leq C r^{s+\theta(m-s)}\rho^{t-m}
 =C r^{s+\theta(t-s)}.
\end{align*}
The remaining ranges satisfy
\begin{align*}
 r^{s+\theta(m-s)}
       \int_{R_0}^1 M_x(u)u^{-m-1}\dd u
 &\leq C r^{s+\theta(m-s)}\rho^{t-d}
            \int_{R_0}^1 u^{d-m-1}\dd u\\
 &\leq \frac{C}{\varepsilon}
              r^{s+\theta(t-s-\varepsilon)}.
\end{align*}
\[
 m r^{s+\theta(m-s)}\int_1^\infty M_x(u)u^{-m-1}\dd u
 \leq\mu(K)r^{s+\theta(m-s)}.
\]
Sum these bounds, using $s<t<m$ and $0<\varepsilon<t-s$.
\end{proof}

\begin{remark}
The choices $\rho=R_0^{1+\eta}$ and $d=m+\varepsilon$ permit
$\dim_{qA}\mu=m$ with the arbitrarily small loss $r^{-\theta\varepsilon}$.
\end{remark}

\subsection{Identification of the capacity zeros}

\begin{proposition}[Point-capacity exponents]
\label{prop:measure-capacity-zeros}
Suppose that $\dim_{qA}\mu\leq m$, where $m>0$ is real.
For every $0<\theta\leq1$, the functions
\begin{equation*}
\begin{split}
 \overline F_\theta^m(s)
 &=\limsup_{r\downarrow0}
     \frac{\log C_{r,\theta}^{s,m}(\mu)}{-\log r}-s,\\
 \underline F_\theta^m(s)
 &=\liminf_{r\downarrow0}
     \frac{\log C_{r,\theta}^{s,m}(\mu)}{-\log r}-s
\end{split}
\end{equation*}
are finite, continuous, and strictly decreasing on $[0,m]$.
Their unique zeros are $\udim_\theta\mu$ and
$\ldim_\theta\mu$, respectively.
\end{proposition}

\begin{proof}
Fix $\theta$. For $0\leq a\leq b\leq m$, kernel comparison
\cite[Lemma~5.3]{AMT2026} gives
\[
 \phi_{r,\theta}^{b,m}\leq\phi_{r,\theta}^{a,m}
 \leq r^{-(b-a)(1-\theta)}\phi_{r,\theta}^{b,m}.
\]
Thus either exponent function $F$ satisfies
\begin{equation*}
 -(b-a)\leq F(b)-F(a)\leq-\theta(b-a).
\end{equation*}
On $K-K$, $c_Kr^m\leq\phi_{r,\theta}^{s,m}\leq1$; hence $F$ is finite,
continuous, and strictly decreasing. Pair $F$ with its corresponding
dimension $D_\theta\in[0,m]$.
For $0<s<D_\theta$, choose $\eta>0$ and
$s<t<D_\theta/(1+\eta)$, and put $\varphi=\theta/(1+\eta)$. Then
\begin{equation*}
 D_\varphi\geq\frac{\varphi}{\theta}D_\theta
 =\frac{D_\theta}{1+\eta}>t.
\end{equation*}
Lemmas~\ref{lem:measure-small-mass-witnesses}
and~\ref{lem:quasi-window-potential} give, for $0<\varepsilon<t-s$,
\begin{gather*}
 \inf_{x\in K}\int\phi_{r,\theta}^{s,m}(x-y)\dd\mu(y)
 \lesssim r^{s+\theta(t-s-\varepsilon)},\\
 F(s)\geq\theta(t-s-\varepsilon)>0,
\end{gather*}
at arbitrarily small scales for $\overline F$, and all small scales
for $\underline F$.

For $D_\theta<v<s\leq m$, the defining radius $u\in[r,r^\theta]$ gives
\[
 \int\phi_{r,\theta}^{s,m}(x-y)\dd\mu(y)
 \geq(r/u)^s\mu(B(x,u))\geq cr^su^{v-s}
 \geq cr^{s-\theta(s-v)},
\]
uniformly in $x$, at all small scales for the upper dimension and along
a sequence for the lower. Thus $F(s)\leq-\theta(s-v)<0$.
Finally $F(0)\geq0$ and $F(m)\leq0$ follow from the same kernel bounds.
Continuity identifies the zero also at $D_\theta=0,m$.
\end{proof}

\subsection{Projection bounds and simultaneous parameters}

\begin{proof}[Proof of Theorem~\ref{thm:quasi-measure-projection}]
Fix $0<\theta\leq1$ and $0<s<D_\theta$, where $D_\theta$ is either
dimension. Proposition~\ref{prop:measure-capacity-zeros} gives
$\alpha>0$ and minimizing centres $x_j\in K$ such that
\begin{equation}\label{eq:measure-small-point-potentials}
 \int\phi_{r_j,\theta}^{s,m}(x_j-y)\dd\mu(y)
 \leq r_j^{s+\alpha}.
\end{equation}
Here $r_j\leq2^{-j}$ in the upper case; in the lower case use every
sufficiently small $r_j=2^{-j}$. Compactness supplies the minimizers,
independently of the direction.

Let $\widetilde\phi_{r,\theta}^s$ equal $1$ on $|z|\leq r$,
$(r/|z|)^s$ on $r<|z|\leq r^\theta$, and $0$ elsewhere.
The standard projection-kernel estimate
\cite[Lemma~5.3]{BFF2021} gives, for $0<s<m$ and $0<r<1/2$,
\begin{equation*}
 \int_{G(n,m)}\widetilde\phi_{r,\theta}^s(P_Vz)
                  \dd\gamma_{n,m}(V)
 \leq C\phi_{r,\theta}^{s,m}(z).
\end{equation*}
Put $I_j(V)=\int\widetilde\phi_{r_j,\theta}^s(P_Vx_j-P_Vy)\dd\mu(y)$.
Tonelli, Markov, and \eqref{eq:measure-small-point-potentials} give
\[
 \int I_j\dd\gamma_{n,m}\lesssim r_j^{s+\alpha},\qquad
 \sum_j\gamma_{n,m}\{I_j>r_j^s\}\lesssim\sum_jr_j^\alpha<\infty.
\]
By Borel--Cantelli, $I_j(V)\leq r_j^s$ eventually for almost every $V$.
For such $V$, put $\nu=P_V{}_\#\mu$, $z_j=P_Vx_j\in\operatorname{supp}\nu$.
Then $(r_j/u)^s\nu(B(z_j,u))\leq I_j(V)$ gives
\begin{equation}\label{eq:measure-projected-small-balls}
 \nu(B(z_j,u))\leq u^s,
 \qquad r_j\leq u\leq r_j^\theta.
\end{equation}
For $0\leq a<s$,
\[
 \sup_{r_j\leq u\leq r_j^\theta}
 \frac{\nu(B(z_j,u))}{u^a}
 \leq r_j^{\theta(s-a)}\longrightarrow0,
\]
so $\udim_\theta\nu\geq s$ in the upper case.

For the lower dimension, if $2^{-j-1}<r\leq2^{-j}$, use $z_j$.
For $r\leq u\leq r^\theta\leq2^{-j\theta}$,
\[
 \nu(B(z_j,u))\leq
 \begin{cases}
 u^s,&u\geq2^{-j},\\
 2^{-js}\leq2^su^s,&u<2^{-j},
 \end{cases}
\]
by~\eqref{eq:measure-projected-small-balls} and monotonicity.
Thus for $0\leq a<s$,
\[
 \sup_{r\leq u\leq r^\theta}
 \frac{\nu(B(z_j,u))}{u^a}
 \leq2^sr^{\theta(s-a)}\longrightarrow0.
\]
Thus $\ldim_\theta\nu\geq s$. Let $s\uparrow D_\theta$ through a countable
sequence; every kernel estimate has $s<m$, also when $D_\theta=m$.
The case $D_\theta=0$ is immediate.

For every $V$, the reverse inequalities follow from
\[
 \nu(B(P_Vx,u))\geq\mu(B(x,u)),\qquad
 \udim_\theta\nu\leq\udim_\theta\mu\leq m,\qquad
 \ldim_\theta\nu\leq\ldim_\theta\mu.
\]
Intersect the events for both dimensions and rational $\theta$.
The window comparison \eqref{eq:measure-dimension-continuity} for
$\mu$ and $\nu$ extends equality to every $0<\theta\leq1$.
\end{proof}

\section{Further questions}\label{sec:discussion}\label{sec:ifs-phase-obstacles}

For the spectral families, two questions remain. Can estimates for
successive higher-order eigenvalue corrections extend the range of
resolved bands? Under what regularity assumptions on the family do the
almost-sure lower covering bounds hold at all real scales?

For $E_{\beta,M}$, Theorem~\ref{thm:ifs-first-transition} gives the sharp
separation of the first departure from the continued-profile collision.
An exact transition equation at fixed $M$, the optimal profile after
departure, and the differentiability at the transition remain open.
The finite pressure formulas in Corollary~\ref{cor:ifs-global-pressure-root}
reduce the remaining branch classification to identifying active
constraints and comparing regional optima. Whether the intermediate
dimension curve has infinitely many nonanalytic points is also open.

\appendix
\section{Proof of the uniform harmonic splitting estimate}
\label{app:natural-splitting}

Work on the even harmonic subspace and write
$\alpha_{2m}=a_m$, $\alpha_0=1/2$. The spectral estimates and
self-adjoint perturbation arguments \cite{Kato1995} below are uniform
in the azimuthal index. For even $\ell\geq2$,
\begin{equation*}
 \alpha_\ell\sim(\ell+1)^{-3},\qquad
 g_\ell:=\min\{\alpha_{\ell-2}-\alpha_\ell,
                         \alpha_\ell-\alpha_{\ell+2}\}
              \sim\frac{\alpha_\ell}{\ell+1}.
\end{equation*}

\subsection{Compression of the quadrupolar density}

The normalized associated Legendre recurrence
\cite[\S\S14.10, 14.30]{DLMF} gives
\begin{equation*}
 x_3Y_\ell^k=A_{\ell,k}Y_{\ell+1}^k
                  +B_{\ell,k}Y_{\ell-1}^k,
\end{equation*}
where absent terms are zero and
\[
 A_{\ell,k}^2=\frac{(\ell+1)^2-k^2}{(2\ell+1)(2\ell+3)},
 \qquad
 B_{\ell,k}^2=\frac{\ell^2-k^2}{(2\ell-1)(2\ell+1)}.
\]
Thus $h=(3x_3^2-1)/2$ preserves $k$ and couples degrees
$\ell,\ell\pm2$, with diagonal coefficient
\[
 q_{\ell,k}=\frac32(A_{\ell,k}^2+B_{\ell,k}^2)-\frac12
   =\frac{\ell(\ell+1)-3k^2}{(2\ell-1)(2\ell+3)}.
\]
This proves \eqref{nat:eq:compression}.

\subsection{Uniform isolation in azimuthal sectors}

For $k\in\mathbb Z$, the operators $T,M_h,C_\delta$ preserve
\[
 \mathcal H^{(k)}=\overline{\operatorname{span}}
 \{Y_j^k:j\geq|k|,\ j\ \text{even}\}.
\]
Fix even $\ell\geq\max\{2,|k|\}$. The simple sector eigenvalue
$\alpha_\ell$ has, by min--max and \eqref{nat:eq:relative}, an ordered
perturbation $\mu\in[(1-\delta)\alpha_\ell,(1+\delta)\alpha_\ell]$.
For sufficiently small absolute $c>0$,
\[
 \delta(\ell+1)\leq c
 \quad\Longrightarrow\quad
 \begin{cases}
 (1-\delta)\alpha_{\ell-2}-(1+\delta)\alpha_\ell\geq c'g_\ell,\\
 (1-\delta)\alpha_\ell-(1+\delta)\alpha_{\ell+2}\geq c'g_\ell.
 \end{cases}
\]
Omit an absent neighbor. These inequalities isolate one eigenvalue
in each of the $2\ell+1$ sectors $|k|\leq\ell$; all other sector spectra
are bounded by $(1+\delta)\alpha_{\ell+2}$. Thus the band has multiplicity
$2\ell+1$ and gap $c'g_\ell$. Complex conjugation identifies the
$k$ and $-k$ eigenvalues.

\subsection{A local Schur complement}

In a sector containing $Y_\ell^k$, let $P$ project onto this vector
and set $Q=I-P$. On $\operatorname{ran}Q$,
\[
 (1-\delta)QTQ\leq QC_\delta Q\leq(1+\delta)QTQ.
\]
Removing $\alpha_\ell$ from the sector spectrum and applying min--max gives
\begin{equation}\label{nat:eq:compressed-resolvent}
 \operatorname{dist}\bigl(\mu,\operatorname{spec}(QC_\delta Q)\bigr)
 \geq c_1g_\ell,
 \qquad
 \|(QC_\delta Q-\mu)^{-1}\|\leq(c_1g_\ell)^{-1},
\end{equation}
uniformly in $\ell,k,\delta$ under $\delta(\ell+1)\leq c$,
including a lowest-degree sector eigenvalue.
Finite coupling range, comparable neighboring eigenvalues, and
$\|h\|_\infty=1$ give
\begin{equation*}
 \begin{split}
 \|QVP\|^2
 &=\alpha_\ell\sum_{j\in\{\ell-2,\ell+2\}}
       \alpha_j|\langle Y_j^k,hY_\ell^k\rangle|^2\\
 &\leq C_1\alpha_\ell^2.
 \end{split}
\end{equation*}
Omit absent terms. Equation~\eqref{nat:eq:compressed-resolvent} makes the
$P$ component nonzero, so the Schur complement is
\[
 \mu=\alpha_\ell+\delta\alpha_\ell q_{\ell,k}
 -\delta^2\langle QVY_\ell^k,
                (QC_\delta Q-\mu)^{-1}QVY_\ell^k\rangle.
\]
Consequently,
\[
 |\mu-\alpha_\ell(1+\delta q_{\ell,k})|
 \leq\delta^2\|QVP\|^2\|(QC_\delta Q-\mu)^{-1}\|
 \lesssim\delta^2\alpha_\ell(\ell+1).
\]
Set $\ell=2m$.

\section{ReLU eigenvalues, gaps, and harmonic derivation}\label{sec:kernels}
\label{sec:spectrum}\label{app:spectrum}

Fix $d\geq2$ and uniform probability measure $\sigma$ on $S^{d-1}$.
The two-layer network
\[
 f_W(x)=\sqrt{2/W}\sum_{i=1}^Wv_i(w_i\cdot x)_+,
 \qquad v_i\sim N(0,1),\quad w_i\sim N(0,I_d),
\]
has mutually independent Gaussian parameters. Weighting the output and
hidden parameter-gradient contributions by $1$ and $\eps\in[0,1]$ gives
the limiting tangent kernel \cite{Jacot2018,BM2019}
\begin{equation}\label{eq:weighted-kernel}
 K_\eps(x,y)=\kappa_{\rm RF}(x\cdot y)
       +\eps(x\cdot y)\kappa_0(x\cdot y),
\end{equation}
where
\[
 \kappa_0(u)=\frac{\pi-\arccos u}{\pi},\qquad
 \kappa_{\rm RF}(u)=
 \frac{\sqrt{1-u^2}+u(\pi-\arccos u)}{\pi}.
\]
The cases $\eps=0$ and $\eps=1$ are the random-feature and ordinary
tangent kernels, respectively; write $K_{\rm RF}=K_0$.
Let $T_\eps$ be the integral operator of $K_\eps$ on $L^2(\sigma)$.

The degree-$\ell$ harmonic dimension is, for $d\geq3$ and $\ell\geq1$,
\[
 N(d,\ell)=\frac{(2\ell+d-2)(\ell+d-3)!}{\ell!(d-2)!},
\]
while $N(2,\ell)=2$ for $\ell\geq1$ and $N(d,0)=1$. In particular,
$L_m:=N(d,2m)\sim m^{d-2}$ for $m\geq1$, giving $b=d-1$.

\begin{proposition}[Exact values and uniform gaps]\label{prop:eigenvalues}
\label{prop:uniform-spectrum}
Put
\[
 C_d=\frac{\Gamma(d/2)^2}{8\pi^2},\qquad
 R_m=\frac{\Gamma(m-1/2)}{\Gamma(m+(d+1)/2)},\qquad
 h_d(m)=4m^2+2(d-2)m+1.
\]
The eigenvalue on the degree-$2m$ harmonic space is
\begin{equation*}
 a_m(\eps)=C_dR_m^2[d+\eps h_d(m)],\qquad m\geq1.
\end{equation*}
All odd harmonic spaces of degree at least three lie in the nullspace.
The degree-zero and degree-one eigenvalues are
\[
 \lambda_0(\eps)=2(d+\eps)
 \left(\frac{\Gamma(d/2)}{2\sqrt\pi\Gamma((d+1)/2)}\right)^2,
 \qquad \lambda_1(\eps)=\frac{1+\eps}{2d}.
\]
The even eigenvalues are strictly decreasing, and their gaps satisfy
\begin{equation}\label{eq:exact-gap}
 a_m(\eps)-a_{m+1}(\eps)
 =C_dR_m^2\frac{d(4m+d)}{(2m+d+1)^2}
       [(d+2)+\eps(4m^2+2dm+1)].
\end{equation}
Consequently, with constants depending only on $d$, uniformly in
$m\geq1$ and $\eps\in[0,1]$,
\begin{align}
 a_m(\eps)&\sim m^{-d-2}+\eps m^{-d},\label{eq:values}\\
 a_m(\eps)-a_{m+1}(\eps)&\sim
       m^{-d-3}+\eps m^{-d-1}.\label{eq:gaps}
\end{align}
These identities imply \eqref{nat:eq:even-spectrum}--\eqref{nat:eq:even-estimates}
when $d=3$ and $\eps=1$. They also verify
\eqref{eq:res-values}--\eqref{eq:res-gaps} with $q=d+2$ and $p=d$.
\end{proposition}

\subsection{Derivation of the ReLU spectral formula}\label{app:relu-derivation}

Let $(Y_{\ell,j})_{1\leq j\leq N(d,\ell)}$, $\ell\geq0$, be a real
orthonormal harmonic basis in $L^2(\sigma)$, with Laplacian eigenvalues
$\ell(\ell+d-2)$. Put
\[
 a=\frac{d-1}{2},\qquad
 c_d=\frac{\Gamma(d/2)}{\sqrt\pi\,\Gamma(a)},\qquad
 w(t)=(1-t^2)^{a-1}.
\]
Let $P_\ell$ be the zonal polynomial normalized by $P_\ell(1)=1$.
For $d=2$, $P_\ell(t)=\cos(\ell\arccos t)$; in general, these are
normalized Jacobi polynomials with both parameters $a-1$.
The Funk--Hecke identity gives
\begin{equation}\label{eq:funk}
 \int_{S^{d-1}}(u\cdot x)_+Y_{\ell,j}(u)\dd\sigma(u)
 =b_\ell Y_{\ell,j}(x),
 \qquad
 b_\ell=c_d\int_0^1 tP_\ell(t)w(t)\dd t.
\end{equation}
See \cite[Appendix C]{BM2019} for Funk--Hecke and
\cite[Section 18.5(ii)]{DLMF} for Rodrigues' formula.

\begin{lemma}\label{lem:relu-coeff}
The coefficients in \eqref{eq:funk} satisfy
\begin{align*}
 b_0&=\frac{\Gamma(d/2)}{2\sqrt\pi\,\Gamma((d+1)/2)},
 &b_1&=\frac1{2d},\\
 b_{2m}&=(-1)^{m-1}\frac{\Gamma(d/2)}{4\pi}
 \frac{\Gamma(m-1/2)}{\Gamma(m+(d+1)/2)},
 &&m\geq1,
\end{align*}
and $b_\ell=0$ for odd $\ell\geq3$.
\end{lemma}

\begin{proof}
Parity, orthogonality, and $t_+=(t+|t|)/2$ give $b_\ell=0$ for odd $\ell\geq3$.
Since $P_0=1$ and $P_1(t)=t$,
\[
 b_0=c_d\int_0^1t(1-t^2)^{a-1}\dd t=\frac{c_d}{2a},
 \qquad b_1=\frac12\int_{S^{d-1}}u_1^2\dd\sigma(u)=\frac1{2d}.
\]

For even degree, the normalized Rodrigues formula is
\[
 w(t)P_\ell(t)=\frac{(-1)^\ell}{2^\ell(a)_\ell}
 \frac{\mathrm d^\ell}{\mathrm dt^\ell}
 (1-t^2)^{\ell+a-1},
 \qquad (a)_\ell=\frac{\Gamma(a+\ell)}{\Gamma(a)}.
\]
For $\ell=2m$ and $F(t)=(1-t^2)^{2m+a-1}$, the boundary derivatives
$F^{(2m-1)}(1)=F^{(2m-2)}(1)=0$ since $a>0$.
Integration by parts and Taylor expansion give
\begin{align*}
 \int_0^1tF^{(2m)}(t)\dd t
 &=F^{(2m-2)}(0)\\
 &=(-1)^{m-1}
 \frac{\Gamma(2m-1)\Gamma(2m+a)}{\Gamma(m)\Gamma(m+a+1)}.
\end{align*}
Substitution in \eqref{eq:funk} gives
\[
 b_{2m}=(-1)^{m-1}
 \frac{\Gamma(d/2)\Gamma(2m-1)}
 {\sqrt\pi\,2^{2m}\Gamma(m)\Gamma(m+(d+1)/2)}.
\]
Use $\Gamma(2m-1)=2^{2m-2}\Gamma(m-1/2)\Gamma(m)/\sqrt\pi$.
\end{proof}

\subsection{The weighted kernel eigenvalues and gaps}

For every $\ell\geq0$, we prove
\begin{equation}\label{eq:all-spectrum}
 \lambda_\ell(\eps)
 =2\left[d+\eps\{1+\ell(\ell+d-2)\}\right]b_\ell^2.
\end{equation}
\begin{proof}[Proof of Proposition~\ref{prop:eigenvalues}]
Put $\phi_x(u)=(u\cdot x)_+$. For $g=Ru\sim N(0,I_d)$,
$u\sim\sigma$ is independent of $R$ and $\E R^2=d$, so
\[
 K_{\rm RF}(x,y)=2d\int\phi_x(u)\phi_y(u)\dd\sigma(u).
\]
Almost everywhere,
\[
 \nabla_S\phi_x(u)=x\boldsymbol1_{\{u\cdot x>0\}}-u\phi_x(u).
\]
Hence
\[
 \nabla_S\phi_x\cdot\nabla_S\phi_y
 =(x\cdot y)\boldsymbol1_{\{u\cdot x>0,\ u\cdot y>0\}}
       -\phi_x\phi_y.
\]
Thus
\begin{equation*}
 K_\eps(x,y)=2(d+\eps)\int\phi_x\phi_y\dd\sigma
       +2\eps\int\nabla_S\phi_x\cdot\nabla_S\phi_y\dd\sigma.
\end{equation*}
Since $\phi_x$ is Lipschitz and belongs to $H^1(S^{d-1})$, Funk--Hecke and
\[
 \int\nabla_SY_{\ell,j}\cdot\nabla_SY_{k,i}\dd\sigma
 =\ell(\ell+d-2)\boldsymbol1_{\{\ell=k,\ j=i\}}
\]
give \eqref{eq:all-spectrum}. Apply Lemma~\ref{lem:relu-coeff}.

The gamma recurrence gives
\[
 \frac{R_{m+1}^2}{R_m^2}
 =\frac{(2m-1)^2}{(2m+d+1)^2}.
\]
This gives the positive gap \eqref{eq:exact-gap}.
Since $R_m^2\sim m^{-d-2}$ \cite[equation~5.11.12]{DLMF}
and $h_d(m)\sim m^2$, the exact formulas give
\eqref{eq:values}--\eqref{eq:gaps}, uniformly in $\eps\in[0,1]$ with
constants depending only on $d$; see also \cite{BM2019,BB2021}.
\end{proof}

\section{Pressure duality and auxiliary branch calculations}
\label{app:ifs-phases}

\subsection{Localization and departure of the first analytic branch}

\begin{proof}[Proof of Proposition~\ref{prop:ifs-entropy-localization}]
Normalize to $[0,1]$ and suppose $A_f(1)\ge d_0L_f(1)$,
as holds for every maximizer since $d_\theta>d_0$.
Set $z=f-1/2$ and
\[
 u(t)=\beta-\mathbf1_{[0,\beta]}(t),\qquad
 v(t)=2\beta\mathbf1_{[0,\beta]}(t)
       -\mathbf1_{[0,\beta^2]}(t)-\beta^2.
\]
Integrals and inner products in this proof are taken over $[0,1]$.
Since $h(1/2+x)\le c-2x^2$ by $h''(y)=-1/[y(1-y)]\le-4$,
\[
 2\|z\|_2^2+d_0c\langle u,z\rangle\le0.
\]
Writing $q_0=c^2/(4\lambda)$ and completing the square gives
\begin{equation}\label{eq:ifs-entropy-ball}
 \|z+q_0u\|_2^2\le q_0^2B.
\end{equation}
Moreover,
\[
 L_f(1)=\lambda+c\langle u,z\rangle,
 \qquad L_f(\beta)-\beta L_f(1)=c\langle v,z\rangle.
\]
Direct integration yields
\[
 \|u\|_2^2=B,\qquad
 \langle v,u\rangle=-\beta B,\qquad
 \|v\|_2^2=\beta^2(1-\beta^2).
\]
For $\xi=\vartheta-\beta$, it follows that
\[
 \langle v-\xi u,u\rangle=-\vartheta B,
 \qquad
 \|v-\xi u\|_2^2=B(\vartheta^2+\beta).
\]
Maximizing a linear functional over the ball
\eqref{eq:ifs-entropy-ball} therefore gives
\begin{align*}
 L_f(\beta)-\vartheta L_f(1)
 &=-\lambda(\vartheta-\beta)
        +c\langle v-(\vartheta-\beta)u,z\rangle\\
 &\le-\lambda(\vartheta-\beta)
       +cq_0B\{\vartheta+\sqrt{\vartheta^2+\beta}\}.
\end{align*}
The right side vanishes at $\theta_{\mathrm{ent}}$, the solution of
\begin{equation*}
 \vartheta-\beta
    =\epsilon\{\vartheta+\sqrt{\vartheta^2+\beta}\}.
\end{equation*}
Since $\epsilon/(1-\beta)=c^3\beta/(4\lambda^2)<1/4$, the left
side minus the right has derivative at least $1-2\epsilon>0$,
is negative at $\beta$ and positive at $1$.

In fact,
\begin{equation}\label{eq:ifs-strict-entropy-localization}
 \frac{L_f(\beta)}{L_f(1)}<\theta_{\mathrm{ent}}.
\end{equation}
Since $h(1/2+x)<c-2x^2$ for $x\ne0$, $z\ne0$ in $L^2$ makes
\eqref{eq:ifs-entropy-ball} strict. If $z=0$, the ratio is $\beta$.
Thus every maximizer has $\tau>\beta$ for
$\theta\ge\theta_{\mathrm{ent}}$.

The first-branch certificate then gives
\begin{equation*}
 d_\theta\le\widehat d_\theta,
 \qquad \theta\ge\theta_{\mathrm{ent}}.
\end{equation*}
Continue the four-frequency candidate from $\theta=1$, with continuous
frequencies and positive denominator. Since $\widehat d_\theta\ge
d_\theta>d_0$, a first collision $\tau_*=\beta$ at
$\theta\ge\theta_{\mathrm{ent}}$ would give
\[
 A_f(1)=\widehat d_\theta L_f(1)>d_0L_f(1),\qquad
 L_f(\beta)/L_f(1)=\theta\ge\theta_{\mathrm{ent}},
\]
contradicting \eqref{eq:ifs-strict-entropy-localization}.
Hence $\tau_*>\beta$ and the candidate attains the upper bound.
The first-branch equality and derivative arguments give uniqueness
and regularity. Combining the two validity intervals gives
$\theta_{\mathrm{loc}}$.
\end{proof}

\begin{proof}[Proof of Proposition~\ref{prop:ifs-genuine-transition}]
By concavity, $s_*=k'(1-)$ exists and
\[
 \frac{D-k(w)}{1-w}\longrightarrow s_*.
\]
In the formula
\[
 k'(w)=\frac{h(v_w)}{\lambda+c\beta(v_w-z_w)},
 \quad
 z_w=\sigma(c((1-\beta)D-k(w))),\quad
 v_w=\sigma\!\left(-\frac{c\beta(D-k(w))}{1-w}\right),
\]
let $w\uparrow1$. Then $z_w\to q$, $v_w\to\sigma(-c\beta s_*)$;
the logistic entropy identity yields \eqref{eq:ifs-limit-pressure}.
The pressure equals $c$ at zero and has derivative
$c\beta\{q-\sigma(-c\beta s)\}-\lambda\le c\beta-\lambda<0$.
At $s=d_0$, put $x=c\beta d_0$; since $D>d_0$, $q<\sigma(-x)$, so
\[
 \log((1+e^{-x})/2)+xq
 =-\int_0^x\sigma(-t)\,dt+xq<0.
\]
Thus $0<s_*<d_0$. Strict concavity gives
$k(w)+(1-w)k'(w)>D$ for $w<1$; the stationary equation therefore
forces $w_\theta\uparrow1$ as $\theta\downarrow0$, proving
\eqref{eq:ifs-continued-limit}.

The continuation increases from $s_*<d_0$ to $D>d_0$ and hence
crosses $d_0$ uniquely, at $k'(w_{\mathrm{cross}})=d_0$ and
\[
 \theta_{\mathrm{cross}}
 =1-\frac{D-d_0}{k(w_{\mathrm{cross}})-w_{\mathrm{cross}}d_0}.
\]
Continuity and $d_{\theta_{\mathrm{cross}}}>d_0$ separate the curves
near $\theta_{\mathrm{cross}}$, while they agree for
$\theta\ge\theta_{\mathrm{loc}}$. This gives
\eqref{eq:ifs-transition-bracket} and equality on
$[\theta_{\mathrm{tr}},1)$.

If $d_\theta$ were analytic near $\theta_{\mathrm{tr}}$, the identity
theorem would extend $d_\theta=\widehat d_\theta$ from its right
neighborhood to its left, contradicting the definition of first departure.
\end{proof}

\subsection{Separation, attained multipliers, and boundary faces}

\begin{proof}[Proof of Proposition~\ref{prop:ifs-pressure-duality}]
For fixed $(w,\eta)$, the Lagrangian is
\[
 \sum_jw_jg_j^s(x)+\eta \mathcal R(x)
 =-\lambda\sum_j\rho_jt_j
            +\sum_i l_i\{W_i h(z_i/l_i)+K_i z_i/l_i\}.
\]
Zero-length terms vanish. Maximizing cellwise, then over $\tau\in I$,
gives the largest endpoint pressure, proving the dual upper bound.

For a nonempty physical slice, the set
\[
 \mathcal C=\{(\mathcal R(x),v):x\in\mathcal K_I,\quad
                           v_j\le g_j^s(x)\ \text{for every }j\}
\]
is convex by concavity and closed by compactness and continuity.
Strictly separate $(0,q,\ldots,q)$ from $\mathcal C$, where
$q>\mathcal V_I(s)$:
\[
 \sup_{x\in\mathcal K_I}
       \left(\eta \mathcal R(x)+\sum_jw_jg_j^s(x)\right)
                         <q\sum_jw_j.
\]
Downward rays force $w_j\ge0$, and feasibility excludes $w=0$.
Normalize $\sum_jw_j=1$ and let $q\downarrow\mathcal V_I(s)$.

For an empty slice, the compact interval $\mathcal R(\mathcal K_I)$
avoids zero; taking $\eta\to\pm\infty$ with opposite sign sends the
Lagrangian uniformly to $-\infty$.

In the two-sided case, $|g_j^s|\le C_s:=c+sb$ and the extrema
of $\mathcal R$ give
\[
 \max_{x\in\mathcal K_I}
       \left(\sum_jw_jg_j^s(x)+\eta \mathcal R(x)\right)
                         \ge-C_s+\delta|\eta|.
\]
The value at $\eta=0$ is at most $C_s$, so minimize the continuous
objective on $\Delta_J\times[-2C_s/\delta,2C_s/\delta]$.
For $\mathcal R\equiv0$, only $\Delta_J$ is needed.
\end{proof}

\begin{proof}[Proof of Proposition~\ref{prop:ifs-boundary-pressure}]
For fixed $\tau$,
\[
 \min_{0\le z_i\le l_i(\tau)}\sigma\mathcal R(\tau,z)
 =h_0(\tau)+\sum_i\min\{0,e_i\}l_i(\tau)=m(\tau)\ge0.
\]
Equality requires $m(\tau)=0$ and the stated endpoint choices for
$e_i\ne0$. Since $m\ge0$ is affine, its nonempty zero set is $I$ or
an endpoint. Forced frequencies have zero entropy; maximizing over the
free cells gives the affine pressure $\mathcal P^0_{I,s}$, maximized
on $\partial I_0$. Separation on this compact convex face gives the
dual formula without $\eta$; continuity on $\Delta_J$ gives attainment.
\end{proof}

\begin{proof}[Proof of Corollary~\ref{cor:ifs-global-pressure-root}]
Every constrained window time obeys
\[
 \theta a^2/b\le a\tau\le L_f(t_j)\le b.
\]
Since $g_j^{s_2}-g_j^{s_1}=-(s_2-s_1)L_f(t_j)$, taking the
nodal minimum and the regional maximum gives
\eqref{eq:ifs-pressure-monotonicity}. The profile $f=1/2$, $\tau=\theta$
gives $\mathcal V_\theta(0)\ge c\theta>0$, while $A_f(t)\le ct$
and $L_f(t)\ge at$ give $\mathcal V_\theta(s)<0$ for $s>c/a$.
The unique zero is $d_\theta$ by \eqref{eq:tree-finite-source}.
For $s\ge d_\theta$, all nonempty regional values are nonpositive;
the attained dual formulas supply their endpoint certificates.
\end{proof}

\medskip
\noindent\textbf{Disclosure on AI assistance:}
The authors acknowledge the use of OpenAI tools (GPT 6 Astra) in the
preparation of this manuscript, including language editing, LaTeX
formatting, and the organization of the presentation. The authors have
reviewed the complete manuscript and take full responsibility for all
mathematical statements and proofs.

\end{document}